\documentclass{jsg}
\usepackage{pb-diagram,amssymb,epic,eepic,verbatim,graphicx}
\usepackage{graphics,epsfig,psfrag}
\usepackage{color}

\usepackage{hyperref}
\hypersetup{colorlinks} \usepackage{color}
\definecolor{darkred}{rgb}{0.5,0,0}
\definecolor{darkgreen}{rgb}{0,0.5,0}
\definecolor{darkblue}{rgb}{0,0,0.5} 
\hypersetup{ colorlinks, linkcolor=darkblue, filecolor=darkgreen, urlcolor=darkred,
citecolor=darkblue }

\renewcommand\theenumi{\alph{enumi}}
\renewcommand\theenumii{\roman{enumii}}
\newtheorem{theorem}{Theorem}[section]
\newtheorem{corollary}[theorem]{Corollary}

\newtheorem{lemma}[theorem]{Lemma}
\newtheorem{lem}[theorem]{}
\theoremstyle{definition}
\newtheorem{definition}[theorem]{Definition}
\theoremstyle{remark}
\newtheorem{remark}[theorem]{Remark}
\newtheorem{example}[theorem]{Example}
\newcommand{\blem}{\begin{lem} \rm}
\newcommand{\elem}{\end{lem}}
\newcommand\B{\mathcal{B}}
\newcommand\E{\mathcal{E}}
\newcommand\M{\mathcal{M}}

\renewcommand\M{\mathcal{M}}
\renewcommand\S{\mathcal{S}}

\renewcommand{\L}{\mathcal{L}}
\renewcommand{\O}{\mathcal{O}}

\newcommand{\J}{\mathcal{J}}

\newcommand{\U}{\mathcal{U}}

\newcommand{\F}{\mathcal{F}}
\newcommand{\N}{\mathbb{N}}
\newcommand{\R}{\mathbb{R}}
\renewcommand{\H}{\mathbb{H}}

\newcommand{\C}{\mathbb{C}}
\newcommand{\cC}{\mathcal{C}}

\newcommand{\Z}{\mathbb{Z}}

\newcommand\lie[1]{\mathfrak{#1}}

\renewcommand{\sl}{\lie{sl}}
\newcommand{\on}{\operatorname}

\newcommand{\Don}{\on{Don}}

\newcommand{\Ham}{\on{Ham}}

\newcommand{\Lag}{\on{Lag}}

\newcommand{\height}{\on{ht}}

\newcommand{\Ind}{ \on{Ind}}

\newcommand{\im}{ \on{im}}

\newcommand\dirac{/\kern-1.2ex\partial} 
\newcommand\qu{/\kern-.7ex/} 
\newcommand\lqu{\backslash \kern-.7ex \backslash} 

\newcommand\dr{r_+ \kern-.7ex - \kern-.7ex r_-}

\def\pd{\partial}

\def\tint{{\textstyle\int}}

\renewcommand{\d}{{\mbox{d}}}
\newcommand{\ol}{\overline}

\newcommand\eps{\epsilon}
\newcommand\Om{\Omega}

\newcommand{\lan}{\langle}
\newcommand{\ran}{\rangle}

\newcommand{\ti}{\tilde}

\newcommand\pt{\on{pt}}
\newcommand\cE{\mathcal{E}}
\newcommand\cL{\mathcal{L}}

\newcommand\cF{\mathcal{F}}

\renewcommand\vol{\on{vol}}

\newcommand\Tr{\on{Tr}}

\newcommand\Map{\on{Map}}
\newcommand\rank{\on{rank}}

\newcommand\Vect{\on{Vect}}
\newcommand\ul{\underline}

\newcommand\reg{{\on{reg}}}
\newcommand\bra[1]{ \lan {#1} \ran}
\newcommand\bdefn{\begin{definition}}
\newcommand\edefn{\end{definition}}
\newcommand\bea{\begin{eqnarray*}}
\newcommand\eea{\end{eqnarray*}}
\newcommand\bcv{\left[ \begin{array}{r} }
\newcommand\ecv{\end{array} \right] }

\newcommand\bma{\left[ \begin{array} }
\newcommand\ema{\end{array} \right]}
\newcommand\ben{\begin{enumerate}}
\newcommand\een{\end{enumerate}}
\newcommand\beq{\begin{equation}}
\newcommand\eeq{\end{equation}}
\newcommand\bex{\begin{example}}
\newcommand\bsj{\left\{ \begin{array}{rrr} }
\newcommand\esj{\end{array} \right\}}

\newcommand\Id{\on{Id}}
\newcommand\cI{\mathcal{I}}

\newcommand\eex{\end{example}}

\newcommand\sx{*\kern-.5ex_X}

\def\mathunderaccent#1{\let\theaccent#1\mathpalette\putaccentunder}
\def\putaccentunder#1#2{\oalign{$#1#2$\crcr\hidewidth \vbox
to.2ex{\hbox{$#1\theaccent{}$}\vss}\hidewidth}}

\newcommand{\odd}{{\on{odd}}}
\newcommand{\even}{{\on{even}}}

\newcommand{\labell}\label

\begin{document}

\title{Pseudoholomorphic Quilts}

\author{Katrin Wehrheim and Chris T. Woodward}

\address{Department of Mathematics,
Massachusetts Institute of Technology,
Cambridge, MA 02139.
{\em E-mail address: katrin@math.berkeley.edu}}

\address{Department of Mathematics,
Rutgers University,
Piscataway, NJ 08854.
{\em E-mail address: ctw@math.rutgers.edu}}

\thanks{K.W. was partially supported by NSF CAREER grant DMS0844188.
  C.W. was partially supported by NSF grant DMS0904358}

\begin{abstract}
We define relative Floer theoretic invariants arising from ``quilted pseudoholomorphic
surfaces'': Collections of pseudoholomorphic maps to various target
spaces with ``seam conditions'' in Lagrangian correspondences.
\end{abstract}

\maketitle

\setcounter{tocdepth}{1}
\tableofcontents

\section{Introduction}

Lagrangian Floer cohomology associates to a pair of Lagrangian
manifolds a chain complex whose differential counts pseudoholomorphic
strips with boundary values in the given Lagrangians.  These form a
receptacle for {\em relative invariants} defined from surfaces with
strip-like ends, see e.g.\ \cite{se:bo}.  The simplest instance of
this invariant proves the independence of Floer cohomology from the
choices of almost complex structure and perturbation data.  More
complicated cases of these invariants construct the product on Floer
cohomology, and the structure maps of the Fukaya category.

This paper is one of a sequence in which we generalize these
invariants to include Lagrangian correspondences.  The first papers in
the sequence were \cite{quiltfloer}, \cite{isom}; however, we have
tried to make the paper as self-contained as possible.  Recall that if
$(M_0,\omega_0)$ and $(M_1,\omega_1)$ are symplectic manifolds, then a
{\em Lagrangian correspondence} from $M_0$ to $M_1$ is a Lagrangian
submanifold of $M_0^- \times M_1$, where $M_0^-:= (M_0,-\omega_0)$.
Whereas Lagrangian manifolds form elliptic {\em boundary conditions}
for pseudoholomorphic curves, Lagrangian correspondences form elliptic
{\em seam conditions}. For a pair of curves with boundary in $M_0$ and
$M_1$, a seam condition $L_{01}$ roughly speaking requires
corresponding boundary values to pair to a point on $L_{01}\subset
M_0^-\times M_1$.  (Differently put, a neighbourhood of the seam can
be ``folded up'' to a pseudoholomorphic curve in $M_0^- \times M_1$
with boundary values on $L_{01}$.)  Using such seam conditions between
pseudoholomorphic strips, we defined in \cite{quiltfloer} a {\em
  quilted Floer cohomology} $HF(L_{01},L_{12},\ldots, L_{(k-1)k})$ for
a cyclic sequence of Lagrangian correspondences
$L_{(\ell-1)\ell}\subset M_{\ell-1}^-\times M_\ell$ between symplectic
manifolds $M_0, M_1,\ldots, M_k=M_0$.  In this paper we construct new
Floer type invariants arising from quilted pseudoholomorphic surfaces.
These quilts consist of pseudoholomorphic surfaces (with boundary and
strip-like ends in various target spaces) which satisfy seam
conditions (mapping certain pairs of boundary components to Lagrangian
correspondences) and boundary conditions (mapping other boundary
components to simple Lagrangian submanifolds).  Similar moduli spaces
appeared in the work of Perutz \cite{per:lag} and have been considered
by Khovanov and Rozansky \cite{kr:foam} under the name of {\em
  pseudoholomorphic foams}.  The present paper provides a general
language, invariance and gluing theorems for pseudoholomorphic quilts and the invariants arising from them.
To simplify the algebraic statements, we concentrate on the case of $\Z_2$ coefficients. More general coefficient rings require coherent orientations of the moduli spaces of pseudoholomorphic quilts, which are constructed in \cite{orient}. We indicate the necessary choices of orderings, orientations, and relative spin structures in a series of remarks.
In an independent second series of remarks we establish the interaction between the quilt invariants and Maslov gradings on the Floer cohomologies.

Moreover, we establish an invariance of relative quilt invariants
under strip shrinking and geometric composition of Lagrangian
correspondences.  The geometric composition of two Lagrangian
correspondences $L_{01}\subset M_0^-\times M_1$, $L_{12}\subset
M_1^-\times M_2$ is
$$
L_{01} \circ L_{12} := \bigl\{ (x_0,x_2)\in M_0\times M_2 \,\big|\, \exists x_1 :
(x_0,x_1)\in L_{01}, (x_1,x_2)\in L_{12} \bigr\} .
$$
In general, this will be a singular subset of $M_0^-\times M_2$, with
isotropic tangent spaces where smooth.
However, if we assume transversality of the intersection
$ L_{01} \times_{M_1} L_{12} := \bigl( L_{01} \times L_{12}\bigr) \cap
\bigl(M_0^- \times \Delta_{M_1}\times M_2 \bigr)$, then the
restriction of the projection $ \pi_{02}: M_0^- \times M_1 \times
M_1^- \times M_2 \to M_0^- \times M_2$ to $L_{01} \times_{M_1} L_{12}$
is automatically an immersion.  If in addition $\pi_{02}$ is
injective, then $L_{01} \circ L_{12}$ is a smooth Lagrangian
correspondence and we will call it an {\em embedded} geometric
composition.  If the composition $L_{(\ell-1)\ell}\circ
L_{\ell(\ell+1)}$ is embedded for some $\ell$, then under suitable monotonicity
assumptions there is a canonical isomorphism
\begin{equation}\label{eq:iso}
HF(\ldots, L_{(\ell-1)\ell}, L_{\ell(\ell+1)}, \ldots)
\cong
HF(\ldots, L_{(\ell-1)\ell}\circ L_{\ell(\ell+1)}, \ldots) .
\end{equation}
For the precise monotonicity and admissibility conditions see
\cite{quiltfloer} or Section~\ref{sec:invariance}. The proof in
\cite{isom,quiltfloer} proceeds by shrinking the strip in $M_\ell$ to
width zero and replacing the two seams labeled $ L_{(\ell-1)\ell}$ and
$L_{\ell(\ell+1)}$ with a single seam labeled $L_{(\ell-1)\ell}\circ
L_{\ell(\ell+1)}$.  In complete analogy, Theorem~\ref{intertwine}
shows that the same can be done in a more general quilted surface: The
two relative invariants arising from the quilted surfaces with or
without extra strip are intertwined by the above isomorphism
\eqref{eq:iso} of Floer cohomologies.

A first application of pseudoholomorphic quilts is the proof of
independence of quilted Floer cohomology from the choices of strip
widths in \cite{quiltfloer}.  As a second example we construct in
Section~\ref{app} a morphism on quantum homologies $\Phi_{L_{01}}:
HF(\Delta_{M_0})\to HF(\Delta_{M_1})$ associated to a Lagrangian
correspondence $L_{01}\subset M_0^-\times M_1$.
This is in general not a ring morphism, simply for degree reasons, and even on the classical level.
However, using the invariance and gluing theorems for quilts, we show
that it factors $\Phi_{L_{01}}=\Theta_{L_{01}}\circ\Psi_{L_{01}}$ into
a ring morphism $\Psi: HF(\Delta_{M_0})\to HF(L_{01},L_{01})$ and a
morphism $\Theta: HF(L_{01},L_{01}) \to HF(\Delta_{M_1})$ satisfying
$$ \Theta(x \circ y) - \Theta(x) \circ \Theta(y) = x \circ T_{L_{01}}
\circ y ,
$$
where the right hand side is a composition on quilted Floer cohomology
with an element $T_{L_{01}} \in HF(L_{01}^t,L_{01},L_{01}^t,L_{01})$
that only depends on $L_{01}$. 

In the sequel \cite{ww:cat} we construct a symplectic $2$-category with a
categorification functor via pseudoholomorphic quilts. In particular,
any Lagrangian correspondence $L_{01} \subset M_0^- \times M_1$ gives
rise to a functor $\Phi(L_{01}) : \Don^\#(M_0) \to \Don^\#(M_1)$
between (somewhat extended) Donaldson-Fukaya categories.  Given
another Lagrangian correspondence $L_{12}\subset M_1^-\times M_2$, the
algebraic composition $\Phi(L_{01})\circ\Phi(L_{12}): \Don^\#(M_0) \to
\Don^\#(M_2)$ is always defined, and if the geometric composition
$L_{01} \circ L_{12}$ is embedded then $\Phi(L_{01}) \circ
\Phi(L_{12}) \cong \Phi(L_{01} \circ L_{12})$.  In other words,
embedded geometric composition is isomorphic to the algebraic
composition in the symplectic category, and ``categorification
commutes with composition''.

\medskip

{\em We thank Paul Seidel and Ivan Smith for encouragement and helpful discussions,
and the referee for most valuable feedback.}

\subsection{Notation and monotonicity assumptions}

One notational warning: When dealing with functors we will use
functorial notation for compositions, that is $\Phi_0\circ\Phi_1$ maps
an object $x$ to $\Phi_1(\Phi_0(x))$.  When dealing with simple maps
like symplectomorphisms, we will however stick to the traditional notation
$(\phi_1\circ\phi_0)(x)=\phi_1(\phi_0(x))$.

For Lagrangian correspondences and (quilted) Floer cohomology we will use the notation developed in \cite{quiltfloer}. Moreover, we will frequently refer to assumptions on monotonicity, Maslov indices, and grading of symplectic manifolds $M$ and Lagrangian submanifolds $L\subset M$.
We briefly summarize these here from \cite{quiltfloer}.

\begin{itemize}
\setlength{\itemsep}{1mm}
\item[\bf(M1):] $(M,\omega)$ is {\em monotone}, that is  $[\omega] = \tau c_1(TM) $
for some $\tau\geq 0$.
\item[\bf(M2):] If $\tau>0$ then $M$ is compact. If $\tau=0$ then
$M$ is (necessarily) noncompact but satisfies ``bounded geometry''
assumptions as in \cite[Chapter~7]{se:bo}.\footnote{
More precisely, we consider symplectic manifolds that are the interior of Seidel's compact symplectic manifolds with boundary and corners.
We can in fact deal with more general noncompact exact manifolds, such as cotangent bundles or symplectic manifolds with convex ends. For that purpose note that we do not consider cylindrical ends mapping to noncompact symplectic manifolds, so the only technical requirement on exact manifolds is that the compactness in Theorem~\ref{quilttraj} holds, see the footnote in the proof there.
}
\end{itemize}

\begin{itemize}
\setlength{\itemsep}{1mm}
\item[\bf(L1):]
$L$ is {\em monotone}, that is $2 \int u^*\omega = \tau I(u)$ for all $[u] \in \pi_2(M,L)$.
Here $ I: \pi_2(M,L) \to \Z$ is the Maslov index and $\tau\geq 0$ is (necessarily) as in (M1).
\item[\bf(L2):]
$L$ is compact and oriented.
\item[\bf(L3):]
$L$ has (effective) minimal Maslov number $N_L\geq 3$.
Here $N_L$ is the generator of $I(\{[u]\in\pi_2(M,L)| \int u^*\omega >0\}) \subset \N$.
\end{itemize}

When working with Maslov coverings and gradings we will restrict our considerations to those that are compatible with orientations as follows.

\begin{itemize}
\setlength{\itemsep}{1mm}
\item[\bf(G1):]
$M$ is equipped with a Maslov covering $\Lag^N(M)$ for $N$ even,
and the induced $2$-fold Maslov covering $\Lag^2(M)$ is the space of oriented Lagrangian subspaces of $T M$.
\item[\bf(G2):]
$L$ is equipped with a grading $\sigma_L^N:L\to\Lag^N(M)$, and
the induced $2$-grading $L\to\Lag^2(M)$ is the one induced by the orientation of $L$.
\end{itemize}

Finally, note that our conventions differ from Seidel's definition of graded
Floer cohomology in \cite{se:gr} in two points which cancel each other:
The roles of $x_-$ and $x_+$ are interchanged and we switched
the sign of the Maslov index in the definition of the degree.

\section{Invariants for surfaces with strip-like ends}
\label{rel inv}

We begin with a formal definition of surfaces with strip-like ends
analogous to \cite[Section 2.4]{se:lo} (in the exact case) and \cite{sch:coh} (in
the case of surfaces without boundary).
To reduce notation somewhat we restrict to strip-like ends, i.e.\
punctures on the boundary.  One could in addition allow cylindrical
ends by adding punctures in the interior of the surface, see
Remark~\ref{cyl}.

\begin{definition} \label{surfstrip}
A {\em surface with strip-like ends} consists of the following data:

\begin{enumerate}
\item
$\ol{S}$  is a compact Riemann surface with boundary $\pd\ol{S}= C_1 \sqcup \ldots
\sqcup C_m $ and $d_n \ge 0$ distinct points $z_{n,1},\ldots,z_{n,d_n}\in C_n$ in cyclic order on each boundary circle $C_n\cong S^1$.  We will use the indices on $C_n$ modulo $d_n$, index all marked points by
$$ 
\E=\E(S)=\bigl\{ e=(n,l) \,\big|\, n\in\{1,\ldots,m\},
l\in\{1,\ldots,d_n\} \bigr\} ,
$$
and use the notation $e\pm 1:=(n,l\pm 1)$ for the cyclically
adjacent index to $e=(n,l)$.
We denote by $I_e=I_{n,l}\subset C_n$ the component of
$\partial S$ between $z_e=z_{n,l}$ and $z_{e+1}=z_{n,l+1}$.
However, $\partial S$ may also have compact components $I=C_n\cong S^1$.
\item
$j_S$  is a complex structure on $S:=\ol{S}\setminus\{z_e \,|\, e\in \E \}$.
\item
$(\eps_e)_{e\in\cE}$ is a set of {\em strip-like ends} for $S$, that is a set of embeddings with disjoint images
$$ 
\eps_e : \R^\pm \times [0,\delta_e] \to S 
$$
for all $e\in\cE$ such that $\eps_{e}(\R^\pm\times\{0,\delta_e\})\subset\partial S$,
$\lim_{s \to \pm \infty}(\eps_{e}(s,t)) = z_e$,
and $\eps_{e}^*j_S=j_0$ is the canonical complex structure on
the half-strip $\R^\pm \times[0,\delta_e]$ of width\footnote{
Note that here, by a conformal change of coordinates, we can always assume the
width to be $\delta_e=1$. The freedom of widths will only become relevant in the definition
of quilted surfaces with strip-like ends.
}
$\delta_e>0$.
We denote the set of incoming ends
$\eps_{e}: \R^- \times [0,\delta_e] \to S$ by $\E_-=\E_-(S)$ and the
set of outgoing ends $\eps_{e}: \R^+ \times [0,\delta_e] \to S$ by
$\E_+=\E_+(S)$.
\end{enumerate}
\end{definition}

Elliptic boundary value problems are associated to surfaces with
strip-like ends as follows.  Let $E$ be a complex vector bundle over
$S$ and $\F=(F_{I})_{I\in\pi_0(\partial S)}$ a tuple of totally real
subbundles $F_{I}\subset E|_{I}$ over the boundary components
$I\subset\partial S$.  Suppose that the sub-bundles
$F_{I_{e-1}},F_{I_{e}}$ are constant and intersect transversally in a
trivialization of $E$ near $z_e$ for each $e\in\cE$.\footnote
{As pointed out to us by the referee, one can avoid the use
  of trivializations here by assuming that the bundles $F_I$ are
  defined over the closure $\ol{I} \subset \ol{S}$ (which is
  diffeomorphic to a closed interval) and intersect transversely at
  the points of $\ol{S} - S$.}
Let
$$ 
D_{E,\F} : \Omega^0(S,E;\F) \to \Omega^{0,1}(S,E) 
$$
be a real Cauchy-Riemann operator acting on sections with boundary
values in $F_{I}$ over each component $I\subset\partial S$.
Transversality on the ends implies that the operator $D_{E,\F}$ is
Fredholm.  If $S=\overline{S}$ has no strip-like ends, and $S_0
\subset S$ denotes the union of components without boundary, we denote
by $I(E,\F)$ the topological index
$$ I(E,\F) = \deg(E | S_0) + \sum_{I\in\pi_0(\partial S)} I(F_{I}) ,$$
where $I(F_{I})$ is the Maslov index of the boundary data determined
from a trivialization of $E | ( S \backslash S_0) \cong (S \backslash S_0) \times\C^r$.
The index theorem for surfaces with boundary \cite[Appendix C]{ms:jh} implies
\begin{equation} \label{indthm}  \Ind(D_{E,\F}) = \rank_\C(E)\chi(S) + I(E,\F) .\end{equation}
Here the topological index $I(E,\F)$ is automatically even if the fibers of $\F$ are oriented, since loops of oriented totally real subspaces have even Maslov index and complex bundles have even degree.

A special case of these oriented totally real boundary conditions will
arise from oriented Lagrangian submanifolds.  We fix a compact,
monotone (or noncompact, exact) symplectic manifold $(M,\omega)$
satisfying (M1-2).  For every boundary component
$I\subset\pi_0(\partial S)$ let $L_I\subset M$ be a compact, monotone, Lagrangian submanifold satisfying (L1-2).  We will also write $L_e:=L_{I_e}$ for the Lagrangian associated to the noncompact boundary component $I_e\cong\R$ between $z_{e-1}$ and $z_e$.
In order for the Floer cohomologies $HF(L_{e},L_{e'})$ for $e,e'\in\cE$ to be well defined, we have to assume in addition that the Lagrangians $L_e, L_{e'}$ satisfy (L3).
The underlying Floer chain groups
$$
CF(L_{e},L_{e'}) := \sum_{x\in \cI(L_e,L_{e'})} \Z_2 \bra{x} , \qquad 
\cI(L_e,L_{e'}) := \phi_1(L_e) \pitchfork L_{e'}
$$
depend on a choice of regular Hamiltonian $H\in\cC^\infty(M)$ whose time-$1$-flow $\phi_1$ yields the above transverse intersection.
They carry a $\Z_2$-grading $|\bra{x}|\in\{\pm 1\}$  induced by the orientations of $(L_{e})_{e\in \E(S)}$, and the Floer differential is defined from counts $\#_2\M(x,y)_0\in\Z_2$ of the $0$-dimensional moduli spaces of nonconstant Floer trajectories from $x$ to $y$.
This involves choices of regular almost complex structures so that moduli spaces are cut out transversely. In particular, they can only be nonempty if $x$ and $y$ have different grading, hence Floer cohomology splits into graded summands $HF^{\even}(L_{e},L_{e'})$ and $HF^{\odd}(L_{e},L_{e'})$ generated by intersection points $x$ with $|\bra{x}|=1$ resp.\ $|\bra{x}|=-1$.

With these preparations we can construct moduli spaces of
pseudoholomorphic maps from the surface $S$.
For each pair $(L_{e-1},L_{e})$ for $e\in\cE_+$
resp.\ $(L_{e},L_{e-1})$ for $e\in\cE_-$ choose a regular pair
$(H_e,J_e)$ of Hamiltonian and almost complex structure (as in
\cite{quiltfloer})
such that the graded Floer cohomology
$HF(L_{e-1},L_{e})$ resp.\ $HF(L_{e},L_{e-1})$ is well defined.
Here $J_e:[0,\delta_e]\to \J(M,\omega)$ is a smooth family in the space of $\omega$-compatible almost complex structures on $M$.
Let $\Ham(S;(H_e)_{e\in\cE})$ denote the set of $C^\infty(M)$-valued
one-forms $K_S \in \Omega^1(S, C^\infty(M))$ such that ${K_S|_{\pd
    S}=0}$ and $\eps_{S,e}^* K_S = H_e\d t$ on each strip-like end.
Let $Y_S \in \Omega^1(S,\Vect(M))$ denote the corresponding
Hamiltonian vector field valued one-form, then $\eps_{S,e}^* Y_S$
equals to $X_{H_e} \d t$ on each strip-like end.  We denote by
$\J(S;(J_e)_{e\in\cE})$ the subset of
$J_S\in\cC^\infty(S,\J(M,\omega))$ that equals to the given
perturbation datum $J_{e}$ on each strip-like end.
We denote by $\cI_-$ the set of tuples
$X^-=(x_e^-)_{e\in \E_-}$ with $x_e^- \in
\cI(L_{e},L_{e-1})$ and by $\cI_+$ the set of tuples
$X^+=(x_e^+)_{e\in \E_+}$ with $x_e^+ \in \cI(L_{e-1},L_{e})$.
For each of these tuples we denote by
$$
\M_S(X^-,X^+) := \bigl\{ u:S \to M \,\big|\, (a)-(d) \bigr\}
$$
the space of $(J_S,K_S)$-holomorphic maps with Lagrangian boundary conditions,
finite energy, and fixed limits, that is
\ben
\item
$\overline\partial_{J,K} u := J_S(u) \circ (\d u - Y_S(u) ) - (\d u - Y_S(u) ) \circ j_S = 0$,
\item  $u(I) \subset L_I $ for all $I\in\pi_0(\partial S)$,
\item $E_{K_S}(u) := \int_S 
\frac12 \bigl| \d u - Y_S(u)\bigr|^2 {\rm d vol}_S
<\infty$,
\item $\lim_{s \to \pm \infty} u(\eps_{S,e}(s,t)) = x_e^\pm(t) $
for all $e\in \E_\pm$.
\een

\begin{remark} \label{monotone3}
\begin{enumerate}
\renewcommand\theenumi{\arabic{enumi}}
\item
For any map $u:S\to M$ that satisfies the Lagrangian boundary
conditions (b) and exponential convergence to the limits in (d) (and
hence automatically also satisfies (c)) one constructs a linearized
operator $D_u$ as in \cite{ms:jh} with a minor modification to handle
the boundary conditions: For any section $\xi$ of $E_u=u^*TM$
satisfying the linearized boundary conditions in
$\cF_u=\bigl((u|_I)^*TL_I\bigr)_{I\in\pi_0(\partial S)}$ set
$\sigma_u(\xi)=\Phi_u(\xi)^{-1} \overline\partial_{J,K}(\exp_u(\xi))$.
For the moduli space to be locally homeomorphic to the zero set of
$\sigma_u$, we need to define the exponential map $\exp$ such that
$\exp(TL_I)\subset L_I$, and $\Phi_u(\xi)$ has to be parallel
transport of $(0,1)$-forms from base point $u$ to base point
$\exp_u(\xi)$.  To satisfy the first requirement we choose metrics
$g_I$ on $M$ that make the Lagrangians $L_I$ totally geodesic and
define the exponential map $\exp:S\times TM \to M$ by using metrics
varying along $S$, equal to $g_I$ near the boundary components
$I\subset\partial S$. Restricting to a sufficiently small
neighbourhood $\U\subset TM$ of the zero section, we can rephrase this
as $\exp(z,\xi)=\exp^0(\xi + Q(z,\xi))$ in terms of a standard
exponential map $\exp^0:TM \to M$ using a fixed metric and a quadratic
correction $Q:S\times\U \to TM$ satisfying $Q(\cdot,0)=0$ and $d Q
(\cdot,0)=0$. In fact, the quadratic correction is explicitly given by
$Q(z,\xi\in T_p M\cap\U):=(\exp_p^0)^{-1}(\exp(z,\xi)) - \xi$.  To
satisfy the second requirement we define the parallel transport
$\Phi_u(\xi)$ by the usual complex linear connection constructed from
Levi-Civita connections of fixed metrics, but we do parallel transport
along the paths $s\mapsto \exp_{u(z)}(z,s\xi(z))=\exp^0_u(s\xi +
Q(s\xi))$.
%
%
Note that the derivative at $s=0$ of these paths still is $\xi$,
independent of the quadratic correction.

The operator $\sigma_u$ is well-defined for any $p>2$ as a map from the
$W^{1,p}$-closure of $\Omega^0(S,u^* TM)$ to the $L^p$-closure of
$\Omega^{0,1}(S,u^* TM)$.
We denote by $D_u:=d\sigma_u(0)$ its derivative at zero.  The linearized
operator constructed in this way takes the same form as in
\cite{ms:jh} -- it is in fact independent of the choice of quadratic
correction.  So it takes the form $D_u=D_{E_u,\cF_u}$ of a real
Cauchy-Riemann operator as discussed in the linear theory above.  When
$u$ is holomorphic, the operator $D_u$ is the linearized operator of
(a) -- defined independently from the choice of connection and metric.
\item
If the tuple of Lagrangians $(L_I)_{I\in\pi_0(\partial S)}$ is monotone
in the sense of \cite{quiltfloer}, then elements ${u\in\M_S(X^-,X^+)}$
(and more generally maps $u$ as in 1)) satisfy the area-index relation
\beq \label{clmm}
2 \tint_S u^* \omega = \tau \cdot {\rm Ind}(D_u) + c(X^-,X^+) .
\eeq
To prove this identity we fix one element $v_0\in\M_S(X^-,X^+)$ but view it as map $v_0: S^-\to M$ on the surface $S^-=(S,-j)$ with reversed orientation.
Now given any $u\in\M_S(X^-,X^+)$ we glue it with $v_0$ (by reparametrization that makes them constant near the ends) to a map $w=u\#v_0$ on the compact doubled surface $S\#S^-$, in which each end of $S$ is glued to the corresponding end of $S^-$.
Then a linear gluing theorem as in \cite[Section 3.2]{sch:coh} together with \eqref{indthm} provide
$$
{\rm Ind}(D_u) +  {\rm Ind}(D_{v_0}) =  {\rm Ind}(D_w)
=  I(w^*TM,(w^*TL_I)_{I\in\pi_0(\partial (S\#S^-))}) + \tfrac{\dim M}2 \chi(S\#S^-) .
$$
On the other hand, the monotonicity assumption yields
$$
2 \tint_S u^*\omega + 2 \tint_{S^-} v_0^*\omega = 2 \tint_{S\# S^-} w^*\omega = \tau I(w^*TM,(w^*TL_I)_{I\in\pi_0(\partial (S\#S^-))}) .
$$
Combining these relations proves \eqref{clmm} with
$c(X^-,X^+)= \tau\cdot{\rm Ind}(D_{v_0})  - 2 \tint v_0^*\omega -  \tau \tfrac{\dim M}2 \chi(S\#S^-) $.
\item
Elements ${u\in\M_S(X^-,X^+)}$ (unlike general maps $u:S\to M$) satisfy the energy identity
$$
E_{K_S}(u) = \tint_S  u^* \omega + \tint_S  ( R_K\circ u  ) \, {\rm dvol}_S
$$
with curvature term $R_K {\rm dvol}_S =  
 - \d K_S + \tfrac 12 \{ K_S \wedge K_S\} \in \Omega^2(S, C^\infty(M))$ given in local coordinates $(s,t)$ on $S$ with $K_S= F\d s + G \d t$ and ${\rm dvol}_S=\d s \wedge \d t$ by
$R_K =   
\partial_t F - \partial_s G +  \omega(X_F , X_G)$; see e.g.\ \cite[Lemma~8.1.6]{ms:jh}\footnote{
The uncommon sign in the curvature arises from the fact that $\d u - Y_S$ corresponds to $\d u + X_H$ in \cite{ms:jh}.
}.
Combining this with the monotonicity relation \eqref{clmm} we obtain the energy-index relation
\beq
 E_{K_S}(u) -   \tint_S  ( R_K\circ u  ) \, {\rm dvol}_S = \tfrac12 \tau \cdot {\rm Ind}(D_u) + \tfrac12 c(X^-,X^+) ,
\eeq
which gives an upper bound on the energy of solutions $u$ for fixed limits $X^\pm$ and index. 
Indeed, our assumptions on the Hamiltonian term $K_S$ guarantee that the curvature $R_K$ vanishes on the strip-like ends, hence is compactly supported on $S$. Moreover, $S\times M \to \R, (z,p) \mapsto R_K(z,p)$ is uniformly bounded by $\| R_K\|_\infty \le \|\d K_S\|_{\cC^0(S\times M)} + \| Y_S\|^2_{\cC^0(S\times M)}$, so that 
$\int_S  ( R_K\circ u  ) \, {\rm dvol}_S \le \| R_K\|_\infty {\rm Vol}({\rm supp}\, R_K)$ is bounded independent of the map $u$.
\item
If the Lagrangians $(L_I)_{I\in\pi_0(\partial S)}$ are all oriented, then the index
$ {\rm Ind}(D_u)$ of a map $u$ as in (1) is determined $\text{mod}\;2$ by the surface $S$ and limit conditions $X^-,X^+$.
This is since in the above index formula a change of $u$ with fixed ends is only reflected in a change of the topological index $ I(w^*TM,(w^*TL_I)_{I\in\pi_0(\partial S\#S^-)})$.
This index however is always even since by (L2) the totally real subbundles $w^*TL_I$ are oriented.

\end{enumerate}
\end{remark}

\begin{figure}[ht]
\begin{picture}(0,0)
\includegraphics{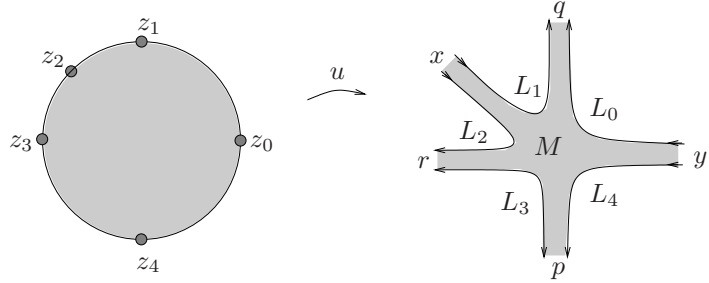}
\end{picture}
\setlength{\unitlength}{2072sp}
\begingroup\makeatletter\ifx\SetFigFont\undefined
\gdef\SetFigFont#1#2#3#4#5{
  \reset@font\fontsize{#1}{#2pt}
  \fontfamily{#3}\fontseries{#4}\fontshape{#5}
  \selectfont}
\fi\endgroup
\begin{picture}(8041,3057)(1021,-2977)
\put(2150,-200){\makebox(0,0)[lb]{{$z_1$}}}
\put(621,-1600){\makebox(0,0)[lb]{{$z_3$}}}
\put(2142,-3056){\makebox(0,0)[lb]{{$z_4$}}}
\put(1000,-600){\makebox(0,0)[lb]{{$z_2$}}}
\put(3501,-1600){\makebox(0,0)[lb]{{$z_0$}}}
\put(4456,-736){\makebox(0,0)[lb]{{$u$}}}
\put(6001,-1501){\makebox(0,0)[lb]{{$L_2$}}}
\put(6656,-961){\makebox(0,0)[lb]{{$L_1$}}}
\put(7561,-1231){\makebox(0,0)[lb]{{$L_0$}}}
\put(6900,-1646){\makebox(0,0)[lb]{{$M$}}}
\put(7561,-2256){\makebox(0,0)[lb]{{$L_4$}}}
\put(6500,-2351){\makebox(0,0)[lb]{{$L_3$}}}
\put(5650,-551){\makebox(0,0)[lb]{{$x$}}}
\put(8806,-1801){\makebox(0,0)[lb]{{$y$}}}
\put(7111,-3141){\makebox(0,0)[lb]{{$p$}}}
\put(7126,-16){\makebox(0,0)[lb]{{$q$}}}
\put(5506,-1801){\makebox(0,0)[lb]{{$r$}}}
\end{picture}
\caption{A holomorphic curve $u\in\M_{S}((x,y),(p,q,r))$ for a surface
$S$ with ends $\E_-=\{2,0\}$ and $\E_+=\{4,1,3\}$}
\end{figure}

\begin{theorem} \label{package}
Suppose that $(L_I)_{I\in\pi_0(\partial S)}$ is a monotone tuple of
Lagrangian submanifolds satisfying (L1-2) and (M1-2), and that regular
perturbation data $(H_e,J_e)$ are chosen for each end $e\in\E$.  Then
for any $J_S \in \J(S;(J_e)_{e\in\E})$ there exists a comeagre (in particular dense)\footnote{A subset of a topological space is {\it comeagre} if
  it is the intersection of countably many open dense subsets. In a
  Baire space, this intersection is still dense, and the space of
  Hamiltonian perturbations fixed on the complement of a compact set
  is naturally a Baire space.  }subset
$\Ham^\reg(S;(H_e)_{e\in\E},J_S) \subset \Ham(S;(H_e)_{e\in\E})$ such
that for any
$H_S\in\Ham^\reg(S;(H_e)_{e\in\E},J_S)$
the following holds for all tuples $(X^-,X^+)\in\cI_-\times\cI_+$.
\ben
\item  $\M_S(X^-,X^+)$ is a smooth manifold.
\item   The zero dimensional component
$\M_S(X^-,X^+)_0$ is finite.
\item
The one-dimensional component $\M_S(X^-,X^+)_1$ has a
compactification as a one-manifold with boundary
\begin{align*}
\partial \overline{\M_S(X^-,X^+)_1}
&\cong \bigcup_{ e\in \E_- ,\, y\in\cI(L_{e},L_{e-1}) }
\M(x^-_e,y)_0\times\M_S(X^-|_{x^-_e\to y},X^+)_0  \\
&\quad \cup \bigcup_{ e\in \E_+,\, y\in\cI(L_{e-1},L_{e}) }
\M_S(X^-,X^+|_{x^+_e\to y})_0 \times \M(y,x^+_e)_0 ,
\end{align*}
where the tuple $X|_{x_e\to y}$ is $X$ with the
intersection point $x_e$ replaced by $y$.
\een
\end{theorem}

The proof is similar to the regularity and compactness theorems \cite{oh:fl1} in monotone Floer theory; see \cite{se:bo} for the exact case.
In the moduli spaces of dimension $0$ and $1$ the bubbling of spheres and disks is
ruled out by monotonicity and (L2).
We can thus define
$$ 
C\Phi_S: \bigotimes_{e\in \E_-} CF(L_{e},L_{e-1})
\to \bigotimes_{e\in \E_+} CF(L_{e-1},L_{e}) 
$$
with $\Z_2$ coefficients by 
$$ 
C\Phi_S \biggl( \bigotimes_{e\in \E_-} \bra{x^-_e} \biggr)
:= \sum_{X^+\in\cI_+} 
\#_2 \M_S(X^-,X^+)_0
\bigotimes_{e\in \E_+} \bra{x^+_e} .
$$
By item (c), the maps $C\Phi_S$ are chain maps
and so descend to a map of Floer cohomologies
\begin{equation} \label{relative}
 \Phi_S:  \bigotimes_{e\in \E_-} HF(L_{e},L_{e-1})
 \to \bigotimes_{e\in \E_+} HF(L_{e-1},L_{e}) .
\end{equation}
Floer's argument using parametrized moduli spaces
carries over to this case to show that
$\Phi_{S}$ is independent of the choice of perturbation data,
complex structure $j$ on $S$, and the strip-like ends.
We sketch this argument in Section~\ref{sec:invariance} for the more general quilted case.

\begin{remark}  \label{rmk oriented S}
To work over other coefficient rings such as $\Z$, we need to assume that the tuple $(L_I)_{I\in\pi_0(\partial S)}$ is {\em relatively
spin}. This means that all Lagrangians $L_I$ are relatively spin with respect to one
fixed background class $b\in H^2(M,\Z_2)$,  see \cite{orient} for more details. 
A choice of such relative spin structures induces orientations $\epsilon:\M(x,y)_0 \to \{\pm 1\}$ which allow to define the Floer cohomology groups over $\Z$ by replacing
$\#_2 \M(x,y)_0$ with $\sum_{u \in \M(x,y)_0} \eps(u)$.

To define the relative invariant $\Phi_S$ with $\Z$-coefficients we also need to choose an ordering of the set of (compact) boundary components of $\ol{S}$ and orderings
$ \E_-=(e^-_1,\ldots, e^-_{N_-}), \E_+=(e^+_1,\ldots, e^+_{N_+})$
of the sets of incoming and outgoing ends.
Then in Theorem~\ref{package}, by \cite{se:bo} or \cite{orient}, there exists a coherent set of orientations $\eps$ on the zero and one-dimensional moduli spaces so that the inclusion of the boundary in (c)
has the signs $(-1)^{\sum_{f<e} |x_{f}^-|}$ (for incoming trajectories) and
$-(-1)^{\sum_{f<e} |x_{f}^+|}$ (for outgoing trajectories).
This implies that $C\Phi_S: \bigotimes_{e\in \E_-} CF(L_{e},L_{e-1})
 \to \bigotimes_{e\in \E_+} CF(L_{e-1},L_{e})$, defined by replacing $\#_2 \M_S(X^-,X^+)_0$ with $\sum_{u \in \M_S(X^-,X^+)_0} \eps(u)$, is a chain map with $\Z$-coefficients between the tensor products of Floer complexes.

If we are working with field coefficients such as $\Z_2$, then the K\"unneth theorem identifies the homology of these tensor products with the tensor products of the Floer cohomologies, so that $C\Phi_S$ descends to the latter. For general ring coefficients such as $\Z$, we obtain a map $\Phi_S: \bigotimes_{e\in \E_-} HF(L_{e},L_{e-1})
 \to H_*\bigl( \bigotimes_{e\in \E_+} CF(L_{e-1},L_{e}) \bigr)$ that factors through $\bigotimes_{e\in \E_+} HF(L_{e-1},L_{e})$ only up to some torsion terms. However, note that in the special cases of no outgoing ends $\cE_+=\emptyset$ or a single outgoing end $\cE_+=\{e_+\}$, we always obtain a well defined map $\Phi_S$ to $\Z$ resp.\ $HF(L_{e_+-1},L_{e_+})$.
 \end{remark}

\begin{remark}  \label{rmk graded S}
To obtain a graded theory let $M$ be equipped with an $N$-fold Maslov covering satisfying (G1), and suppose that each $L_I\subset M$ is a graded Lagrangian submanifold satisfying (G2). 
A choice of such Maslov covers induces a grading $|x|\in\Z_N$ on the Floer cohomology groups such that $|\bra{x}|=(-1)^{|x|}$, by \cite{se:gr}.
To study the interaction of $\Phi_S$ with the gradings, suppose that $S$ is connected.  Then we claim that the effect of the relative invariant $\Phi_S$ on the
grading is by a shift in degree of
$$ |\Phi_S| = \frac 12 \dim M (  \#\cE_+ - \chi(\overline{S}) ) \ \
\text{mod} \ N .$$
That is, the coefficient of $ C\Phi_S \bigl( \bigotimes_{e\in \E_-}
\bra{x^-_e} \bigr)$ in front of $\bigotimes_{e\in \E_+} \bra{x^+_e}$
is zero unless the degrees
$|x_e^-|=d(\sigma_{L_e}^N(x_e^-),\sigma_{L_{e-1}}^N(x_e^-))$ and
$|x_e^+|=d(\sigma_{L_{e-1}}^N(x_e^+),\sigma_{L_e}^N(x_e^+))$ satisfy
$$
\sum_{e \in \cE_+} |x_e^+| -
\sum_{e \in \cE_-} |x_e^-|
= \tfrac 12 \dim M \bigl( \#\cE_+ -  \chi(\overline{S}) \bigr) \qquad \text{mod} \ N .
$$
Here $\#\cE_+$ is the number of outgoing ends of $S$. So, for example,
$\Phi_S$ preserves the degree if $S$ is a disk with one outgoing end
and any number of incoming ends.

To check the degree identity fix paths
$\ti{\Lambda}_e:[0,1]\to\Lag^N(T_{x_e}M)$ from
$\sigma_{L_{e-1}}^N(x_e)$ to $\sigma_{L_e}^N(x_e)$ for each end
$e\in\cE(S)$, and denote their projections by ${\Lambda_e:[0,1]\to
T_{x_e}M}$.  Let $D_{T_{x_e}M,\Lambda_e}$ be the Cauchy-Riemann
operator in $T_{x_e}M$ on the disk with one incoming strip-like end
and with boundary conditions $\Lambda_e$.  Then (see \cite{quiltfloer}) we have
$$|x_e^\pm| = {\rm Ind}(D_{T_{x_e}M,\Lambda_e^{\pm 1}}), \ \ \ e\in\cE_\pm$$
with the reversed path
$\ti{\Lambda}_e^{-1}$ from $\sigma_{L_e}^N(x_e^-)$
to $\sigma_{L_{e-1}}^N(x_e^-)$ in case $e\in\cE_-$. In this case we have
$${\rm Ind}(D_{T_{x_e}M,\Lambda_e^{-1}}) + {\rm Ind}(D_{T_{x_e}M,\Lambda_e})= \tfrac 12
\dim M \qquad\text{mod}\; N$$
since gluing the two disks gives rise to a Cauchy-Riemann operator on
the disk with boundary conditions given by the loop
$\Lambda_e^{-1}\#\Lambda_e$, which lifts to a loop in
$\Lag^N(T_{x_e}M)$ and hence has Maslov index $0\;\text{mod}\; N$.
Now consider an isolated solution $u \in \M_S(X^-,X^+)_0$.
For each end $e\in\cE(S)$ we can glue the operator
$D_{T_{x_e}M,\Lambda_e^{-1}}$
on the disk to the linearized Cauchy-Riemann
operator $D_{u^*TM,(u^*TL_I)_{I\in\pi_0(\partial S)}}$ on the surface $S$.  This gives rise
to a Cauchy-Riemann operator on the compact surface $\overline{S}$
with boundary conditions given by Lagrangian subbundles
(given by $u^*TL_I$ for compact boundary components $I\subset\partial S$ and
composed of $u^*TL_e$ and $\Lambda_e^{-1}$ for noncompact components $I_e$)
that lift to loops in $\Lag^N(M)$ (given by $\sigma_{L_I}\circ u|_I$ resp.\
composed of $\sigma_{L_e}\circ u|_{I_e}$ and $(\ti{\Lambda}_e)^{-1}$).
In a trivialization of $u^*TM$ their Maslov indices are hence divisible by $N$, and so the
index of the glued Cauchy-Riemann operator is
\begin{align*}
\tfrac 12 \dim M \cdot \chi(\overline{S})
&= {\rm Ind}(D_{u^*TM,(u^*TL_I)_{I\in\pi_0(\partial S)}}) +
\sum_{e\in\cE(S)}{\rm Ind}(D_{TM,\Lambda_e^{-1}})
 \qquad\text{mod}\; N \\
& = 0 + \sum_{e\in\cE_+(S)}  ( \tfrac 12\dim M -|x_e^+|)
+ \sum_{e\in\cE_-(S)} |x_e^-|
 \qquad\text{mod}\; N .
\end{align*}
\end{remark}

In the following we assume that gradings are given as in the preceding remark and will work through some examples to prepare for the gluing result, which will serve as template for the quilted gluing result.

\begin{example}[Strip Example]  \label{strip}
If $S_\parallel$ is the strip $\R\times[0,1]$ (i.e.\ the disk with one
incoming and one outgoing puncture) then we can choose perturbation
data that preserve the $\R$-invariance of the holomorphic
curves. Then any nonconstant solution comes in a $1$-dimensional
family and hence $\Phi_{S_\parallel}$ is generated by the constant solutions.
The same holds for the disks $S_\cap$ and $S_\cup$ with two incoming resp.\
two outgoing ends.
From the strip we obtain the identity map of degree $0$ 
$$
\Phi_{\parallel}:=\Phi_{S_\parallel}={\rm Id}:HF(L_0,L_1)\to HF(L_0,L_1).
$$
(This continues to hold with $\Z$ coefficients by our choice of coherent orientations in \cite{orient}.)
The disks $S_\cap$ and $S_\cup$ give rise to maps of degree
$-\frac 12\dim M$ and $\frac 12\dim M$ respectively,
\begin{equation} \label{phicap}
\Phi_\cap:=\Phi_{S_\cap} : HF(L_0,L_1)\otimes HF(L_1,L_0) \to \Z_2
,\end{equation}
\begin{equation} \label{phicup}
 \Phi_\cup:=\Phi_{S_\cup} :\Z_2 \to HF(L_0,L_1)\otimes
 HF(L_1,L_0). \end{equation}
Let us write $\bra{x_i}_{01} \in CF(L_0,L_1)$ resp.\ $\bra{x_i}_{10} \in CF(L_1,L_0)$ for the generators corresponding to the intersection points $\cI(L_0,L_1)\cong\cI(L_1,L_0)=\{x_1,x_2, \ldots \}$.
Their degrees are related by $|\bra{x_i}_{01}|+|\bra{x_i}_{10}| =\frac 12\dim M$ as in Remark~\ref{rmk graded S}. Then on the chain level the maps are given by 
$C\Phi_\parallel : \bra{x_i}_{01} \mapsto\bra{x_i}_{01}$,
$ 
C\Phi_\cap: \bra{x_i}_{01} \otimes \bra{x_j}_{10} \mapsto \delta_{ij}$, and 
$C\Phi_\cup: 1 \mapsto \sum_i \bra{x_i}_{01} \otimes \bra{x_i}_{10}$.
(For $\Z$ coefficients as in Remark~\ref{rmk oriented S}, the maps $C\Phi_\cap$ and $C\Phi_\cup$ will pick up signs that are discussed in \cite{orient}.)
\end{example}

The relative invariants evidently satisfy a tensor product law for disjoint union
of two surfaces $S_1,S_2$ with strip like ends,
\begin{equation}\label{eq disjoint}
 \Phi_{S_1 \sqcup S_2} = \Phi_{S_1} \otimes \Phi_{S_2} .
\end{equation}
(This continues to hold with $\Z$ coefficients by an appropriate definition of the graded tensor product, see \cite{orient}.)
After these preparations we can see that
the relative invariants satisfy a composition law for gluing along
ends. Suppose that $e_+\in \E_+(S)$ and
$e_-\in \E_-(S)$ are outgoing resp.\ incoming ends of $S$
such that the Lagrangians agree, $L_{e_+-1}=L_{e_-}$ and
$L_{e_+}=L_{e_--1}$.
Then we can algebraically define the trace of $\Phi_{S}$ at $(e_-,e_+)$
$$ \Tr_{e_-,e_+}(\Phi_{S}):
\bigotimes_{e\in \E_-(S) \setminus \{ e_- \}} HF(L_{e},L_{e-1})
\;\;\to \bigotimes_{e\in \E_+(S) \setminus \{ e_+\}} HF(L_{e-1},L_{e})
$$
by
\begin{align} \label{trace}
 \Tr_{e_-,e_+}(\Phi_{S}) &:= \bigl( {\rm Id}^{\E_+ \setminus \{e_+\} }
\otimes \Phi_\cap^{e_+,e_0} \bigr) \circ
\bigl( \Psi_{e_+} \otimes {\rm Id}^{e_0} \bigr) \circ \bigl( \Phi_{S}
\otimes {\rm Id}^{e_0} \bigr) \\ &\qquad \circ \bigl( \Psi_{e_-}
\otimes {\rm Id}^{e_0} \bigr) \circ \bigl({\rm Id}^{\E_-\setminus \{e_-\}}
\nonumber \otimes \Phi_\cup^{e_-,e_0}\bigr), 
\end{align}
where
\begin{itemize}
\item superscripts indicate the ends (and associated Floer cohomology
groups) that the maps act on,
\item $e_0$ is an additional end associated to the Floer cohomology
group $HF(L_{e_--1},L_{e_-})= HF(L_{e_+},L_{e_+-1})$,
\item $\Psi_{e_\pm}$ are the permutations of the factors in the graded
tensor product needed to make the compositions well-defined (and for $\Z$ coefficients involve appropriate signs).
\end{itemize}

On the other hand, let $\#^{e_-}_{e_+}(S)$ denote the surface obtained by gluing together the ends $e_\pm$.
The glued surface $\#^{e_-}_{e_+}(S)$ can be written as the ``geometric trace''
$$
\bigl( S_{\parallel}^{\E_+ \setminus \{e_+\} } \sqcup S_\cap^{e_+,e_0} \bigr)
\# \bigl( S_{\Psi_{e_+}} \sqcup S_\parallel^{e_0} \bigr)
\# \bigl( S \sqcup S_\parallel^{e_0} \bigr)
\# \bigl( S_{\Psi_{e_-}} \sqcup S_\parallel^{e_0} \bigr)
\# \bigl( S_\parallel^{\E_-\setminus \{e_-\}} \sqcup S_\cup^{e_-,e_0} \bigr),
$$ 
where $S_0\#S_1$ denotes the gluing of all incoming ends of $S_0$
to the outgoing ends of $S_1$ (which must be identical and in the same
order).  Here superscripts indicate the indexing of the ends of the
surfaces, so e.g\ $S_{\parallel}^{\E_+ \setminus \{e_+\} }$ is a product
of strips $\R\times[0,1]$ with both incoming and outgoing ends indexed
by $\E_+ \setminus\{e_+\}$.  The surfaces $S_{\Psi_{e_\pm}}$ are the
products of strips with incoming ends indexed by
$(\E_-\setminus\{e_-\},e_-)$ resp.\ $\E_+$ and outgoing ends indexed
by $\E_-$ resp.\ $(\E_+\setminus\{e_+\},e_+)$ (in the order
indicated).  The relative invariants associated to the surfaces in
this geometric trace are exactly the ones that we compose in the
definition \eqref{trace} of the algebraic trace.  In fact, the
standard Floer gluing construction implies the following analogue of
the gluing formula \cite[2.30]{se:lo} in the exact case.

\begin{figure}[ht]
\begin{picture}(0,0)
\includegraphics{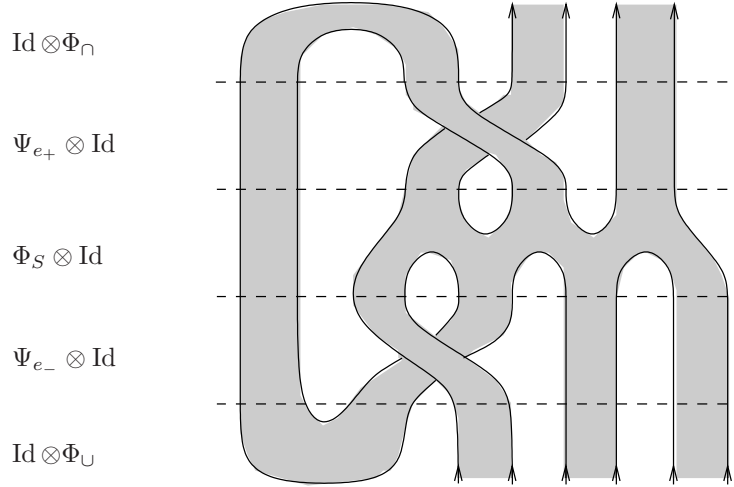}
\end{picture}
\setlength{\unitlength}{4144sp}
\begingroup\makeatletter\ifx\SetFigFont\undefined
\gdef\SetFigFont#1#2#3#4#5{
  \reset@font\fontsize{#1}{#2pt}
  \fontfamily{#3}\fontseries{#4}\fontshape{#5}
  \selectfont}
\fi\endgroup
\begin{picture}(4447,2916)(-914,-3576)
\put(-899,-993){\makebox(0,0)[lb]{{{$\Id \otimes \Phi_{\cap}$}
}}}
\put(-899,-3453){\makebox(0,0)[lb]{{{$\Id \otimes \Phi_{\cup}$}
}}}
\put(-899,-2268){\makebox(0,0)[lb]{{{$\Phi_S \otimes \Id$}
}}}
\put(-899,-1625){\makebox(0,0)[lb]{{{$\Psi_{e_+} \otimes \Id$}
}}}
\put(-899,-2911){\makebox(0,0)[lb]{{{$\Psi_{e_-} \otimes \Id$}
}}}
\end{picture}

\caption{Gluing example for a connected surface $S$}
\label{gluingexample}
\end{figure}

\begin{theorem}[Gluing Theorem]   \label{thm Fglue1}
Let $S$ be a surface with strip-like ends and $(L_e)_{e\in\cE(S)}$ Lagrangians
as in Theorem \ref{package}, satisfying in addition (L3).
Suppose that $e_\pm\in \E_\pm(S)$ such that $L_{e_+-1}=L_{e_-}$ and
$L_{e_+}=L_{e_--1}$. Then
\begin{equation*}
\Phi_{\#^{e_-}_{e_+}(S)} = \Tr_{e_-,e_+} (\Phi_{S}).
\end{equation*}
\end{theorem}


\begin{proof}[Sketch of Proof:]
By Theorem~\ref{thm:inv}, the relative invariant for the geometric trace can be computed using a surface with long necks between the glued surfaces.  Solutions (of both the linear and non-linear equation) on this surface are in one-to-one correspondence with pairs of solutions on the two separate surfaces; counting the latter exactly corresponds to composition.  The one-to-one correspondence is proven by an implicit function theorem (using the fact that the linearized operator as well as its adjoint are surjective in the index $0$ case) and a compactness result (using monotonicity to exclude bubbling). Details for the analogous closed case can be found in \cite[Chapter 5.4]{sch:coh}.  
By definition, our algebraic trace is the composition of the relative invariants of
$S_{\parallel}^{\E_+ \setminus \{e_+\} } \sqcup S_\cap^{e_+,e_0}$,
$S_{\Psi_{e_+}} \sqcup S_\parallel^{e_0}$,
$ S \sqcup S_\parallel^{e_0}$,
$S_{\Psi_{e_-}} \sqcup S_\parallel^{e_0}$,
and $S_\parallel^{\E_-\setminus \{e_-\}} \sqcup S_\cup^{e_-,e_0}$,
see \eqref{eq disjoint} 
%
\end{proof}

The main applications of this theorem are the following special cases:
Gluing a surface with one outgoing end to the first incoming end of another surface,
and gluing the ends of a surface with single incoming and outgoing ends.
(For $\Z$ coefficients, all these gluing formulas carry additional signs determined by the ordering conventions for ends and boundary components, see \cite{orient}.)

\begin{corollary}\label{cor Fglue2}
Let $S$ be a disjoint union $S = S_0 \sqcup S_1$ with $z_{e_-}\in S_0$, $z_{e_+}\in S_1$, and if $S_1$ has a single outgoing end $\cE_+(S_1)=\{e_+\}$.
Then we have
$$ 
\Phi_{(S_0)\#^{e_-}_{e_+}(S_1)} =
\Phi_{S_0}\circ ( 1^{\E_-(S_0) \setminus \{e_-\} } \otimes  \Phi_{S_1}  ) .
$$
\end{corollary}

\begin{corollary} \label{cor Fglue3}
Let $S$ be connected with exactly one incoming and one outgoing end $\cE=\{e_+,e_-\}$ lying on the same boundary component.
Then we have
\begin{equation} \label{truetrace}
 \Phi_{\#^{e_-}_{e_+}(S)}: 1 \mapsto \Tr(\Phi_S) = \sum_{i} \bra{C\Phi_S(\bra{x_i}) , \bra{x_i}},
\end{equation}
the sum over the $\bra{x_i}$ coefficients of $C\Phi_S(\bra{x_i})$.
\end{corollary}

\begin{example}
The gluing formulas allow to compute the invariants for closed surfaces as follows:

\begin{enumerate} \label{closedcount}
\item (Disk) If $S$ is the disk with boundary condition $L$, then
$\Phi_S$ is the number of isolated perturbed $J$-holomorphic disks
with boundary in $L$.  Because of the monotonicity assumption, and
since we do not quotient out by automorphisms of the disk, each component
of the moduli space of such disks has at least the dimension of $L$,
hence $\Phi_S = 0$.

\item (Annulus) 
Let $A=\# S_\parallel$ denote the annulus, obtained by gluing along the two ends of the infinite strip $S_\parallel=\R\times[0,1]$ with boundary conditions $L_0$ and $L_1$.
Then the gluing formula produces the identity in $\Z_2$
$$ 
\Phi_A =\Tr(\Id) = 
 \rank HF^{\even}(L_0,L_1) - \rank
HF^{\odd}(L_0,L_1) 
\quad {\rm mod}\; 2 .
$$
The same result can be obtained by decomposing the annulus into cup and cap,
$\Phi_A=\Phi_{\cap} \circ \Phi_{\cup}$.

\item (Sphere with holes) 
Let $S$ denote the sphere with $g+1$ disks removed and boundary condition $L$ over each component.  $S$ can be obtained by gluing together $g-1$ copies of the surface $S_0$, which is obtained by removing a disk from the strip $\R \times[0,1]$; see
Figure \ref{S21}.  The latter defines an automorphism $\Phi_{S_0}$ on
$HF(L,L)$, and the gluing formulas give
$$ 
\Phi_S = \Tr(\Phi_{S_0}^{g-1})
= \sum 
\bra{\Phi_{S_0}^{g-1}(\bra{x_i}),\bra{x_i}} .
$$
\end{enumerate}
\end{example}

\begin{figure}[ht]
\includegraphics[width=1.5in,height=2in]{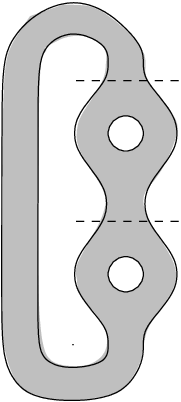}
%
\caption{Gluing copies of $S_0$}
\label{S21}
\end{figure}

\begin{remark} \label{cyl} If $M$ is compact, then one can also allow surfaces to have incoming or outgoing cylindrical ends, equipped with a periodic Hamiltonian perturbation.
In this case the relative invariant acts on the product of Floer
cohomology groups with a number of copies of the cylindrical Floer
cohomology $HF(\Id)$, isomorphic to the quantum cohomology $QH(M)$ of
$M$.  For instance, a disk with one puncture in the interior gives
rise to a canonical element $\phi_L \in HF(\Id)$.
Splitting the annulus into two half-cylinders glued at a cylindrical end
gives rise to the identity in $\Z_2$
$$ 
\lan\phi_{L_0},\phi_{L_1}\ran_{HF(\Id)} 
= \rank HF^\even(L_0,L_1) - \rank HF^\odd(L_0,L_1) 
\quad {\rm mod}\; 2.
$$
By considering a disk with one interior and one
boundary puncture one obtains the open-closed map $HF(L,L) \to HF(\Id)$, which is discussed by Albers \cite[Theorem 3.1]{alb:ext}.
\end{remark}

\section{Invariants for quilted surfaces}
\label{inv quilts}

Quilted surfaces are obtained from a collection of surfaces with
strip-like ends by ``sewing together'' certain pairs of boundary
components.  We give a formal definition below, again restricting to
strip-like ends, i.e.\ punctures on the boundary.  One could in
addition allow cylindrical ends by adding punctures in the interior of
the surface, see Remark~\ref{qcyl}.

\begin{definition}
A {\em quilted surface} $\ul{S}$ with strip-like ends consists of the
following data:
\ben
\item A collection $\ul{S} = (S_k)_{k=1,\ldots,m}$ of {\em patches}, that is
surfaces with strip-like ends as in Definition~\ref{surfstrip}~(a)-(c).  In
particular, each $S_k$ carries a complex structure $j_k$ and has
strip-like ends $(\eps_{k,e})_{e\in \E(S_k)}$ of widths $\delta_{k,e}>0$
near marked points $\lim_{s \to \pm \infty} \eps_{k,e}(s,t) = z_{k,e} \in
\partial\overline{S}_k$.
We denote by $I_{k,e}\subset\partial S_k$ the noncompact boundary component between
$z_{k,e-1}$ and $z_{k,e}$.
\item
A collection $\S = \bigl(  \{(k_\sigma,I_\sigma),(k_\sigma',I_\sigma')\} \bigr)_{\sigma\in\S}$
of {\em seams}, that is pairwise disjoint $2$-element subsets
$$ 
\sigma \subset \bigcup_{k=1}^m \{ k \} \times \pi_0(\partial S_k) , 
$$
and for each $\sigma \in \S$, a diffeomorphism of boundary components
$$
\varphi_\sigma: \partial S_{k_\sigma} \supset I_\sigma \overset{\sim}{\to}
I_{\sigma}'\subset \partial S_{k_\sigma'}  
$$
that is
\begin{enumerate}
\item {\it real analytic:}  Every $z\in I_\sigma$ has an open neighbourhood
$\U\subset S_{k_\sigma}$ such that $\varphi_\sigma|_{\U\cap I_\sigma}$ extends to an embedding
$\psi_z:\U\to S_{k'_\sigma}$ with $\psi_z^* j_{k'_\sigma} = - j_{k_\sigma}$.
In particular, this forces $\varphi_\sigma$ to reverse the orientation on the boundary components.
\footnote{
In order to drop the real analyticity condition, one would have to extend each part of the analysis for pseudoholomorphic curves with Lagrangian boundary conditions (e.g.\ linear theory, nonlinear regularity, and removal of singularities) to the case of maps $(u_0,u_1): (\H^2,\partial\H^2) \to (M_0 \times M_1, L_{01})$ on the half space with Lagrangian boundary conditions such that each component $u_i$ is pseudoholomorphic with respect to a different complex structure $j_i$ on $\H^2$, that is, $J_i \circ \d u_i = \d u_i \circ j_i$ for $i = 0,1$.
This requires a lot of technical work, parts of which are completed in \cite{bottman}.
However, none of this work is necessary if one restricts to real analytic seams. In that case, the only fundamental difficulty is resolved by Lemma~\ref{homotopy of quilts} establishing homotopies between quilted surfaces of the same combinatorial type.
}
\item {\it compatible with strip-like ends :}
If $I_{\sigma}$ (and hence $I_{\sigma}'$) is noncompact, i.e.\ lie between marked points,
$I_{\sigma}=I_{k_\sigma,e_\sigma}$ and $I_{\sigma}'=I_{k_\sigma',e_\sigma'}$, then we
require that $\varphi_\sigma$ matches up the end $e_\sigma$ with $e_\sigma'-1$
and the end $e_\sigma-1$ with $e_\sigma'$.
That is $\eps_{k_\sigma',e_\sigma'}^{-1}\circ\varphi_\sigma\circ\eps_{k_\sigma,e_\sigma-1}$ maps $(s,\delta_{k_\sigma,e_\sigma-1}) \mapsto (s,0)$ if both ends are incoming, or it maps
$(s,0)\mapsto (s,\delta_{k_\sigma',e_\sigma'})$ if both ends are outgoing.
We disallow matching of an incoming with an outgoing end, and the condition on the other pair of ends is analogous,
see Figure~\ref{realseam}.
\end{enumerate}
\item
As a consequence of (a) and (b) we obtain a set of remaining boundary components $I_b\subset\partial S_{k_b}$ that are not identified with another boundary component of $\ul{S}$.
These {\em true boundary components} of $\ul{S}$ are indexed by
$$ 
\B  = \bigl(  (k_b,I_b) \bigr)_{b\in\B}
 := \bigcup_{k=1}^m \{ k \} \times \pi_0(\partial S_k) \; \setminus \;
\bigcup_{\sigma\in\S} \sigma .
$$
Moreover, we can read off the {\em quilted ends}
$\ul{e}\in \E(\ul{S})=\E_-(\ul{S})\sqcup \E_+(\ul{S})$ consisting of a
maximal sequence $\ul{e}=(k_i,e_i)_{i=1,\ldots,n_{\ul{e}}}$ of ends of patches
with boundaries $\eps_{k_i,e_i}(\cdot,\delta_{k_i,e_i}) \cong
\eps_{k_{i+1},e_{i+1}}(\cdot,0)$ identified via some seam
$\phi_{\sigma_i}$.
This end sequence could be cyclic, i.e.\ with an additional identification $\eps_{k_n,e_n}(\cdot,\delta_{k_n,e_n}) \cong \eps_{k_{1},e_{1}}(\cdot,0)$ via some seam $\phi_{\sigma_n}$.
Otherwise the end sequence is noncyclic, i.e.\
$\eps_{k_1,e_1}(\cdot,0)\in I_{b_0}$ and  $\eps_{k_n,e_n}(\cdot,\delta_{k_n,e_n})\in I_{b_n}$ take value in some true boundary components $b_0,b_n\in\B$.
In both cases, the ends $\eps_{k_i,e_i}$ of patches in one quilted end $\ul{e}$ are either all incoming, $e_i\in \E_-(S_{k_i})$, in which case we call the quilted end incoming, $\ul{e}\in\cE_-(\ul{S})$, or they are all outgoing, $e_i\in \E_+(S_{k_i})$, in which case we call the quilted end incoming, $\ul{e}\in\cE_+(\ul{S})$.

%
\een
\end{definition}

\begin{figure}[ht]
\begin{picture}(0,0)
\includegraphics{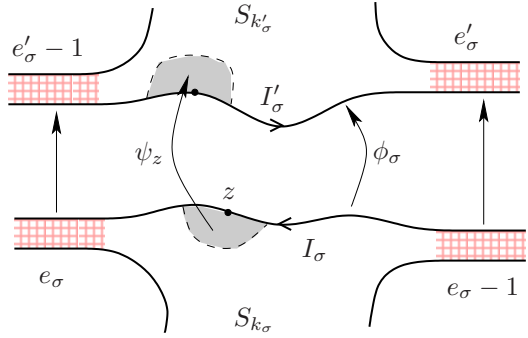}
\end{picture}
\setlength{\unitlength}{3315sp}
\begingroup\makeatletter\ifx\SetFigFont\undefined
\gdef\SetFigFont#1#2#3#4#5{
  \reset@font\fontsize{#1}{#2pt}
  \fontfamily{#3}\fontseries{#4}\fontshape{#5}
  \selectfont}
\fi\endgroup
\begin{picture}(3894,2545)(7780,-4214)
\put(9356,-1906){\makebox(0,0)[lb]{$S_{k'_\sigma}$}}
\put(11001,-2041){\makebox(0,0)[lb]{$e'_{\sigma}$}}
\put(7696,-2131){\makebox(0,0)[lb]{$e'_{\sigma}-1$}}
\put(10950,-3900){\makebox(0,0)[lb]{$e_{\sigma}-1$}}
\put(7900,-3811){\makebox(0,0)[lb]{$e_{\sigma}$}}
\put(9551,-2500){\makebox(0,0)[lb]{$I'_\sigma$}}
\put(9866,-3621){\makebox(0,0)[lb]{$I_\sigma$}}
\put(9356,-4156){\makebox(0,0)[lb]{$S_{k_\sigma}$}}
\put(9271,-3150){\makebox(0,0)[lb]{$z$}}
\put(8616,-2901){\makebox(0,0)[lb]{$\psi_z$}}
\put(10396,-2896){\makebox(0,0)[lb]{$\phi_\sigma$}}
\end{picture}
\caption{Totally real seam compatible with strip-like ends} \label{realseam}
\end{figure}

A picture of a quilt is shown in Figure \ref{quiltpants}.
Ends at the top resp.\ bottom of a picture will always indicate outgoing resp.\ incoming ends.  The alternative picture is that in which the ends are neighbourhoods of punctures and we indicate by arrows whether the ends are outgoing or incoming.  Here we draw the quilted surface as a disc, but in general this could be a more general surface.
Also, the complex structures on the patches are not necessarily induced by the embedding into the plane, as drawn.
Ignoring complex structures, the two views of the quilted surface in Figure \ref{quiltpants} are diffeomorphic.
In this example the end sequences are $(2,0),(1,0)$ and $(2,1)$ for the incoming ends (at the bottom), and $(2,2),(3,0),(2,3),(1,1)$ for the outgoing end (at the top).  

\begin{figure}[ht]
\begin{picture}(0,0)
\includegraphics{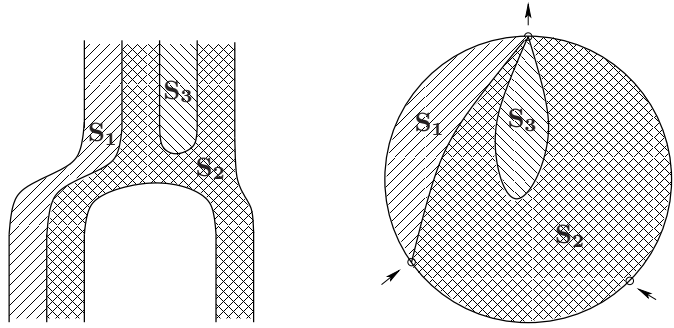}
\end{picture}
\setlength{\unitlength}{2200sp}
\begingroup\makeatletter\ifx\SetFigFont\undefined
\gdef\SetFigFont#1#2#3#4#5{
  \reset@font\fontsize{#1}{#2pt}
  \fontfamily{#3}\fontseries{#4}\fontshape{#5}
  \selectfont}
\fi\endgroup
\begin{picture}(7945,3850)(2889,-3449)
\put(3600,-1434){\makebox(0,0)[lb]{$\mathbf{S_1}$}}
\put(4816,-1821){\makebox(0,0)[lb]{$\mathbf{S_2}$}}
\put(4450,-953){\makebox(0,0)[lb]{$\mathbf{S_3}$}}
\put(7281,-1321){\makebox(0,0)[lb]{$\mathbf{S_1}$}}
\put(8866,-2581){\makebox(0,0)[lb]{$\mathbf{S_2}$}}
\put(8331,-1271){\makebox(0,0)[lb]{$\mathbf{S_3}$}}
\end{picture}
\caption{Two views of a quilted surface} \label{quiltpants}
\end{figure}

\begin{remark} \label{uniform nbhd}
\begin{enumerate}
\item
Every compact boundary component $S^1\cong I\subset\partial S_k$ of a surface with strip-like ends has a uniformizing neighbourhood $U_I\subset\partial S_k$ such that
$(U_I,j_k)\cong (S^1\times [0,\eps), i)$, where $S^1\cong\R/\Z$ and $i$ denotes the standard complex structure on $\R\times\R\cong\C$.
Due to the strip-like ends, the same holds for noncompact boundary components $\R\cong I\subset\partial S_k$, where $(U_I,j_k)\cong (\R\times [0,\eps), i)$.
\item
Seams that are compatible with the strip-like ends are automatically real analytic on these ends
since $\varphi_\sigma$ extends to $(s,\delta_{k,e}-t) \mapsto (s,t)$ resp.\ $(s,t) \mapsto (s,\delta_{k',e'}-t)$.
\item
In the uniformizing neighbourhoods of (a), the condition that a seam $\varphi_\sigma:I_\sigma\to I'_\sigma$ be real analytic is equivalent to $\varphi_\sigma:\R\to\R$ (resp.\ $\varphi_\sigma:S^1\to S^1$ in the compact case) being real analytic.
As a consequence, the extensions $\psi_z$ match up to an embedding
$\psi_\sigma: U_{I_\sigma}\to S_{k'_\sigma}$ with $\psi_\sigma^*j_{k'_\sigma}=-j_{k_\sigma}$
for some possibly smaller uniformizing neighbourhood $U_{I_\sigma}\cong\R\times[0,\eps)$. 
\end{enumerate}
\end{remark}

\begin{remark} \label{other quilt def}
Equivalently, we can define a quilted surface $\ul{S}$ by specifying a single surface with strip-like ends $\ti{S}$ and a finite number of connected, non-intersecting, real analytic submanifolds $\ti{I}_\sigma\subset\ti{S}$. (The real analytic condition means that every point $z\in\ti{I}_\sigma$ has a uniformizing neighbourhood $\ti{S}\supset\ti{U}\to\C$ that maps $\ti{I_\sigma}$ to $\R\subset\C$.)
We only need to impose a compatibility condition between the seams $\ti{I}_\sigma$ and the strip-like ends of $\ti{S}$: On every end $\ti{\eps}_e:\R^\pm\times[0,\ti{\delta}_e] \hookrightarrow \ti{S}$, the seams are given by straight lines, i.e.\
$\ti{\eps}_e^{-1}(\ti{I}_\sigma)$ consists of none, one, or two lines $\R^\pm\times\{t\}$ for pairwise disjoint $t\in[0,\ti{\delta}_e]$.

The patches $S_k$ of the corresponding quilted surface $\ul{S}$ are then determined by the closures of the connected components of $\ti{S}\setminus\bigcup_{\sigma\in\S}\ti{I}_\sigma$. The seams and seam maps are given by the submanifolds $\ti{I_\sigma}$ and their embedding into the two (not necessarily different) adjacent connected components.
The true boundary components of $\ul{S}$ are exactly the boundary components of $\ti{S}$, and the incoming and outgoing quilted ends $\E_-(\ul{S})\sqcup \E_+(\ul{S})$ are exactly the  incoming and outgoing ends of $\ti{S}$.
\end{remark}

Eventually, we will see that the relative quilt invariants only depend on the combinatorial structure of a quilted surface, and the specific choice of complex structure, strip-like ends, and seam maps is immaterial. (So it will suffice to define quilted surfaces by pictures as in Figure \ref{quiltpants}.)
The following lemma provides the homotopies for the proof in Section~\ref{sec:invariance}.

\begin{lemma} \label{homotopy of quilts}
Let $\ul{S}^0$ and $\ul{S}^1$ be two quilted surfaces of same combinatorial type, i.e.\ with the same patches $(S_k)_{k=1,\ldots, m}$ (as oriented $2$-dimensional manifolds with boundary), in- and outgoing marked points $(z_{k,e})$, and seams $\sigma\in\S$, 
but possibly different complex structures $j^i_k$, strip-like ends $\eps_{k,e}^i$, and seam maps $\varphi_\sigma^i$ for $i=0,1$.
Then there exists a smooth homotopy $(\ul{S}^\tau)_{\tau\in[0,1]}$ of quilted surfaces connecting the two.
It is smooth in the sense that $j^\tau_k$, $\eps_{k,e}^\tau$, and $\varphi_\sigma^\tau$ depend smoothly on~$\tau$.
\end{lemma}
\begin{proof}
Since the orientations and combinatorial compatibility with the ends are fixed, we can choose homotopies of diffeomorphisms $\varphi^\tau_\sigma$ connecting the given seam maps.
(This is since ${\rm Diff}_+(\R)=\{f\in\cC^\infty(\R,\R)\,|\, f'(x)>0 \;\forall x\in\R, \lim_{x\to\pm\infty}f(x)=\pm\infty\}$ is convex.)
We can moreover choose homotopies of embeddings $\eps^\tau_{k,e}|_{\R^\pm\times\{0\}}$ on one boundary component and homotopies $\delta^\tau_{k,e}>0$ of widths for each marked point. Then the compatibility condition for seams fixes embeddings $\eps^\tau_{k,e}|_{\R^\pm\times\{\delta^\tau_{k,e}\}}$ on all other boundary components near marked points, except for those that are true boundary components of $\ul{S}^i$ for which we can choose the homotopy freely.
Now, with fixed limits and boundary values, we can find a homotopy of embeddings $\eps^\tau_{k,e}$ for each marked point. This fixes the complex structures $j^\tau_k$ in neighbourhoods of each marked point. By construction, the seams are automatically real analytic in these neighbourhoods.

Next, recall that we have (anti-holomorphic) embeddings $\psi^i_\sigma:U^i_\sigma\to S_{k'_\sigma}$ extending the seam maps $\phi^i_\sigma:I_\sigma\to S_{k'_\sigma}$. We can pick the neighbourhoods to be the same $U^i_\sigma=U_\sigma$ and sufficiently small such that no two neighbourhoods or images thereof intersect each other.
Then we choose homotopies of embeddings  $\psi^\tau_\sigma:U_\sigma\to S_{k'_\sigma}$ extending the seam maps $\phi^\tau_\sigma$, of standard form
($(s,t)\mapsto (s,\delta-t)$ resp.\ $(s,\delta-t)\mapsto (s,t)$) on the strip-like ends, connecting the given $\psi^i_\sigma$, and preserving the empty intersections.
Now we can find homotopies of the complex structures $j^\tau_{k_\sigma}|_{U_\sigma}$ on each of the fixed seam neighbourhoods. This choice induces complex structures $j^\tau_{k'_\sigma}|_{\psi^\tau_\sigma(U_\sigma)}=-
(\psi_\sigma^\tau)_*j^\tau_{k_\sigma}$ on each image of a seam neighbourhood.
This construction ensures that the homotopy of seam maps becomes a homotopy of real analytic seams, compatible with the strip-like ends.
Finally, we can extend the complex structures to the complement of the strip-like ends and neighbourhoods of seams to find the full homotopy of complex structures $j^\tau_k$ on each patch $S_k$.
\end{proof}

Elliptic boundary value problems are associated to quilted surfaces
with strip-like ends as follows.  Suppose that
$\ul{E}=(E_k)_{k=1,\ldots,m}$ is a collection of complex vector
bundles over the patches $S_k$, and $\F$ is a collection
of totally real sub-bundles
$$ 
\F = ( F_\sigma )_{\sigma\in\S} \cup ( F_{b} )_{b\in\B} ,
\qquad
F_\sigma
\subset E_{k_\sigma}^-|_{I_\sigma} \times \phi_\sigma^* ( E_{k_\sigma'}|_{I'_\sigma}) , \quad
F_b \subset E_{k_b}|_{I_b}  .
$$
Here we write $E^-$ as short hand for the complex vector bundle with
reversed complex structure, $(E,J)^-=(E,-J)$. Let
$$
\Omega^0(\ul{S},\ul{E};\F) \subset {\textstyle\bigoplus_{k=1}^m} \Omega^0(S_k,E_k) 
$$
denote the subspace of collections of sections $\xi_k\in\Gamma(E_k)$
such that $(\xi_{k_\sigma},\xi_{k_\sigma'}\circ\varphi_\sigma)$ maps $I_\sigma$
to $F_{\sigma}$ for every seam $\sigma=\{(k_\sigma,I_\sigma),(k'_\sigma,I'_\sigma)\}\in\S$
and $\xi_{k_b}$ maps $I_b$ to $F_b$ for every boundary component $b=(k_b,I_b)\in\B$.

\begin{lemma}\label{quilt fred}
Suppose that, in a suitable trivialization of the bundles $E_k$ near each end $z_{k,e}$
the subbundles $F_{\sigma}$ and $F_b$ are constant and transverse
in the sense that for each quilted end $\ul{e}=(k_i,e_i)_{i=1,\ldots,n}\in\cE(\ul{S})$ we have
\beq \label{tranni}
F_{b_0}\times F_{\sigma_1}\times\ldots\times F_{\sigma_{n-1}}\times F_{b_n}
\pitchfork
\Delta_{E_{k_1}}\times\Delta_{E_{k_2}}\times\ldots\times\Delta_{E_{k_n}} .
\eeq
(Here $\sigma_i\in\S$ are the seams between the ends and $b_0,b_n\in\B$ are the remaining boundary components of the quilted end, resp.\ $\sigma_n$ is the additional seam in the case of a cyclic end, in which case we replace $F_{b_n}\times F_{b_0}$ above with $F_{\sigma_n}$.)
Then the direct sum of Cauchy-Riemann operators
$$
D_{\ul{E},\F} = \oplus_{k=1}^m D_{E_k} : \Omega^0(\ul{S},\ul{E};\F) \to
\Omega^{0,1}(\ul{S},\ul{E}) :=  {\textstyle\bigoplus_{k=1}^m} \Omega^{0,1}(S_k,E_k) 
$$
is Fredholm as a map between the $W^{1,p}$ and $L^p$ Sobolev completions for any $1<p<\infty$.
\end{lemma}
\begin{proof}
The Fredholm property follows from local estimates for $D_{\ul{E},\F}$ and its formal adjoint.
For sections $\xi_k\in\Gamma(E_k)$ compactly supported in the interior of $S_k$ we have
$\| \xi_k \|_{W^{1,p}(S_k)} \leq C ( \| D_{E_k} \xi_k \|_{L^{p}(S_k)} + \| \xi_k \|_{L^{p}(S_k)} )$
as in the standard theory, e.g.\ \cite{ms:jh}.
To obtain local estimates near a seam point $z\in I_\sigma$ we use the embedding
$\psi_z:S_{k_\sigma}\supset\U\to S_{k'_\sigma}$ and consider the section
$\eta:= (\xi_{k_\sigma}, \psi_z^*\xi_{k'_\sigma})$ of $E_{k_\sigma}^-|_\U \times \psi_z^*E_{k'_\sigma}$. It satisfies Lagrangian boundary conditions $\eta(\U\cap I_\sigma)\subset F_\sigma$ and, due to the reversal of complex structures in both $E_{k_\sigma}^-$ and by $\psi_z^*$, the operators
$D_{E_{k_\sigma}} \xi_{k_\sigma}$ and $D_{E_{k_\sigma'}} \xi_{k'_\sigma}$ combine to a Cauchy-Riemann operator on $\eta$ with respect to the complex structure $-j_{k_\sigma}$ on $\U$.
Hence, if $\xi_{k_\sigma}$ is supported in $\U$ and $\xi_{k'_\sigma}$ is supported in $\psi_z(\U)$, then we obtain an estimate
$$
\| \xi_{k_\sigma} \|_{W^{1,p}} + \| \xi_{k_\sigma'} \|_{W^{1,p}}  \leq
C \bigl( \| D_{E_{k_\sigma}} \xi_{k_\sigma} \|_{L^{p}} + \| D_{E_{k_\sigma'}} \xi_{k'_\sigma} \|_{L^{p}}
+ \| \xi_{k_\sigma} \|_{L^{p}} + \| \xi_{k_\sigma'} \|_{L^{p}} \bigr).
$$
Finally, given sections $(\xi_k\in\Gamma(E_k))_{k=1,\ldots,m}$ supported sufficiently close to the marked points in a quilted end $\ul{e}=(k_i,e_i)_{i=1,\ldots,n}$, we can trivialize the bundles
$E_{k_i}|_{\im\eps_{k_i,e_i}} \cong \R^\pm\times[0,\delta_{k_i,e_i}]\times (V_i,J_i)$
and boundary conditions $F_{\sigma_i}\subset V_i^-\times V_{i+1}$ resp.\
$F_{b_0}\subset V_{1}$, $F_{b_n}\subset V_{n}$.
Then we view the sections as tuple
$\ul{\xi}=\bigl( \xi_{k_i}|_{\im\eps_{k_i,e_i}} :\R\times[0,\delta_{k_i,e_i}]\to V_{i} \bigr)_{i=1,\ldots,n}$
of maps on strips, extended trivially to the complement of $\im\eps_{k_i,e_i}\subset\R\times[0,\delta_{k_i,e_i}]$.
The (trivialized) Cauchy-Riemann operator on these sections can now be rewritten
$D_{\ul{E},\F}(\xi_k)_{k=1,\ldots,m} = (\partial_s + \ul{D}) \ul{\xi}$, where
$\ul{D}=\oplus_{i=1}^n J_i\partial_t$ is a self-adjoint operator on
$\oplus_{i=1}^n L^2([0,\delta_{k_i,e_i}], V_i)$ whose domain is given by
$W^{1,2}$-maps $\ul{\gamma}=(\gamma_i)_{i=1,\ldots,n}$ that satisfy the matching conditions
$(\gamma_i(\delta_{k_i,e_i}),\gamma_{i+1}(0))\in F_{\sigma_i}$ resp.\
$(\gamma_n(\delta_{k_n,e_n}),\gamma_{0}(0))\in F_{b_n}\times F_{b_0}$.
The transversality assumption \eqref{tranni} implies that $\ul{D}$ is invertible and hence,
as in \cite{rs:spec}, $\partial_s + \ul{D}$ is a Banach space isomorphism, i.e.\ for $(\xi_k)$ supported near $\ul{e}$
as above
\begin{align*}
\| (\xi_k) \|_{W^{1,2}(\ul{S})} &:=
\sum_k \|\xi_k\|_{W^{1,2}(S_k)} = \|\ul{\xi}\|_{W^{1,2}} \\
&\leq C \|  (\partial_s + \ul{D}) \ul{\xi} \|_{L^2} =
C \sum_k \|D_{E_k}\xi_k\|_{L^{2}(S_k)}
=: C \| D_{\ul{E},\F} (\xi_k) \|_{L^{2}(\ul{S})} .
\end{align*}
This estimate generalizes to $W^{1,p}$ and $L^p$ by the same construction as in \cite[Lemma~2.4]{sa:lec}.
Putting the estimates together we obtain
$ \| (\xi_k) \|_{W^{1,p}(\ul{S}) } \leq C ( \| D_{\ul{E},\F} (\xi_k) \|_{L^{p}(\ul{S})} + \|K (\xi_k)\|_{L^p({\ul{S}})} $ for all $(\xi_k)\in\Omega^0(\ul{S},\ul{E};\F)$.
Here $K:W^{1,p}(\ul{S}) \to L^{p}(\ul{S})$ is the restriction operator to some compact subsets of the patches $S_k$ followed by the standard Sobolev inclusion, hence it is a compact operator.
This estimate shows that $D_{\ul{E},\F}$ has finite dimensional kernel and closed image.
The analogous estimate for the adjoint operator shows that $D_{\ul{E},\F}$ has
finite dimensional cokernel and hence is Fredholm.
\end{proof}

In the case that $\ul{S}$ has no strip-like ends, we define a topological index
$I(\ul{E},\F)$ as follows.  For each patch $S_k$ with boundary
we choose a complex trivialization of the bundle $E_k\cong S_k\times
\C^{r_k}$.  Each bundle
$F_{b} \subset  E_{k_b}|_{I_b} \cong I_{b} \times \C^{r_{k_b}} $
has a Maslov index $I(F_b)$ depending on the trivialization of $E_{k_b}$.
Similarly,
$$
F_\sigma  \subset E_{k_\sigma}^- |_{I_\sigma} \times
\varphi_\sigma^*( E_{k_\sigma'}|_{I_\sigma'}) \cong I_\sigma \times
(\C^{r_{k_\sigma}})^- \times \C^{r_{k'_\sigma}} 
$$
has a Maslov index $I(F_{\sigma})$
depending on the trivializations of $E_{k_\sigma}$ and
$E_{k_\sigma'}$. Let $S_0 \subset \ul{S}$ be the union of components
without boundary and define
$$ 
I(\ul{E},\F) := \deg(\ul{E} | S_0) + \sum_{\sigma\in\S}
I(F_{\sigma}) \;+ \sum_{b\in\B} I(F_b). 
$$
We leave it to the reader to check that the sum is
independent of the choice of trivializations.  Both the topological
and analytic index are invariant under deformation, and by deforming
the seam conditions $F_\sigma$ to those of split form
($\ti{F}_\sigma\times\ti{F}'_\sigma$ with $\ti{F}_\sigma\subset E^-_{k_\sigma}|_{I_\sigma}$ and
$\ti{F}'_\sigma\subset \phi_\sigma^*(E_{k'_\sigma}|_{I'_\sigma}$)
one obtains from
\eqref{indthm} an index formula
$$ 
\Ind(D_{\ul{E},\F}) = \sum_i \rank_\C(E_i) \chi(S_i) +
I(\ul{E},\F) .
$$

We now construct moduli spaces of pseudoholomorphic quilted surfaces.
Let
$$ 
\ul{M} = \bigl( (M_k,\omega_k) \bigr)_{k =1,\ldots,m} 
$$
be a collection of symplectic manifolds.  A {\em Lagrangian boundary
condition} for $(\ul{S},\ul{M})$ is a collection
$$ 
\cL = ( L_{\sigma} )_{\sigma\in\S} \cup ( L_{b} \subset M_{k_b} )_{b\in\B},
\qquad
L_{\sigma} \subset M_{k_\sigma}^- \times M_{k_\sigma'}, \quad L_{b} \subset M_{k_b}
$$
of Lagrangian correspondences and Lagrangian submanifolds associated
to the seams and boundary components of the quilted surface. We will
indicate the domains $M_k$, the seam conditions $L_\sigma$, and the true boundary
conditions $L_b$ by marking the surfaces, seams and boundaries
of the quilted surfaces as in figure \ref{bc quilt}.

\begin{figure}[ht]
\begin{picture}(0,0)
\includegraphics{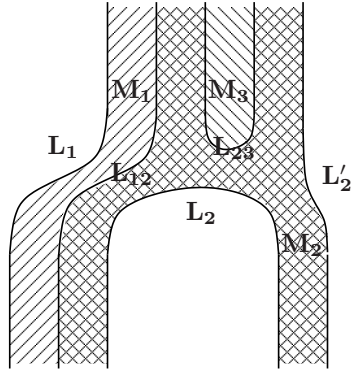}
\end{picture}
\setlength{\unitlength}{2693sp}
\begingroup\makeatletter\ifx\SetFigFont\undefined
\gdef\SetFigFont#1#2#3#4#5{
  \reset@font\fontsize{#1}{#2pt}
  \fontfamily{#3}\fontseries{#4}\fontshape{#5}
  \selectfont}
\fi\endgroup
\begin{picture}(2969,3419)(2819,-3458)
\put(5176,-2401){\makebox(0,0)[lb]{$\mathbf{M_2}$}}
\put(3611,-981){\makebox(0,0)[lb]{$\mathbf{M_1}$}}
\put(3026,-1501){\makebox(0,0)[lb]{$\mathbf{L_1}$}}
\put(3601,-1751){\makebox(0,0)[lb]{$\mathbf{L_{12}}$}}
\put(4300,-2106){\makebox(0,0)[lb]{$\mathbf{L_2}$}}
\put(4501,-981){\makebox(0,0)[lb]{$\mathbf{M_3}$}}
\put(5551,-1801){\makebox(0,0)[lb]{$\mathbf{L_2'}$}}
\put(4546,-1506){\makebox(0,0)[lb]{$\mathbf{L_{23}}$}}
\end{picture}
\caption{Lagrangian boundary conditions for a quilt}
\label{bc quilt}
\end{figure}

To each quilted end $\ul{e}_\pm=(k_i,e_i)_{i=1,\ldots,n}\in\E_\pm(\ul{S})$
we associate the sequence of Lagrangian correspondences of length $n=n_{\ul{e}}$
that label the seams and boundary components which run into the end.
For an incoming quilted end $\ul{e}_-$ the seams between the ends of the patches are
$\phi_{\sigma_i}: I_{\sigma_i}=I_{k_i,e_i-1}\overset{\sim}{\rightarrow}I_{\sigma_i}'=I_{k_{i+1},e_{i+1}}$
for $i=1,\ldots, n-1$, whereas for an outgoing quilted end $\ul{e}_+$ the seams are
$\phi_{\sigma_i}: I_{\sigma_i}=I_{k_i,e_i}\overset{\sim}{\rightarrow}I_{\sigma_i}'=I_{k_{i+1},e_{i+1}-1}$.
In both cases $\eps_{k_{i+1},e_{i+1}}^{-1} \circ \phi_{\sigma_i} \circ \eps_{k_i,e_i}$ is the map
$(s,\delta_{k_i,e_i}) \mapsto (s,0)$ for $s\in\R^\pm$.
The end could be cyclic, i.e.\ with an additional seam
$\phi_{\sigma_n}: I_{k_n,e_n-1}\overset{\sim}{\rightarrow}I_{k_{1},e_{1}}$
resp.\ $\phi_{\sigma_n}: I_{k_n,e_n}\overset{\sim}{\rightarrow}I_{k_{1},e_{1}-1}$.
In that case the generalized
Lagrangian correspondence associated to the end $\ul{e}$ is
the cyclic correspondence
\beq \label{end sequence c}
\ul{L}_{\ul{e}} :=\bigl( L_{\sigma_1}, \ldots, L_{\sigma_{n-1}},L_{\sigma_n}) =
\bigl( \;
{\longrightarrow} M_{k_1}
\overset{L_{\sigma_1}}{\longrightarrow} M_{k_2} \;\ldots
\overset{L_{\sigma_{n-1}}}{\longrightarrow} M_{k_n}
\overset{L_{\sigma_n}}{\longrightarrow}  \bigr) .
\eeq
If the end is noncyclic then the boundary components
$I_{k_1,e_1}=I_{b_0}$ and  $I_{k_n,e_n-1}=I_{b_n}$ in the incoming case,
resp.\  $I_{k_1,e_1-1}=I_{b_0}$ and  $I_{k_n,e_n}=I_{b_n}$ in the outgoing case,
are true boundary components of the quilted surface $\ul{S}$.
In that case the generalized Lagrangian correspondence associated to the end $\ul{e}$ is the correspondence
\beq \label{end sequence}
\ul{L}_{\ul{e}} :=(L_{b_0}, L_{\sigma_1}, \ldots, L_{\sigma_{n-1}},L_{b_n}) =
\bigl(
\{pt\} \overset{L_{b_0}}{\longrightarrow} M_{k_1}
\overset{L_{\sigma_1}}{\longrightarrow} M_{k_2} \;\ldots
\overset{L_{\sigma_{n-1}}}{\longrightarrow} M_{k_n}
\overset{L_{b_n}}{\longrightarrow} \{pt\} \bigr) .
\eeq
Equivalently, for a noncyclic end, the correspondence $L_{\sigma_n}$ in \eqref{end sequence c} is replaced by the split type Lagrangian correspondence $L_{b_n}\times L_{b_0}\subset M_{k_n}^-\times M_{k_1}$.

\begin{definition}  \label{corr monotone}
We say that the tuple of Lagrangian boundary and seam conditions $\L$ for $\ul{S}$ is
{\em monotone} if the sequences $\ul{L}_{\ul{e}}$ in \eqref{end sequence c} resp.\ \eqref{end sequence} 
for each end $\ul{e}\in\E(\ul{S})$ are monotone for Floer theory (in the sense of \cite{quiltfloer}) with a fixed constant $\tau\geq 0$, and the following holds: 
Let $\ul{S}\#\ul{S}^-$ denote the quilted surface obtained by gluing a
copy $\ul{S}^-$ of $\ul{S}$ with reversed complex structure (and hence
reversed ends) to $\ul{S}$ at all corresponding ends.  This quilted
surface has components $(S_k\#S_k^-)_{k=1,\ldots,m}$,
compact seams $\phi_\sigma: I_\sigma\overset{\sim}{\rightarrow}I_\sigma'\cong S^1$ for $\sigma\in\ti{\S}$ (one for each noncompact and two for each compact seam of $\ul{S}$),
compact boundary components $I_b\cong S^1$ for $b\in\ti{\B}$
(one for each noncompact and two for each compact boundary component of $\ul{S}$),
but no strip-like ends.
Then for each tuple of
maps $\ul{u}: \ul{S}\#\ul{S}^- \to \ul{M}$ (that is $u_k:S_k^-\#S_k\to M_k$)
that takes values in $\cL$ over the seams and boundary
components (that is $(u_{k_\sigma}\times u_{k_\sigma'}\circ\phi_\sigma)(I_\sigma)\subset
L_{\sigma}$ and $u_k(I_b)\subset L_{b}$) we have the action-index relation
$$ 
2\sum_{k=1}^m \int u_k^*\omega_k = \tau\cdot I(\ul{E}_u,\F_u) 
$$
for $\ul{E}_u=(u_k^*TM_k)_{k=1,\ldots,m} $ and
$\F_u=\bigl( (u_{k_\sigma}\times u_{k_\sigma'}\circ\phi_\sigma)^*TL_{\sigma} \bigr)_{\sigma\in\ti{\S}} \cup ( u_{k_b}^*TL_b)_{b\in\ti{\B}} $.
\end{definition}

\begin{remark}
Note the following analogue of the monotonicity Lemma in \cite{quiltfloer}:
If each $M_k$ is monotone in the sense of (M1) and each
$L_{\sigma}$ and $L_{b}$ is monotone in the sense of (L1),
all with the same constant $\tau\geq 0$, and each $\pi_1(L_{\sigma})\to
\pi_1(M_{k_\sigma}^- \times M_{k_\sigma'})$ and
$\pi_1(L_{b})\to\pi_1(M_{k_b})$ has torsion image,\footnote{
In the case of disconnected symplectic or Lagrangian manifolds, this is to be interpreted as separate assumption for each connected component of $L_b$ and the corresponding component of $M_{k_b}$.
}
then $\cL$ is monotone.
This is because any tuple $\ul{u}$ in Definition \ref{corr monotone} is homotopic to the union
of a disc $(D,\partial D)\to (M_{k_\sigma}^-\times M_{k_\sigma'}, L_\sigma)$ at each seam,
a disc $(D,\partial D)\to (M_{k_b}, L_b)$ at each true boundary component of the quilt,
and a closed curve $\Sigma_k\to M_k$ for each patch 
(where $\Sigma_k$ is obtained from $S_k\#S_k^-$ by collapsing each boundary circle to a point).
\end{remark}

Associated to each quilted end $\ul{e}\in\E(\ul{S})$ we moreover have the widths $\ul{\delta}_{\ul{e}}=(\delta_{k_i,e_i})_{i=1,\ldots,n_{\ul{e}}}$ of the strip-like ends that constitute $\ul{e}$.
Given $\ul{\delta}_{\ul{e}}$ and assuming monotonicity as above, we can now as in \cite{quiltfloer,quiltconst}
fix regular Hamiltonian perturbations $\ul{H}_{\ul{e}}=(H_{k_i,e_i})_{i=1,\ldots,n_{\ul{e}}}$ and almost complex structures $\ul{J}_{\ul{e}}=(J_{k_i,e_i})_{i=1,\ldots,n_{\ul{e}}}$ such that the quilted Floer cohomology $HF(\ul{L}_{\ul{e}})$ is well defined.
In particular, the generalized intersection points of the correspondence are a finite set for each end $e\in\cE(\ul{S})$,
$$ 
\cI(\ul{L}_{\ul{e}}) =  L_{b_0} \times_{\phi_1} L_{\sigma_1}\times_{\phi_2} \ldots
L_{\sigma_{n-1}} \times_{\phi_n} L_{b_n}
\subset \prod_{i = 1}^{n} M_{k_{i}} 
$$
resp.\
$$ 
\cI(\ul{L}_{\ul{e}}) = \times_{\phi_1} \bigl( L_{\sigma_1}\times_{\phi_2} \ldots
L_{\sigma_{n-1}} \times_{\phi_n} L_{\sigma_n} \bigr)
\subset \prod_{i = 1}^{n} M_{k_{i}} . 
$$
Here $\phi_i$ denotes the time $\delta_{k_i,e_i}$ flow of the Hamiltonian $H_{k_i,e_i}$ on
$M_{k_i}$.  Next, let $\Ham(\ul{S},(\ul{H}_{\ul{e}})_{\ul{e}\in\E})$ denote the set of tuples
$$ 
\ul{K} = \bigl( K_k \in \Omega^1(S_k,C^\infty(M_k)) \bigr)_{k=1,\ldots,m} 
$$
such that $K_k|_{\pd S_k}=0$ and on each end $\eps_{k,e}^*K_k =
H_{k,e} \d t$.  We denote the corresponding Hamiltonian vector field
valued one-forms by $\ul{Y} \in \Omega^1(\ul{S},\Vect(\ul{M}))$. These
satisfy $\eps_{k,e}^* Y_k = X_{H_{k,e}} \d t$ on each strip-like end.
Finally, let $\J(\ul{S},(\ul{J}_{\ul{e}})_{\ul{e}\in\E})$ denote the set of collections
$$ 
\ul{J} = \bigl( J_k \in \Map(S_k,\J(M_k,\omega_k)) \bigr)_{k =1,\dots,m } 
$$
agreeing with the chosen almost complex structures on the ends.  Now
we denote by $\ul{\cI}_-(\ul{S},\cL)$ resp.\ $\ul{\cI}_+(\ul{S},\cL)$ the set of tuples
$X^\pm=(\ul{x}^\pm_{\ul{e}})_{\ul{e}\in \E_\pm(\ul{S})}$ consisting of
one intersection tuple $\ul{x}^\pm_{\ul{e}} =
(x^\pm_{k_i,e_i})_{i=1,\ldots,n_{\ul{e}}} \in\cI(\ul{L}_{\ul{e}})$ for
each incoming resp.\ outgoing end $\ul{e}$.  For each pair
$(X^-,X^+)\in\ul{\cI}_-(\ul{S},\cL)\times\ul{\cI}_+(\ul{S},\cL)$
we introduce a moduli space of {\em pseudoholomorphic quilts}:
$$
\M_{\ul{S}}(X^-,X^+) := \bigl\{ \ul{u}=\bigl( u_k :S_k \to M_k \bigr)_{k=1,\ldots,m} \,\big|\, (a)-(d) \bigr\}
$$
is the space of collections of $(\ul{J},\ul{K})$-holomorphic maps with
Lagrangian boundary and seam conditions, finite energy, and fixed limits, that is
\ben
\item
$\overline\partial_{J_k,K_k}u_k := J_k(u_k) \circ (\d u_k - Y_k(u_k)) -
(\d u_k - Y_k(u_k)) \circ j_k = 0 \quad$ for $k=1,\ldots,m$,
\item
$ (u_{k_\sigma},u_{k_\sigma'}\circ\varphi_{\sigma}) (I_{\sigma}) \subset L_{\sigma} \quad$
for all $\sigma \in \S$ and $ \quad u_{k}(I_b)\subset L_b \quad$ for all $b \in \B$.
\item $E_{\ul{K}}(\ul{u}) = \sum_{k=1}^m E_{K_k}(u_k) <\infty$,
\item $\lim_{s \to \pm \infty} u_{k_i}(\eps_{k_i,e_i}(s,t)) = x^\pm_{k_i,e_i}(t) \quad$
for all $\ul{e}=(k_i,e_i)_{i=1,\ldots,n_{\ul{e}}}\in \E_\pm(\ul{S})$.
\een

\begin{remark} \label{rmk:monotone quilt}
\begin{enumerate} \renewcommand\theenumi{\arabic{enumi}}
\item
To show that (a)--(d) defines a Fredholm problem we proceed as usual (see e.g.\ \cite{ms:jh}):
Construct a Banach manifold $\B$ of tuples of maps satisfying (b)--(d) and show that the tuples satisfying (a) are the zeros of a Fredholm section $\overline\partial_{\ul{J},\ul{K}}$ of the bundle $\cE\to\B$ of tuples of $(0,1)$-forms.
As in Remark~\ref{monotone3} we also define the linearized operator $D_{\ul{u}}=d\sigma_{\ul{u}}(0)$ at a not necessarily holomorphic $\ul{u}\in\B$ from a map $\sigma_{\ul{u}}:T_{\ul{u}}\B \to \cE_{\ul{u}}$ given by $\sigma_{\ul{u}}({\ul{\xi}})=\Phi_{\ul{u}}({\ul{\xi}})^{-1}
\bigl(\overline\partial_{\ul{J},\ul{K}}(\exp_{\ul{u}}(\ul{\xi}))\bigr)$ and suitable quadratic corrections.
To be more precise, the tangent space $T_{\ul{u}}\B$ is a Sobolev completion of the space of tuples of sections $\ul{\xi}$ of the complex bundles $\ul{E}_u=(u_k^*TM_k)_{k=1,\ldots,m}$ satisfying boundary and seam conditions in
the totally real subbundles $\F_u=\bigl( (u_{k_\sigma}\times u_{k_\sigma'}\circ\phi_\sigma)^*TL_{\sigma} \bigr)_{\sigma\in\S} \cup ( u_{k_b}^*TL_b)_{b\in\B}$.
The component of $\sigma_{\ul{u}}({\ul{\xi}})$ on the patch $S_k$ is the $(0,1)$-form
$\Phi_{u_k}(\ul{\xi})^{-1} (\overline\partial_{J_k,K_k}(\exp^0_{u_k}(\xi_k + Q_k(\ul{\xi})))$,
where $Q_k(\ul{\xi})\in u_k^*TM_k$ is a quadratic correction,
$\exp^0$ is the standard exponential map using a fixed metric on $M_k$,
and $\Phi_{u_k}(\ul{\xi})$ is parallel transport with the usual complex linear connection on $T^*M_k$ along the paths
$s\mapsto \exp^0_u(s\xi_k + Q_k(s\ul{\xi}))$.
The quadratic correction $Q_k(\ul{\xi})$ vanishes on the complement of uniformizing neighbourhoods of $\partial S_k$.
Near a true boundary component $I_b\subset\partial S_k$ the correction actually only depends on $\xi_k$ and is given exactly as in Remark~\ref{monotone3} by interpolating between the fixed metric on $M_k$ and a metric in which $L_b\subset M_k$ is totally geodesic.
Near a seam $I_\sigma\subset\partial S_k$ that connects the patches $S_k$ and $S_{k'}$ we use the extended seam map $\psi_\sigma:\U_{I_\sigma} \to S_{k'}$ of Remark~\ref{uniform nbhd} to pair $(u_k,u_{k'})$ and $(\xi_k,\xi_{k'})$ to
$u_\sigma:\U_{I_\sigma}\to M_k^-\times M_{k'}$ and $\xi_\sigma\in u_\sigma^* T(M_k^-\times M_{k'})$.
Now the quadratic corrections $(Q_k(\ul{\xi}),Q_{k'}(\ul{\xi}))$, viewed as one section $Q_\sigma(\xi_k,\xi_{k'})\in u_\sigma^* T(M_k^-\times M_{k'})$ are given by interpolating the fixed product of metrics on $M_k$ and $M_{k'}$ to a metric on $M_k^-\times M_{k'}$ in which $L_\sigma$ is totally geodesic.

Thus the $S_k$-component of $\sigma_{\ul{u}}({\ul{\xi}})$ in a neighbourhood of the seam $I_\sigma$ actually also depends on the section $\xi_{k'}$ via the quadratic correction. However, since $Q_k(0)=0$ and $dQ_k(0)=0$ as in Remark~\ref{monotone3}, the linearized operator constructed in this way is independent of the choice of quadratic correction and takes the same form as in \cite{ms:jh} on each patch.
So it takes the form $D_{\ul{u}}=D_{\ul{E}_{\ul{u}},\cF_{\ul{u}}}$ of a direct sum of real Cauchy-Riemann operators as discussed in the linear theory above.
\item
If $\cL$ is monotone, then elements $\ul{u}\in\M_{\ul{S}}(X^-,X^+)$ (and more generally tuples of maps as in (1) above) satisfy an area-index relation as in Remark \ref{monotone3}~2),
\beq \label{q-ar-in}
A(\ul{u}) := {\textstyle \sum_{k=1\ldots m}} \tint_{S_k}  u_k^* \omega_k = \tfrac12 \tau \cdot\Ind(D_{\ul{u}}) + \tfrac12 c(X^-,X^+) .
\eeq
This is proven as before, by gluing a fixed element
$\ul{v}_0\in\M_{\ul{S}}(X^-,X^+)$ to $\ul{u}$ to obtain a quilt map $\ul{w}=\ul{u}\#\ul{v}_0$
defined on $\ul{S}\#\ul{S}'$, which satisfies the monotonicity relation in Definition~\ref{corr monotone}.
\item
As Remark~\ref{monotone3}~3) we have an energy identity for elements ${\ul{u}\in\M_{\ul{S}}(X^-,X^+)}$, 
$$
E_{\ul{K}}(\ul{u}) =  \sum_{k=1\ldots m} \int_{S_k}   u_k^* \omega_k   + ( R_k\circ u_k  ) \, {\rm dvol}_{S_k} 
$$
with curvature terms $R_k {\rm dvol}_{S_k} = - \d K_k + \tfrac 12 \{ K_k , K_k\} \in \Omega^2(S_k, C^\infty(M_k))$ that are compactly supported on each $S_k$ and uniformly bounded as maps $R_k: S_k\times M_k\to\R$.
Combining this with the monotonicity relation \eqref{clmm} we obtain the energy-index relation
\beq \label{q-en-in}
 E_{\ul{K}}(\ul{u}) -  \sum_{k=1\ldots m} \int_{S_k}  ( R_k\circ u_k  ) \, {\rm dvol}_{S_k}  = \tfrac12 \tau \cdot {\rm Ind}(D_{\ul{u}}) + \tfrac12 c(X^-,X^+) ,
\eeq
which gives an upper bound on the energy of solutions $\ul{u}$ for fixed limits $X^\pm$ and index. 
\item
If all Lagrangians in $\cL$ are oriented, then for any tuple $\ul{u}$ as in (1) the index $\Ind(D_{\ul{u}})$ is determined $\text{mod}\;2$ by $\ul{S}$ and the limit conditions $X^-,X^+$. This also follows as in Remark \ref{monotone3} from the index identity
$
{\rm Ind}(D_{\ul{u}}) +  {\rm Ind}(D_{\ul{v}_0}) =  {\rm Ind}(D_{\ul{w}})
=  I(\ul{E}_{\ul{w}},\cF_{\ul{w}}) + \sum_{k=1}^m \tfrac{\dim M_k}2 \chi(S_k\#S_k^-) ,
$
where the topological index $ I(\ul{E}_{\ul{w}},\cF_{\ul{w}})$ is even since each of the totally real subbundles in $\cF_{\ul{w}}$ has an orientation induced from the orientations of the seam and boundary conditions $L_\sigma$ and $L_b$.
\end{enumerate}
\end{remark}

\begin{theorem}  \label{quilttraj}
Suppose that each $M_k$ satisfies (M1-2),
each Lagrangian in $\cL$ satisfies (L1-2), and $\cL$ is monotone in the sense of Definition~\ref{corr monotone}.
Moreover, for any end $\ul{e}\in\E(\ul{S})$ fix regular perturbation data
$(\ul{H}_{\ul{e}},\ul{J}_{\ul{e}})$.
Then for any $\ul{J} \in \J(\ul{S},(\ul{J}_{\ul{e}})_{\ul{e}\in\E})$
there exists a comeagre (in particular dense) subset
$ \Ham^\reg(\ul{S},(\ul{H}_{\ul{e}})_{\ul{e}\in\E},\ul{J}) \subset
 \Ham(\ul{S},(\ul{H}_{\ul{e}})_{\ul{e}\in\E})$
such that for
any $\ul{H}\in \Ham^\reg(\ul{S},(\ul{H}_{\ul{e}})_{\ul{e}\in\E},\ul{J})$
the following holds for all $X^\pm\in\ul{\cI}_\pm(\ul{S},\cL)$.
\ben
\item $\M_{\ul{S}}(X^-,X^+)$ is a smooth manifold
whose dimension near a solution $\ul{u}$ is given by the formal dimension $\Ind(D_{\ul{u}})$.
\item The zero dimensional component $\M_{\ul{S}}(X^-,X^+)_0$ is finite.
\item The one-dimensional component $\M_{\ul{S}}(X^-,X^+)_1$ has a
  compactification as a one-manifold with boundary
\begin{align*}
\partial \overline{\M_{\ul{S}}(X^-,X^+)_1} &\cong
\bigcup_{\ul{e} \in \E_- ,\ul{y}\in\cI(\ul{L}_{\ul{e}})}
\M(\ul{x}^-_{\ul{e}},\ul{y})_0\times\M_{\ul{S}}(X^-|_{\ul{x}^-_{\ul{e}} \to \ul{y}}, X^+)_0 \\
&\quad \cup \bigcup_{\ul{e} \in \E_+, \ul{y}\in\cI(\ul{L}_{\ul{e}})}
\M_{\ul{S}}(X^-,X^+|_{\ul{x}^+_{\ul{e}} \to \ul{y}})_0 \times \M(\ul{y},\ul{x}^+_{\ul{e}})_0 .
\end{align*}
Here the multi-tuple $X|_{\ul{x}_{\ul{e}} \to \ul{y}}$ is $X$ with the tuple $\ul{x}_{\ul{e}}$ replaced by $\ul{y}$, and $\M(\ul{x}^-,\ul{x}^+)$ for $\ul{x}^\pm\in\cI(\ul{L}_{\ul{e}})$
are the moduli spaces of quilted Floer trajectories  defined in \cite{quiltfloer} with the perturbation data $\ul{\delta}_{\ul{e}},\ul{H}_{\ul{e}}, \ul{J}_{\ul{e}}$.
\een
\end{theorem}

\begin{proof}
The proof requires no new or nonstandard techniques except for the observation that a
pseudoholomorphic quilt with seam condition in a Lagrangian correspondence is locally equivalent to a pseudoholomorphic curve with boundary condition in a Lagrangian.
Formally, consider any point $z\in I_\sigma$ on a seam $\sigma\in\S$. Since the seam is assumed to be real analytic we have an embedding $\psi:\U\to S_{k'_\sigma}$ of a neighbourhood $\U\subset S_{k_\sigma}$ of $z$ that restricts to the seam map $\psi|_{\U\cap\partial S_k}=\varphi_\sigma|_{\U\cap I_\sigma}: \U\cap I_\sigma \to I_\sigma'$.
Given any tuple of maps $\ul{u}=( u_k :S_k \to M_k )_{k=1,\ldots,m}$ consider the map
$w:=(u_{k_\sigma}|_\U , u_{k_\sigma'}\circ\psi): \U \to M^-_{k_\sigma}\times M_{k_\sigma'}$.
The perturbed holomorphic equation in (a) for $k_\sigma$ on $\U$ and for $k_\sigma'$ on $\psi(\U)$ is equivalent to the perturbed holomorphic equation
$J_\sigma(w) \circ (\d w - Y_\sigma(w)) = (\d w - Y_\sigma(w)) \circ (- j_{k_\sigma})$ on $\U$.
Here the almost complex structure
$J_\sigma:=(- J_{k_\sigma}) \oplus J_{k_\sigma'}\circ\psi^{-1}
\in \cC^\infty(\U, \J(M_{k_\sigma}\times M_{k_\sigma'}, \omega_{\sigma}))$
is compatible with the symplectic form $\omega_\sigma:=(-\omega_{k_\sigma})\oplus \omega_{k'_\sigma}$, and the perturbation
$Y_\sigma:= Y_{k_\sigma}|_\U \oplus \psi^*Y_{k'_\sigma}\in\Omega^1(\U,{\rm Vect}(M_{k_\sigma}\times M_{k'_\sigma}))$, which corresponds to the Hamiltonian-valued one-form
$K_\sigma:= K_{k_\sigma}|_\U + \psi^*K_{k'_\sigma}\in\Omega^1(\U,\cC^\infty(M_{k_\sigma}\times M_{k'_\sigma}))$.
The seam condition in (b) for $\sigma$ on $\U\cap I_\sigma$ is equivalent to
$w(\U\cap I_\sigma)\subset L_\sigma$, and the contribution to the energy in (c)
on $\U$ and $\psi(\U)$ is
$$
\int_\U 
\tfrac 12 \bigl| \d u_{k_\sigma} - Y_{S_{k_\sigma}}(u_{k_\sigma}) \bigr|^2  {\rm dvol}_{S_{k_\sigma}} 
+ \int_{\psi(\U)}
\tfrac 12 \bigl| \d u_{k'_\sigma} - Y_{S_{k'_\sigma}}(u_{k'_\sigma}) \bigr|^2  {\rm dvol}_{S_{k'_\sigma}}
=
\int_\U 
\tfrac 12 \bigl| \d w - Y_{\sigma}(w) \bigr|^2  {\rm dvol}_{S_{k_\sigma}}
$$
On the quilted ends we have exponential decay by the same arguments as in
\cite{quiltfloer}: The quilted end can be seen as one simple strip-like end with values in a product of manifolds and a deformed almost complex structure (where the deformation only stems from the different widths of strips).
Here we require just the finiteness of energy. For the compactness of moduli spaces we will need the energy actually bounded, which follows from the energy-index relation \eqref{q-en-in} when fixing the index and taking into account the uniform bounds on the curvature terms.
The proof of compactness follows the strategy in Lemma~\ref{quilt fred}.  We have local energy bounds in the interior of patches, in a neighbourhood of any boundary point, and in the above ``folded neighbourhoods'' around seams. The standard estimates (e.g.\ as in \cite{ms:jh}) for holomorphic curves in these neighbourhoods combine to prove compactness if we can exclude bubbling, i.e.\ prove uniform gradient bounds. In the case of noncompact symplectic manifolds we moreover need to establish a $\cC^0$-bound on the curves.  The latter follows from uniform gradient bounds together with the compactness of the Lagrangian boundary and seam conditions, since we do not allow patches without boundary.  Bubbling effects can be analyzed locally, either in the interior of a single patch or in a seam neighbourhood as above. Local rescaling near points in a seam neighbourhood where the gradient in a sequence $\ul{u}^\nu$ blows up leads to spheres in $M_{k_\sigma}^-\times M_{k'_\sigma}$, or to discs with boundary on $L_\sigma$.\footnote{
In the case of noncompact symplectic manifolds, this is the one point that requires additional attention. If the manifold has convex ends,  then the maximum principle guarantees that all relevant holomorphic curves are contained in a fixed compact set, hence we obtain convergence to a sphere or disk. More generally, one could argue with monotonicity and an energy quantization argument as in \cite{we:en}, which only requires uniform bounds on the second derivatives of the almost complex structures.
}  
These are ruled out, for the zero and one-dimensional moduli spaces, by the monotonicity of $\cL$ and the observation that symplectic area of bubbles is positive, as follows: Additivity of symplectic area implies a strict inequality $A(\ul{u}^\infty)<A(\ul{u}^\nu)$ between the sequence and its $\cC^\infty_{\rm loc}$ limit $\ul{u}^\infty$, which by the area-index relation \eqref{q-ar-in} is translated into a strict index inequality ${\rm Ind}(D_{\ul{u}^\infty})<{\rm Ind}(D_{\ul{u}^\nu})$.
This implies ${\rm Ind}(D_{\ul{u}^\infty})\le -1$ since we assumed ${\rm Ind}(D_{\ul{u}^\nu})\le 1$, and orientation (L2) of the Lagrangians ensures that quilts with the same limits $X^\pm$ have Fredholm, resp.\ Maslov indices of the same parity.
So the limit solution $\ul{u}^\infty$, after removal of singularities at the bubbling points, would have negative index. However, assuming transversality, this moduli space is empty.

To achieve transversality we have to work with Hamiltonian
perturbations rather than variations of the almost complex structure
on each patch of the quilt, since we cannot a priori exclude constant
patches.  The Hamiltonian perturbations are fixed (and only
$t$-dependent) on the strip-like ends, but can vary freely on the rest
of each patch. By unique continuation on the ends, each element in the
cokernel must be nonzero somewhere in the complement of the strip-like
ends. This can be excluded by an appropriate variation of the
Hamiltonian perturbation on this complement, given by the Sard-Smale
theorem from a transverse universal moduli space.

The gluing construction, which shows the reverse inclusion in (c), is
local near each quilted end in the sense that it uses only the
non-degeneracy of operator on the neck and the regularity of the
linearized operators for the Floer trajectories, which hold by the
choice of the perturbation data.
\end{proof}

Associated to the data $(\ul{S},\ul{M}, \cL)$ as in Theorem \ref{quilttraj}
we construct a relative invariant $\Phi_{\ul{S}}$ as follows. Define
$$ 
C\Phi_{\ul{S}} :\bigotimes_{\ul{e} \in \E_-(\ul{S})} CF(\ul{L}_{\ul{e}})
\to  \bigotimes_{\ul{e} \in \E_+(\ul{S})} CF(\ul{L}_{\ul{e}})
$$
by
\begin{equation*} \label{defrel}
 C\Phi_{\ul{S}} \biggl( \bigotimes_{\ul{e} \in \E_-} \bra{\ul{x}^-_{\ul{e}}} \biggr)
 := \sum_{X^+\in\ul{\cI}_+(\ul{S},\cL)}
\#_2 \M_{\ul{S}}(X^-,X^+)_0  \cdot \bigotimes_{\ul{e} \in \E_+} \bra{\ul{x}^+_{\ul{e}}} ,
\end{equation*}
where $\#_2 \M_{\ul{S}}(X^-,X^+)_0\in\Z_2$ is the number, modulo $2$, of points in the $0$-dimensional component of the moduli space $\M_{\ul{S}}(X^-,X^+)$. 
By Theorem \ref{quilttraj}~(c), the maps $C\Phi_{\ul{S}}$
are chain maps and so descend to a map of Floer cohomologies
\begin{equation} \label{relinv}
 \Phi_{\ul{S}} :  \bigotimes_{\ul{e} \in \E_-(\ul{S})} HF(\ul{L}_{\ul{e}})
\to  \bigotimes_{\ul{e} \in \E_+(\ul{S})} HF(\ul{L}_{\ul{e}}) .
\end{equation}
Here we assume in addition that all Lagrangians in $\cL$ satisfy (L3) and
hence the Floer cohomologies are well defined.
In Section~\ref{sec:invariance} below we will see that the relative invariants $\Phi_{\ul{S}}$ are indeed invariants of the symplectic data $(\ul{M},\cL)$ and the combinatorial data of the quilted surface $\ul{S}$, i.e. they are independent of the choices of perturbation data $(\ul{K},\ul{J})$, complex structures
$\ul{j}=(j_{k})_{k=1,\ldots,m}$, strip-like ends on each $S_k$,
and seam maps $\ul{\phi}=(\phi_{\sigma})_{\sigma\in\S}$.

\begin{remark} \label{qcyl}   As in Remark \ref{cyl}, one can also allow
those component surfaces labeled by compact symplectic manifolds
to have incoming and outgoing cylindrical ends.
In this case, the relative invariants \eqref{relinv} have additional
factors of cylindrical Floer cohomologies $HF(\Id_{M_k})$ on either
side.
\end{remark}

\begin{remark} \label{qorient}
We say that $\cL$ is {\em relatively spin} if all Lagrangians in the tuple are relatively spin with respect to one fixed set of background classes $b_k\in H^2(M_k,\Z_2)$, see \cite{orient}.
In order to be able to associate Floer cohomology groups with $\Z$-coefficients to the ends of $\ul{S}$, we moreover require that the cyclic generalized Lagrangian correspondences $\ul{L}_{\ul{e}}$ for each end $\ul{e}\in\E(\ul{S})$ in \eqref{end sequence c} are
either of even length or contain at least one symplectic manifold $M_k$ with $w_2(M_k)=0$.  (The ends of type \eqref{end sequence} satisfy this condition automatically since $w_2(\pt)=0$.) This condition ensures that the relative spin structure on $\cL$ induces a relative spin structure on $\ul{L}_{\ul{e}}$ for each end in the sense of \cite{quiltfloer}\footnote{
The published version of \cite{quiltfloer} requires a shift in background class at a specific point, but this can be relaxed to a single shift at any point, see the updated preprint versions.
}.

To define the relative invariant $\Phi_{\ul S}$ with $\Z$-coefficients we also need to choose 
orderings of the patches and of the boundary components of each $\ol{S}_k$ as in Definition~\ref{surfstrip}~(d). There is no ordering of ends of single patches but orderings
$\E_-(\ul{S})=(\ul{e}^-_1,\ldots,\ul{e}^-_{N_-(\ul{S})})$ and 
$\E_+(\ul{S})=(\ul{e}^+_1,\ldots,\ul{e}^+_{N_-(\ul{S})})$ of the quilted ends.
Moreover, we fix an ordering $\ul{e}=\bigl((k_1,e_1),\ldots (k_{n_{\ul{e}}},e_{n_{\ul{e}}})\bigr)$ of strip-like ends for each quilted end $\ul{e}$. For noncyclic ends, this ordering is determined by the order of seams as in (c). For cyclic ends, we need to fix a first strip-like end $(k_1,e_1)$ to fix this ordering.\footnote{
If we describe a quilted surface $\ul{S}$ as in Remark~\ref{other quilt def} by putting seams on a single surface $\ti{S}$, then orderings of boundary components and ends of $\ti{S}$ induce the necessary orderings for $\ul{S}$, with the exception that we need to specify an ordering of the patches.}

Given such orderings and a relative spin structure on $\cL$, we construct in \cite{orient} a coherent set of orientations $\eps$ on the zero and one-dimensional moduli spaces so that
the inclusion of the boundary in (c) has the signs $(-1)^{\sum_{\ul{f}<\ul{e}}
|\ul{x}_{\ul{f}}^-|}$ (for incoming trajectories) and $-(-1)^{\sum_{\ul{f}<\ul{e}}
|\ul{x}_{\ul{f}}^+|}$ (for outgoing trajectories), and thus $C\Phi_{\ul S}$   
As in Remark~\ref{rmk oriented S}, this implies that $C\Phi_{\ul S}$, defined by replacing $\#_2 \M_{\ul S}(X^-,X^+)_0$ with $\sum_{u \in \M_{\ul{S}(X^-,X^+)_0}} \eps(u)$, is a chain map with $\Z$-coefficients between the tensor products of the quilted Floer complexes.
It thus induces a map $\Phi_S: \bigotimes_{\ul{e}\in \E_-(\ul{S})} HF(\ul{L}_e)
 \to H_*\bigl( \bigotimes_{\ul{e}\in \E_+(\ul{S})} CF(\ul{L}_e) \bigr)$ that factors through $\bigotimes_{\ul{e}\in \E_+(\ul{S})} HF(\ul{L}_e)$ only up to some torsion terms. 
 
However, in most applications (e.g.\ \cite{ww:cat}) we use quilted surfaces with no outgoing ends or a single outgoing end $\cE_+=\{e_+\}$, and in these special cases we always obtain a well defined map $\Phi_S$ to $\Z$ resp.\ $HF(\ul{L}_{e_+})$. 
\end{remark}

\begin{remark} \label{degree shift}
We say that the tuples $\ul{M}$ and $\cL$ are {\em graded} if each $M_k$ is equipped with an $N$-fold Maslov covering for a fixed $N\in\N$ and each Lagrangian $L_{\sigma}\subset M_{k_\sigma} \times M_{k_\sigma'}$ and $L_b\subset M_{k_b}$ is graded with respect to the respective Maslov covering.  Moreover, we assume that the gradings are compatible with orientations in the sense of (G1-2).
Then the effect of the relative invariant $\Phi_{\ul{S}}$ on the grading is by a shift in degree of
\begin{equation} \label{quiltgrad}
|\Phi_{\ul{S}}| =  \sum_{k=1}^m \tfrac 12 \dim M_k \bigl( \#\cE_+(S_k) - \chi(\overline{S}_k) \bigr) \qquad
\text{mod} \ N .
\end{equation}
To check this we consider an isolated solution $\ul{u} \in \M_{\ul{S}}(X^-,X^+)_0$.
At each end $\ul{e}\in\cE(\ul{S})$ we can deform the linearized seam conditions
\beq \label{end seq}
(T_{(x_{k_1,e_1},x_{k_2,e_2})}L_{\sigma_1},T_{(x_{k_2,e_2},x_{k_3,e_3})}L_{\sigma_2}, \ldots, T_{(x_{k_n,e_n},x_{k_1,e_1})} L_{\sigma_n} )
\eeq
(and similar in the noncyclic case) to split type,
\beq \label{split seq}
(\Lambda'_{\ul{e},1}\times\Lambda_{\ul{e},2}, \Lambda_{\ul{e},2}'\times \Lambda_{\ul{e},3}, \ldots,
\Lambda_{\ul{e},n}'\times\Lambda_{\ul{e},1})
\eeq
for Lagrangian subspaces $\Lambda_{\ul{e},j},\Lambda_{\ul{e},j}'\subset T_{x_{k_j,e_j}}M_{k_j}$,
without changing the index.
This is possible since the space of Lagrangian subspaces transverse to a given one is always connected.
To be more precise, we can first homotope the Hamiltonian perturbations $H_{k_j,e_j}$ on the end to zero while moving the Lagrangians by a Hamiltonian isotopy. This deforms the linearized operator $D_{\ul{u}}$ through Fredholm operators to another direct sum of Cauchy-Riemann operators.
Next, homotoping the boundary conditions in \eqref{end seq} does not affect the Fredholm property (and hence preserves the index) as long as the product of correspondences remains transverse to the diagonal $\bigl(\Delta_{M_{k_1}}\times\Delta_{M_{k_2}}\times\ldots \Delta_{M_{k_n}}\bigr)^T$, where $(\cdot)^T$ shifts the first $M_{k_1}$ factor to the end to obtain a Lagrangian submanifold of
$M_{k_1}^-\times M_{k_2} \times M_{k_2}^-\ldots\times M_{k_n}^-\times M_{k_1}$.
Transversality of the split boundary conditions \eqref{split seq} then means that each pair
$\Lambda_{\ul{e},j}\pitchfork\Lambda_{\ul{e},j}'$ is transverse.

Now \cite[Lemma 3.1.5]{quiltfloer} expresses the degree of each end as a
sum $ |\ul{x}_{\ul{e}}| = \sum_{j=1}^{n_{\ul{e}}}
d(\tilde\Lambda_{\ul{e},j},\tilde\Lambda_{\ul{e},j}')$, and summing
this over all incoming resp.\ outgoing ends $\ul{e}$ of $\ul{S}$ is
the same as summing $d(\Lambda_{k,e}^0 ,\Lambda_{k,e}^1)$ over all
incoming resp.\ outgoing ends $e$ of all patches $S_k$.  Here the
transverse pair $\Lambda_{k,e}^0\pitchfork\Lambda_{k,e}^1\subset
T_{x_{k,e}}M_k$ of boundary conditions at $e\in\cE(S_k)$ corresponds
to a unique pair
$\Lambda_{\ul{e},j}\pitchfork\Lambda_{\ul{e},j}'\subset
T_{x_{k_j,e_j}} M_k$ for some quilted end $\ul{e}\in\cE(\ul{S})$ and
$j=1,\ldots,n_{\ul{e}}$ with $k_j=k$.

Next, we deform the linearized boundary conditions with fixed ends to split type over each seam. Again, this does not affect the index since it is a homotopy of Fredholm operators.
We thus obtain a splitting of the index $0= \sum_{k=1}^m \Ind(D_{E_k,\F_k})$ into the
indices of Cauchy-Riemann operators on $E_k=u_k^*TM_k$ with boundary
conditions in totally real subbundles $\F_k\subset E_k|_{\partial S_k}$.
From Remark \ref{rmk graded S} we have a mod $N$ degree index identity for each surface,
$$
\tfrac 12 \chi(\overline{S}_k) \dim M_k  = \Ind(D_{E_k,\F_k})
+ \sum_{e\in\cE_+(S_k)} \bigl( \tfrac 12 \dim M_k - d(\tilde\Lambda_{e}^0,\tilde\Lambda_{e}^1) \bigr)
+ \sum_{e\in\cE_-(S_k)}  d(\tilde\Lambda_{e}^0,\tilde\Lambda_{e}^1) .
$$
So, summing over all surfaces $S_k$ we obtain as claimed
$$
\sum_{k=1}^m\tfrac 12 \dim M_k \bigl(\chi(\overline{S}_k) - \#\cE_+(S_k) \bigr)  =
- \sum_{\ul{e}\in\cE_+(\ul{S})} |\ul{x}_{\ul{e}}^+|
+ \sum_{\ul{e}\in\cE_-(\ul{S})}  |\ul{x}_{\ul{e}}^-| \qquad\text{mod} \ N.
$$
\end{remark}

The gluing theorem \ref{thm Fglue1} generalizes to the quilted case as follows.

\begin{theorem}[Quilted Gluing Theorem] \label{thm:quiltglue}
Suppose that
$\ul{e}_\pm=(k^\pm_i,e^\pm_i)_{i=1,\ldots,N}\in \E_\pm({\ul{S}})$ are
ends with $n_{\ul{e}_-}=n_{\ul{e}_+}=N$ and such that the data
$M_{k^-_i,e^-_i}=M_{k^+_i,e^+_i}$ and
$\ul{L}_{\ul{e}_-}=\ul{L}_{\ul{e}_+}$ coincide.  Then we have
\begin{equation}\label{glue quilt}
 \Phi_{\#^{\ul{e}_-}_{\ul{e}_+}(\ul{S})} =
 \Tr_{\ul{e}_-,\ul{e}_+}(\Phi_{{\ul{S}}}),
\end{equation}
where $\#^{\ul{e}_-}_{\ul{e}_+}(\ul{S})$ is the quilted surface
obtained by gluing the ends in $\ul{e}_-$ to the corresponding ends in
$\ul{e}_+$. The algebraic trace $\Tr_{\ul{e}_-,\ul{e}_+}$ is defined
by the formula \eqref{trace} but using the quilted cup and cap (a
union of strips with Lagrangian boundary and seam conditions given by
$\ul{L}_{\ul{e}_\pm}$ as for quilted Floer theory, but with two outgoing resp.\ incoming ends as in Figure \ref{quiltcupcap}) to define $\Phi_\cup,\Phi_\cap$. 
\end{theorem}

\begin{figure}[ht]
\begin{picture}(0,0)
\includegraphics{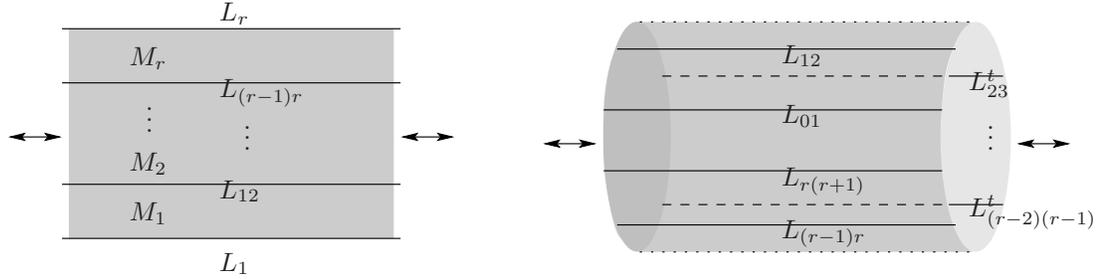}
\end{picture}
\setlength{\unitlength}{3729sp}
\begingroup\makeatletter\ifx\SetFigFont\undefined
\gdef\SetFigFont#1#2#3#4#5{
  \reset@font\fontsize{#1}{#2pt}
  \fontfamily{#3}\fontseries{#4}\fontshape{#5}
  \selectfont}
\fi\endgroup
\begin{picture}(6310,1890)(3139,-1231)
\put(4441,-116){\makebox(0,0)[lb]{$L_{(r-1) r}$}}
\put(4441,434){\makebox(0,0)[lb]{$L_{r}$}}
\put(4441,-1231){\makebox(0,0)[lb]{$L_{1}$}}
\put(4601,-401){\makebox(0,0)[lb]{$\vdots$}}
\put(4441,-751){\makebox(0,0)[lb]{$L_{12}$}}
\put(3841,134){\makebox(0,0)[lb]{$M_{r}$}}
\put(3951,-281){\makebox(0,0)[lb]{$\vdots$}}
\put(3841,-571){\makebox(0,0)[lb]{$M_2$}}
\put(3841,-901){\makebox(0,0)[lb]{$M_1$}}
\put(8176,150){\makebox(0,0)[lb]{$L_{12}$}}
\put(8176,-250){\makebox(0,0)[lb]{$L_{01}$}}
\put(8176,-701){\makebox(0,0)[lb]{$L_{r(r+1)}$}}
\put(8176,-1051){\makebox(0,0)[lb]{$L_{(r-1)r}$}}
\put(9541,-401){\makebox(0,0)[lb]{$\vdots$}}
\put(9421,-51){\makebox(0,0)[lb]{$L_{23}^t$}}
\put(9421,-931){\makebox(0,0)[lb]{$L_{(r-2)(r-1)}^t$}}
\end{picture}
\caption{The quilted cup and cap for a noncyclic and a cyclic sequence of Lagrangian correspondences}
\label{quiltcupcap}
\end{figure}

\section{Independence of quilt invariants}
\label{sec:invariance}

The purpose of this section is to prove the independence of the relative invariants arising from quilted surfaces.

\begin{theorem}\label{thm:inv}
Let $\ul{S}_0$ and $\ul{S}_1$ be two quilted surfaces of the same combinatorial type as in Lemma~\ref{homotopy of quilts}, and with fixed widths
$\ul{\delta}_{\ul{e}}$ at each quilted end.
(We can view this as two different choices of complex structures
$\ul{j}_i$ for $i=0,1$, and seam maps $\ul{\phi}_i$ for $i=0,1$ on the same quilted surface $\ul{S}$ with fixed strip-like ends.)
Fix one tuple of symplectic manifolds $\ul M$ and Lagrangian boundary and seam conditions $\cL$ as in Theorem~\ref{quilttraj} for both quilted surfaces, and let $(\ul{K}_i,\ul{J}_i)$ for $i=0,1$ be regular choices of perturbation data on each $\ul{S}_i$ with the same regular values on each strip-like end.
Then the chain maps $C\Phi_{\ul{S}_0}$ and $C\Phi_{\ul{S}_1}$ induced by these choices descend to the same map $\Phi_{\ul{S}_0}=\Phi_{\ul{S}_1}$ on Floer cohomology.
\end{theorem}
\begin{proof}
The key fact is that any two choices $(\ul{K}_i,\ul{J}_i,\ul{j}_i,\ul{\varphi}_i)$ for $i=0,1$, that are of fixed form over the strip-like ends, can be connected by a homotopy $(\ul{K}_\lambda,\ul{J}_\lambda,\ul{j}_\lambda,\ul{\varphi}_\lambda)_{\lambda\in[0,1]}$.
The homotopies of quilt data $\ul{j}_\lambda$ and $\ul{\varphi}_\lambda$ are provided by Lemma~\ref{homotopy of quilts}. For the Hamiltonians $\ul{K}_i$ we can use convex interpolation, and for the almost complex structures $\ul{J}_i$ we employ the fact that the space of compatible almost complex structures is contractible.
Given this homotopy, we can use the standard Floer homotopy argument:

Consider the universal moduli spaces consisting of pairs $(\lambda,\ul{u})$ of $\lambda\in[0,1]$ and a solution $\ul{u}$ with respect to the data $(\ul{K}_\lambda,\ul{J}_\lambda,\ul{j}_\lambda,\ul{\varphi}_\lambda)$.
For a given homotopy $\ul{J}_\lambda$ there exists a comeagre set of homotopies $\ul{K}_\lambda$, for which the universal moduli space is a smooth manifold.  The $0$-dimensional component can be oriented and counted to define a map 
$C\Psi:\otimes_{\ul{e}\in\cE_-(\ul{S})} CF(\ul{L}_{\ul{e}}) \to \otimes_{\ul{e}\in\cE_+(\ul{S})} CF(\ul{L}_{\ul{e}})$.  
The $1$-dimensional component has boundaries corresponding to the solutions contributing to $C\Phi_{\ul{S}_0}$ and$C\Phi_{\ul{S}_1}$ and ends corresponding to pairs of solutions contributing to $C\Psi$ and the boundary operators $\partial_\pm=\sum_{\ul{e}\in\cE_\pm} \partial_{\ul{e}}$ in the Floer complexes for the ends.  (Sphere and disk bubbling is excluded by monotonicity in $0$- and $1$-dimensional moduli spaces.)  Counting these with orientations proves 
$C\Phi_{\ul{S}_0} - C\Phi_{\ul{S}_1} = \partial_+ \circ C\Psi + C\Psi \circ \partial_-$, 
i.e.\ $C\Psi$ defines a chain homotopy between $C\Phi_{\ul{S}_0}$ and $C\Phi_{\ul{S}_1}$, and hence $\Phi_{\ul{S}_0}=\Phi_{\ul{S}_1}$ on cohomology.  See \cite[Chapter 5.2]{sch:coh} for the detailed construction.  

When working with $\Z$ coefficients, the orientations are given by the orientation of the determinant line bundles constructed in \cite{orient} plus a (first) $\R$-factor for the $[0,1]$-variable.  The gluing of orientations is
the same as in \cite{orient}, and the signs for $C\Phi_{\ul{S}_i}$ arise from the boundary orientation of $\partial[0,1]=\{0\}^-\cup\{1\}$.
\end{proof}

The maps $\Phi_{\ul{S}}$ are in fact {\em relative invariants}, depending only on the combinatorial structure of the quilted surface $\ul{S}$, in the following sense.

\begin{theorem}  \label{thm:inv2}
Let $\ul{S}_0$ and $\ul{S}_1$ be two quilted surfaces of the same combinatorial type as in Lemma~\ref{homotopy of quilts}. (Here we allow for different choices of strip-like ends, complex structures, and seam maps.)
Fix one tuple of symplectic manifolds $\ul M$ and Lagrangian boundary and seam conditions $\cL$ as in Theorem~\ref{quilttraj} for both, and let $(\ul{K}_i,\ul{J}_i)$ be regular choices of perturbation data on each $\ul{S}_i$. Then we have
$\Psi_+ \circ \Phi_{\ul{S}_0}  \circ \Psi_- = \Phi_{\ul{S}_1}$.
Here $\Psi_\pm=\otimes_{\ul{e}\in\cE_\pm} \Psi_{\ul{e}}$,
where $\Psi_{\ul{e}}:HF(\ul{L}_{\ul{e}})^0\to HF(\ul{L}_{\ul{e}})^1$ for $\ul{e}\in\cE^+$
resp.\ $\Psi_{\ul{e}}:HF(\ul{L}_{\ul{e}})^1\to HF(\ul{L}_{\ul{e}})^0$ for $\ul{e}\in\cE^-$
are the isomorphisms from \cite{quiltfloer} between the Floer cohomologies for the different widths and perturbation data on the end $\ul{e}$ induced by
$(\ul{K}_i,\ul{J}_i)$.
\end{theorem}
\begin{proof}
Each $\Psi_{\ul{e}}$ is the relative invariant given by a quilted cylinder $\ul{Z}_{01}^{\ul{e}}$ resp.\ $\ul{Z}_{10}^{\ul{e}}$ interpolating between the widths and perturbation data
$(\ul{\delta}_{\ul{e}},\ul{H}_{\ul{e}},\ul{J}_{\ul{e}})^{i}$ given from the quilted surfaces and perturbation data for $i=0,1$. 
The composition
$\Psi_+ \circ \Phi_{\ul{S}_0} \circ \Psi_- = \Phi_{\ul{S}_1}$ is the relative invariant for the glued surface
$\ul{S}_1' = \bigl(\sqcup_{\ul{e}\in\cE^+}\ul{Z}_{01}^{\ul{e}}\bigl) \# \ul{S}_0 \# \bigl(\sqcup_{\ul{e}\in\cE^-}\ul{Z}_{10}^{\ul{e}}\bigl)$
by Theorem~\ref{thm:quiltglue} applied to gluing one quilted cylinder at a time (with the order being immaterial by Theorem~\ref{thm:inv})\footnote{
When working with orientations, the gluing sign is $+1$ if we use the same order of patches and boundary components for $\ul{S}_0$, $\ul{S}_1$, and all intermediate quilted surfaces.
This follows from applying the oriented, quilted version of Corollary~\ref{cor Fglue2} in \cite{orient} to each gluing of a quilted cylinder, since each (compactified) patch of the cylinder has a single boundary component.
} 
Now the widths and perturbations induced on the ends are the same for the quilted surfaces $\ul{S}_1$ and $\ul{S}_1'$, hence Theorem~\ref{thm:inv} proves the identity $ \Phi_{\ul{S}_1}  = \Phi_{\ul{S}_1'}$.
\end{proof}

\section{Geometric composition and quilt invariants}
\label{shrinkingquilt}

Consider a quilted surface $\ul{S}$ containing a patch $S_{\ell_1}$ that
is diffeomorphic to $\R \times [0,1]$ and attached via seams
$\sigma_{01}=\{(\ell_0,I_0),(\ell_1,\R\times\{0\})\}$ and
$\sigma_{12}=\{(\ell_1,\R\times\{1\}),(\ell_2,I_2)\}$ to other surfaces
$S_{\ell_0}, S_{\ell_2}$. (The latter are necessarily different from $S_{\ell_1}$, but we might have $\ell_0=\ell_2$.)
We can allow one but not both of these seams to be replaced by a boundary component,
$(\ell_1,\R\times\{0\})\in\B$ or $(\ell_1,\R\times\{1\})\in\B$.
In that case we set $M_{\ell_0}=\{\pt\}$ resp.\ $M_{\ell_2}=\{\pt\}$.

Let $\L$ be Lagrangian boundary and seam conditions for $\ul{S}$ and suppose
that the Lagrangian correspondences $L_{\sigma_{01}} \subset
M_{\ell_0}^-\times M_{\ell_1}$, $L_{\sigma_{12}} \subset M_{\ell_1}^-
\times M_{\ell_2}$ associated to the boundary components of $S_{\ell_1}$ are
such that $L_{\sigma_{01}} \circ L_{\sigma_{12}}$ is smooth and embedded by projection into
$M_{\ell_0}^-\times M_{\ell_2}$.  Let $\ul{S}'$ denote the quilted
surface obtained by removing the patch $S_{\ell_1}$ and corresponding seams
and replacing it by a new seam $\sigma_{02}:=\{(\ell_0,I_0),(\ell_2,I_2)\}$ with seam map
$\varphi_{\sigma_{02}}:=\varphi_{\sigma_{12}}\circ \varphi_{\sigma_{01}} : I_0 \overset{\sim}{\to}I_2$.
We define Lagrangian boundary conditions $\L'$ for $\ul{S}'$ by
$L_{\sigma_{02}}:=L_{\sigma_{01}} \circ L_{\sigma_{12}}$ .  In this setting we
have a canonical identification of Floer chain groups attached to the
ends
\begin{equation} \label{identeq}
 CF(\ul{L}_{\ul{e}}) \overset{\sim}{\to}  CF(\ul{L}'_{\ul{e}})
\end{equation}
as in \cite{quiltfloer} for every $\ul{e}\in \E(\ul{S})\cong \E(\ul{S}')$.
Now consider the relative invariants $\Phi_{\ul{S}}$ and
$\Phi_{\ul{S}'}$ defined in Section~\ref{inv quilts}.

\begin{theorem} \label{intertwine}
Suppose that all symplectic manifolds in $\ul{M}$ satisfy (M1-2)
with the same monotonicity constant, all Lagrangians in $\cL$
satisfy (L1-3), and $\cL$ is monotone and relatively spin.
Assume moreover that $L_{\sigma_{01}} \circ L_{\sigma_{12}}$ is embedded (in the sense defined in the introduction).
Then \eqref{identeq} induces isomorphisms in Floer cohomology
$$ 
\Psi_{\ul{e}} : HF(\ul{L}_{\ul{e}}) \to HF(\ul{L}'_{\ul{e}}) 
$$
and furthermore these maps intertwine with the relative invariants:
$$
\Phi_{\ul{S}'} \circ \biggl( \bigotimes_{\ul{e}\in \E_-}
\Psi_{\ul{e}} \biggr) = \biggl( \bigotimes_{\ul{e} \in \E_+}
\Psi_{\ul{e}} \biggr)\circ \Phi_{\ul{S}} \, [n_{\ell_1} d] . 
$$
Here $[n_{\ell_1}d]$ denotes a degree shift with $2n_{\ell_1}=\dim M_{\ell_1}$ and $d=1,0,\,\text{or}\; -1$ according to whether the removed strip $S_{\ell_1}$
has two outgoing ends, one in- and one outgoing, or two incoming ends.
\end{theorem}

\begin{proof}[Sketch of Proof:]
We showed in \cite{isom}\footnote{
At present, the strip shrinking analysis in \cite{isom} requires constant almost complex structures near the seams of the adjacent patches. However, \cite{quiltconst} only proves that the regular choices of compatible almost complex structures form a comeagre subset of the general smooth compatible almost complex structures.
This gap is closed by the more general width independent elliptic estimates in \cite{bottman}.
} 
that \eqref{identeq} intertwines the Floer differentials, and hence induces an isomorphism of Floer homology groups, if the strip $\R\times[0,\delta]$ that maps to $M_{\ell_1}$ with seam conditions $L_{01},L_{12}$ has sufficiently small width $\delta>0$. 
So let us also equip $\ul{S}$ with complex structures in which $S_{\ell_1}$ is a strip $\R\times[0,\delta]$ of sufficiently small width $\delta>0$ with standard complex structure. (Note that biholomorphisms do not affect this width since they would have to extend -- via seam maps -- to the neighbouring patches.)
Then on the level of chain complexes the maps $ \Psi_{\ul{e}} $ are simply the identity, and it suffices to show that the maps $C\Phi_{\ul{S}}$ and $C\Phi_{\ul{S}'}[d_k n_k]$ are equal by identifying the moduli spaces of pseudoholomorphic quilts that contribute to them. (See \eqref{quiltgrad} and the example below for an explanation of the grading shift.)  This bijection between moduli spaces is obtained from essentially the same strip width degeneration as in \cite{isom} (which was also used in \cite{quiltfloer}).  However, in this case the adjoint of the linearized operator is honestly surjective.  (There is no translational symmetry, so elements of the zero-dimensional moduli space have linearized operators of index zero, not one.)  The surfaces to the left and right of the shrinking strip are arbitrary quilted surfaces, but this is of no relevance in the proof.
It suffices to work with the uniformizing neighbourhoods $\R\times(-\epsilon,0]\hookrightarrow S_{\ell_0}$ and $\R\times[0,\epsilon)\hookrightarrow S_{\ell_2}$ in which the seam maps are the identity on $\R$.
The only point that requires an extra argument is the role of the curvature terms $R_{K_{\ell_i}}$ in the exclusion of bubbling. For that purpose we work -- as in \cite{isom} -- with $K_{\ell_1}\equiv 0$. 
Then if the energy density $e^\nu_\ell =\frac 12 \bigl| \d u^\nu_{\ell} - Y_{S_\ell}(u^\nu) \bigr|^2$ for $\ell\in\{\ell_0,\ell_1,\ell_2\}$ concentrates for $\nu\to\infty$ with strip widths $\delta^\nu\to 0$ at a point $z^\infty$ on the new seam $I_0\cong I_2$, we apply the mean value inequality arguments of \cite{isom} to deduce convergence of a subsequence (still denoted by $\nu\in\N$) as follows:
We have $\cC^\infty_{\rm loc}$ convergence $u_\ell^\nu \to u_\ell^\infty$  on a punctured neighbourhood of $z^\infty$ together with convergence of the measure on a neighbourhood of $z^\infty$,
$$
e^\nu_\ell {\rm dvol}_{S_\ell} \; = \; {u^\nu_\ell}^*\omega_\ell + (R_{K_\ell}\circ u_\ell^\nu) {\rm dvol}_{S_\ell} \; \overset{\nu\to\infty}{\longrightarrow} \;
\hbar \, \delta_{z^\infty} + {u^\infty_\ell}^*\omega_\ell + (R_{K_\ell}\circ u^\infty_\ell) {\rm dvol}_{S_\ell}
$$
for some $\hbar>0$.
Here we have $L^1$-convergence $R_{K_\ell}\circ u_\ell^\nu \; \overset{\nu\to\infty}{\longrightarrow} \; R_{K_\ell}\circ u^\infty_\ell$ of the curvature terms due to the uniform bounds on $R_{K_\ell}$ from Remark~\ref{rmk:monotone quilt}.
We deduce that the energy concentration must also come with a concentration of symplectic area ${u^\nu_\ell}^*\omega_\ell \; \overset{\nu\to\infty}{\longrightarrow} \; \hbar \, \delta_{z^\infty} +  {u^\infty_\ell}^*\omega_\ell$.
With that, we can use the monotonicity assumptions and the resulting area-index relation \eqref{q-ar-in} to deduce a jump in Fredholm index ${\rm Ind}(D_{\ul{u}^\infty}) < {\rm Ind}(D_{\ul{u}^\nu}) = 0$, which excludes bubbling since the moduli spaces of negative Fredholm index can assumed to be regular and hence empty.
\end{proof}

Note that Theorem~\ref{intertwine} does not assume $L_{\sigma_{01}} \circ L_{\sigma_{12}}$ to satisfy (L1-3) or monotonicity of $\cL'$. (L2) holds automatically for the geometric composition, and monotonicity of $\L'$ follows directly from that of $\L$ since any quilted map in Definition~\ref{corr monotone} for $\L'$ can be lifted to one for $\L$ with the same area and index.
Similarly, all cyclic sequences $\ul{L}'_{\ul{e}}$ are automatically monotone; as a consequence (L1) holds for $L_{\sigma_{01}} \circ L_{\sigma_{12}}$, and as explained in \cite{quiltfloer} the Floer cohomologies are well defined, even if (L3) may not hold for $L_{\sigma_{01}} \circ L_{\sigma_{12}}$.

\begin{example}  
To see the necessity of the degree shift in a simple example, suppose that $\ul{S} = (S)$ is the disk with two incoming ends, and $\Phi_S$ the corresponding relative invariant described in \eqref{phicap}. Suppose that $L^0,L^1$ intersect in a single point $x$.  Then the theorem above applies, $\ul{S}'$ is empty, $\Phi_{\ul{S}'}$ is the trivial invariant, and $\Psi_{\ul{e}_-}$ maps $\bra{x} \otimes \bra{x} \mapsto 1$.  On the other hand, $\Phi_{\ul{S}}$ is the duality pairing, which has degree $-n$. 
\end{example}

\begin{remark}
When working with $\Z$-coefficients we can keep the above setup in
case all correspondences are equipped with spin structures.
If some of them only carry relative spin structures with respect to nontrivial
background classes, then we have to make the following adjustment: Let
$\ul{S}'=\ul{S}$ be the same quilted surfaces but, instead of
canceling the strip $S_{\ell_1}$, label its seams by the composed
correspondence and a diagonal. Including relative spin structures with
background classes $b_{\ell_i}\in H^2(M_{\ell_i};\Z^2)$ there are two
possibilities: We may set $M_{\ell_1}':=M_{\ell_0}$ with background
class $b_{\ell_0}+w_2(M_{\ell_0})$ and seam conditions
$L'_{\sigma_{01}}:=\Delta_{M_0}$ and
$L'_{\sigma_{12}}:=L_{\sigma_{01}} \circ L_{\sigma_{12}}$, which is
equipped with relative spin structure with background class
$(b_{\ell_0}+w_2(M_{\ell_0}),b_{\ell_2})$, see \cite{quiltfloer}\footnote{
Note that the published version of \cite{quiltfloer} does not specify background classes for composition; see the corrections to the preprint for clarification.
}.
Alternatively, we may use $M_{\ell_1}':=M_{\ell_2}$ with background
class $b_{\ell_2}+w_2(M_{\ell_2})$ and seam conditions
$L'_{\sigma_{12}}:=\Delta_{M_2}$ and
$L'_{\sigma_{01}}:=L_{\sigma_{01}} \circ L_{\sigma_{12}}$, equipped
with the alternative relative spin structure with background class
$(b_{\ell_0},b_{\ell_2}+w_2(M_{\ell_2}))$. Note that these
constructions ensure that the new Lagrangian labels $\cL'$ are again
relative spin in the sense of Remark~\ref{qorient}.
\end{remark}

\begin{remark} \label{gescheit}
In \cite{isom} a strip degeneration argument is used to establish a canonical isomorphism
\begin{equation} \label{main2eq}
HF(L_0, L_{01}, L_{12}, L_2) \longrightarrow HF(L_0, L_{02}, L_2 )
\end{equation}
for embedded composition of monotone Lagrangians.
The natural alternative approach to defining an isomorphism, or even
just a homomorphism is to try and interpolate the middle strip to zero width in
the relative invariant that is used in \cite{quiltfloer} to prove the independence of
$HF(L_0, L_{01}, L_{12}, L_2)$ from the choice of width $\delta_1>0$.
The resulting quilted surface might look as indicated on the
left in Figure \ref{alt hom}.
This approach was pioneered by Matthias Schwarz and might be realized as a ``jumping boundary condition'' as investigated in \cite{fl:cot}. (After the completion of our work, \cite{lm} studied the same approach in Morse-Bott terms.)
However, this analytic setup is at present restricted to a highly specialized class of Lagrangians and almost complex structures.
In essence, this approach would construct a canonical element of the Floer cohomology $HF(L_{02}, (L_{01},L_{12}))$ associated to the three seams coming together.
Our approach also induces such a canonical element, given by the identity $1_{L_{02}}\in HF(L_{02}, L_{02})$ and the strip shrinking isomorphism to $HF(L_{02}, (L_{01},L_{12}))$.
So we replace the picture of seams coming together
by one where the seams corresponding to $L_{01}$, $L_{12}$, $L_{02}$ run into an infinite cylindrical end, as on the right in
Figure~\ref{alt hom}, where the inner circle is meant to represent a cylindrical end.  This picture defines a relative invariant
$$ 
\Upsilon : HF(L_{02}, (L_{01},L_{12})) \otimes HF( L_0, L_{01}, L_{12}, L_2)
\to HF( L_0, L_{02}, L_2 ) .
$$
Now the isomorphism \eqref{main2eq} can then alternatively be described by $\Upsilon(T,\cdot)$, where $T\in HF(L_{02}, (L_{01},L_{12}))$ is the morphism corresponding
to the identity $1_{L_{02}}\in HF(L_{02}, L_{02})$, see Corollary
\ref{natuerlich}.

If dealing with relative spin structures for $L_{01}$ and $L_{12}$
with background classes $(b_0,b_1)$ and $(b_1,b_2)$, then, as
discussed in \cite{orient}, the composed correspondence
$L_{02}$ can be equipped with a relative spin structure with
background class either $(b_0,b_2+w_2(M_2))$ or
$(b_0+w_2(M_0),b_2)$. So, unless $w_2(M_0)=0$ or $w_2(M_2)=0$, we need
to add a seam labeled by a diagonal (with canonical relative spin
structure shifting the background class by $w_2$) parallel to the
$L_{02}$ seam in the quilted surface on the right in Figure~\ref{alt
  hom} and in the proof of Corollary~\ref{natuerlich} below.  This is
reflected in the fact that by definition $HF(L_{02},
(L_{01},L_{12}))=HF((\Delta_{M_0},L_{02}),
(L_{01},L_{12}))=HF((L_{02},\Delta_{M_2}), (L_{01},L_{12}))$.
\end{remark}

\begin{figure}[ht]
\begin{picture}(0,0)
\includegraphics{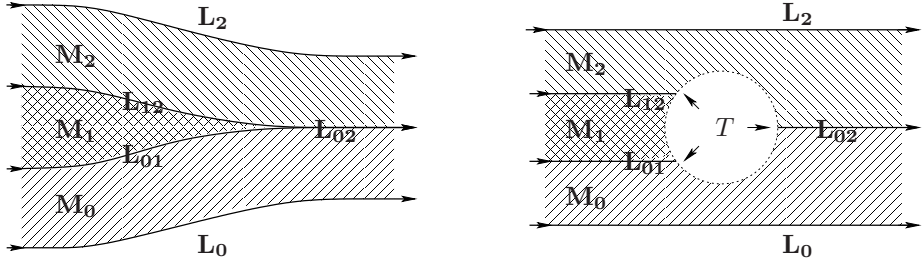}
\end{picture}
\setlength{\unitlength}{2072sp}
\begingroup\makeatletter\ifx\SetFigFont\undefined
\gdef\SetFigFont#1#2#3#4#5{
  \reset@font\fontsize{#1}{#2pt}
  \fontfamily{#3}\fontseries{#4}\fontshape{#5}
  \selectfont}
\fi\endgroup
\begin{picture}(11107,2987)(806,-2758)
\put(2900,-100){\makebox(0,0)[lb]{{$\mathbf{L_2}$}}}
\put(2000,-1110){\makebox(0,0)[lb]{{$\mathbf{L_{12}}$}}}
\put(2000,-1750){\makebox(0,0)[lb]{{$\mathbf{L_{01}}$}}}
\put(4300,-1450){\makebox(0,0)[lb]{{$\mathbf{L_{02}}$}}}
\put(2900,-2821){\makebox(0,0)[lb]{{$\mathbf{L_0}$}}}
\put(9900,-50){\makebox(0,0)[lb]{{$\mathbf{L_2}$}}}
\put(8000,-1070){\makebox(0,0)[lb]{{$\mathbf{L_{12}}$}}}
\put(8000,-1830){\makebox(0,0)[lb]{{$\mathbf{L_{01}}$}}}
\put(10300,-1450){\makebox(0,0)[lb]{{$\mathbf{L_{02}}$}}}
\put(9900,-2831){\makebox(0,0)[lb]{{$\mathbf{L_0}$}}}
\put(9100,-1380){\makebox(0,0)[lb]{{$T$}}}
\put(1200,-1450){\makebox(0,0)[lb]{{$\mathbf{M_1}$}}}
\put(1200,-550){\makebox(0,0)[lb]{{$\mathbf{M_2}$}}}
\put(1200,-2350){\makebox(0,0)[lb]{{$\mathbf{M_0}$}}}
\put(7300,-1450){\makebox(0,0)[lb]{{$\mathbf{M_1}$}}}
\put(7300,-650){\makebox(0,0)[lb]{{$\mathbf{M_2}$}}}
\put(7300,-2250){\makebox(0,0)[lb]{{$\mathbf{M_0}$}}}
\end{picture}
\caption{Alternative approaches to a homomorphism}
\label{alt hom}
\end{figure}

As a first application of Theorem \ref{intertwine} we prove the claim
of Remark \ref{gescheit}.  For that purpose we denote by
\begin{align*}
\Psi:HF(L_0,L_{01},L_{12},L_2)&\to HF(L_0,L_{02},L_2) , \\
\tilde\Psi:HF(L_{02},(L_{01},L_{12}))&\to HF(L_{02},L_{02})
\end{align*}
the isomorphisms \eqref{eq:iso} from \cite{quiltfloer}.  Then we have the
following alternative description of $\Psi$ (which a priori depends on
$L_0$ and $L_2$) in terms of $\tilde\Psi$ and the identity morphism
$1_{L_{02}}\in HF(L_{02},L_{02})$, defined in \cite{ww:cat}.

\begin{corollary} \label{natuerlich}
Let
$$
\Upsilon: HF(L_{02}, (L_{01},L_{12})) \otimes HF( L_0, L_{01}, L_{12}, L_2)
\to HF( L_0, L_{02}, L_2 )
$$
denote the relative invariant associated to the quilted surface
on the right in Figure~\ref{alt hom}.
Then we have for all $f\in HF( L_0, L_{01}, L_{12}, L_2)$
$$
\Psi(f) = \Upsilon(\tilde\Psi^{-1}(1_{L_{02}}) \otimes f) .
$$
Moreover, in the notation of \cite{ww:cat}, we have
$\Upsilon(T \otimes f) = \Phi_T(L_0) \circ f $,
where $\Phi_T:\Phi(L_{02})\to\Phi(L_{01},L_{12})$ is a natural
transformation of functors associated to any
$T\in HF(L_{02},(L_{01},L_{12}))$.
\end{corollary}

\begin{figure}[ht]
\begin{picture}(0,0)
\includegraphics{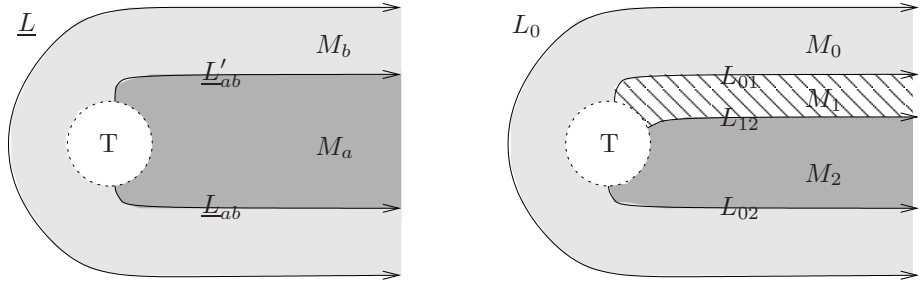}
\end{picture}
\setlength{\unitlength}{3356sp}
\begingroup\makeatletter\ifx\SetFigFont\undefined
\gdef\SetFigFont#1#2#3#4#5{
  \reset@font\fontsize{#1}{#2pt}
  \fontfamily{#3}\fontseries{#4}\fontshape{#5}
  \selectfont}
\fi\endgroup
\begin{picture}(6750,2063)(-9050,-553)
\put(-3284,1100){\makebox(0,0)[lb]{$M_0$}}
\put(-3284,700){\makebox(0,0)[lb]{$M_1$}}
\put(-3285,150){\makebox(0,0)[lb]{$M_2$}}
\put(-3914,-100){\makebox(0,0)[lb]{$L_{02}$}}
\put(-3914,870){\makebox(0,0)[lb]{$L_{01}$}}
\put(-3914,560){\makebox(0,0)[lb]{$L_{12}$}}
\put(-5450,1250){\makebox(0,0)[lb]{$L_0$}}
\put(-8500,400){\makebox(0,0)[lb]{T}}
\put(-4800,400){\makebox(0,0)[lb]{T}}
\put(-6900,350){\makebox(0,0)[lb]{$M_a$}}
\put(-6900,1100){\makebox(0,0)[lb]{$M_b$}}
\put(-9100,1250){\makebox(0,0)[lb]{$\ul{L}$}}
\put(-7739,880){\makebox(0,0)[lb]{$\ul{L}'_{ab}$}}
\put(-7739,-100){\makebox(0,0)[lb]{$\ul{L}_{ab}$}}
\end{picture}
\caption{Natural transformation associated to a Floer cohomology class:
General case from \cite{ww:cat} and the present special case.}
\label{natural}
\end{figure}

\begin{proof}
We apply Theorem \ref{intertwine} to $\Upsilon=\Phi_{\ul{S}}$, where
the quilted surface $\ul{S}$ contains one simple strip in $M_1$. (The
other surfaces are triangles.)  This implies $\Upsilon(T\otimes f) =
\Phi_{\ul{S}'}(\tilde\Psi(T)\otimes\Psi(f))$, where the quilted
surface $\ul{S}'$ is obtained by replacing this strip with a seam
condition in $L_{01}\circ L_{12}=L_{02}$.  To calculate this for
$\tilde\Psi(T)=1_{L_{02}}$ we use the gluing formula (\ref{glue
quilt}) to obtain $\Upsilon(\tilde\Psi^{-1}(1_{L_{02}})\otimes f) =
\Phi_{\ul{S}''}(\Psi(f))$, where $\ul{S}''$ is the surface that is
obtained by gluing the quilted cap of Figure \ref{quiltcupcap} into
$\ul{S}'$. Since $\ul{S}''$ is a simple double strip (with seam
condition $L_{02}$ and boundary conditions $L_0$, $L_2$), and we do
not quotient out by translation, the relative invariant
$\Phi_{\ul{S}''}$ is the identity, as in Example \ref{strip}. This
proves the first claim.

The second claim follows from a deformation of the quilt $\ul{S}$ to
the glued quilt that corresponds, by \eqref{glue quilt}, to the
composition of the natural transformation $T\mapsto\Phi_T(L_0)\in
HF((L_0,L_{02}),(L_0,L_{01},L_{12}))$ (given by the quilt in Figure
\ref{natural}) with the pair of pants product from
$HF((L_0,L_{02}),(L_0,L_{01},L_{12}))\otimes
HF((L_0,L_{01},L_{12}),L_2)$ to $ HF((L_0,L_{02}),L_2)$ defined in \cite{ww:cat}.
Figure \ref{quaek} gives a picture summary of these arguments.
\end{proof}

\begin{figure}[ht]
\begin{picture}(0,0)
\includegraphics{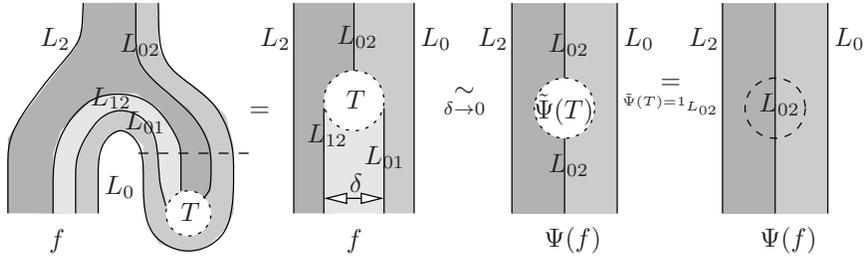}
\end{picture}
\setlength{\unitlength}{4144sp}
\begingroup\makeatletter\ifx\SetFigFont\undefined
\gdef\SetFigFont#1#2#3#4#5{
  \reset@font\fontsize{#1}{#2pt}
  \fontfamily{#3}\fontseries{#4}\fontshape{#5}
  \selectfont}
\fi\endgroup
\begin{picture}(5285,1509)(550,-1018)
\put(3061,-201){\makebox(0,0)[lb]{$\underset{\delta\to0}{\sim}$}
}
\put(4131,-201){\makebox(0,0)[lb]{$\underset{\scriptscriptstyle\tilde\Psi(T)=1_{L_{02}}}{=}$}
}
\put(1491,-826){\makebox(0,0)[lb]{$T$}
}
\put(650,199){\makebox(0,0)[lb]{$L_2$}
}
\put(1126,164){\makebox(0,0)[lb]{$L_{02}$}
}
\put(951,-156){\makebox(0,0)[lb]{$L_{12}$}
}
\put(1156,-301){\makebox(0,0)[lb]{$L_{01}$}
}
\put(1036,-691){\makebox(0,0)[lb]{$L_0$}
}
\put(701,-1016){\makebox(0,0)[lb]{$f$}
}
\put(1901,-201){\makebox(0,0)[lb]{$=$}
}
\put(2481,-1016){\makebox(0,0)[lb]{$f$}
}
\put(1966,199){\makebox(0,0)[lb]{$L_2$}
}
\put(2926,199){\makebox(0,0)[lb]{$L_0$}
}
\put(2426,209){\makebox(0,0)[lb]{$L_{02}$}
}
\put(2481,-151){\makebox(0,0)[lb]{$T$}
}
\put(2246,-376){\makebox(0,0)[lb]{$L_{12}$}
}
\put(2590,-511){\makebox(0,0)[lb]{$L_{01}$}
}
\put(2500,-651){\makebox(0,0)[lb]{$\delta$}
}
\put(4531,199){\makebox(0,0)[lb]{$L_2$}
}
\put(5401,199){\makebox(0,0)[lb]{$L_0$}
}
\put(4956,-1016){\makebox(0,0)[lb]{$\Psi(f)$}
}
\put(3270,199){\makebox(0,0)[lb]{$L_2$}
}
\put(3691,164){\makebox(0,0)[lb]{$L_{02}$}
}
\put(4141,199){\makebox(0,0)[lb]{$L_0$}
}
\put(3691,-556){\makebox(0,0)[lb]{$L_{02}$}
}
\put(3666,-1016){\makebox(0,0)[lb]{$\Psi(f)$}
}
\put(4951,-196){\makebox(0,0)[lb]{$L_{02}$}
}
\put(3600,-250){\makebox(0,0)[lb]{$\tilde\Psi(T)$}
}
\end{picture}
\caption{Proof by pretty picture}
\label{quaek}
\end{figure}

\section{Application: Morphism between quantum homologies}
\label{app}

Given a Lagrangian correspondence $L_{01}\subset M_0^-\times M_1$ a natural question is whether -- analogous to a symplectomorphism or certain birational equivalences as in \cite{cr:crep} -- it induces a ring morphism $\Phi(L_{01}): HF(\Delta_{M_0})\to HF(\Delta_{M_1})$ between the quantum homologies.
A quilted cylinder indeed defines a morphism, which however generally
does not intertwine the product structures.  Conjecturally, it is a
Floer theoretic equivalent of the slant product\footnote{ On de Rham
  cohomology, the product $\Om^k(M_0\times M_1)\times \Om^\ell(M_0)
  \to \Om^{k+\ell - \dim(M_0)}(M_1)$ is given by integration over one
  factor, $(\eta,\alpha) \mapsto \int_{M_0} ( \eta\wedge\alpha ) $.  }
on cohomology $H^*(M_0)\to H^*(M_1)$ with the Poincar\'e dual of the
fundamental class of $L_{01}$ (when the latter is compact and oriented).
Even in this classical case, one generally only obtains a ring homomorphism if $L_{01}$ is the graph of a continuous map.  Moreover, both the classical and Floer theoretic morphism will shift the grading unless $\dim M_0=\dim M_1$.
While Albers-Schwarz are working on refined versions of this morphism, we here use the framework of quilts to clarify the algebraic structure and behaviour under geometric
composition of Lagrangian correspondences.

\begin{remark}
The quilted Floer cohomology $HF(\Delta_M)$ of the diagonal by definition is nothing else but Hamiltonian periodic Floer cohomology $HF(M)=HF(H)$ for a generic Hamiltonian $H:S^1\times M\to \R$. The latter is naturally isomorphic to the quantum homology of $M$, see \cite{pss}.
\end{remark}

\begin{theorem} \label{phi l01}
Let $L_{01}\subset M_0^-\times M_1$ be a Lagrangian correspondence between symplectic manifolds $M_0, M_1$ satisfying (M1--2) and (L1--2), and assume that the pair $(L_{01},L_{01})$ is monotone for Floer theory.\footnote{The latter monotonicity requires an action-index relation for annuli in $M_0^-\times M_1$ with boundaries on $L_{01}$.
All these monotonicity assumptions can be replaced by any other set of assumptions which ensure that all Floer cohomologies involved are well defined and disk bubbles on $L_{01}$ cannot appear in $0$- or $1$-dimensional moduli spaces of holomorphic quilts.
}
Then the quilted cylinder in Figure~\ref{qhmorph} defines a natural map
$$
\Phi_{L_{01}}: HF(\Delta_{M_0})\to HF(\Delta_{M_1}) .
$$
\vspace{-5mm}
\begin{figure}[ht]
\begin{picture}(0,0)
\includegraphics{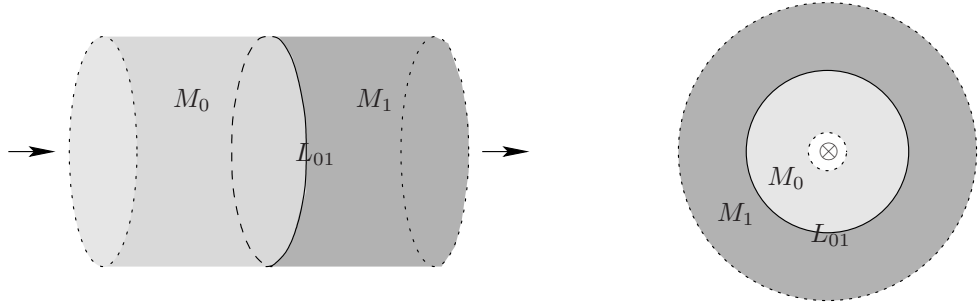}
\end{picture}
\setlength{\unitlength}{3729sp}
\begingroup\makeatletter\ifx\SetFigFont\undefined
\gdef\SetFigFont#1#2#3#4#5{
  \reset@font\fontsize{#1}{#2pt}
  \fontfamily{#3}\fontseries{#4}\fontshape{#5}
  \selectfont}
\fi\endgroup
\begin{picture}(5961,1998)(6490,-1330)
\put(11420,-581){\makebox(0,0)[lb]{$M_0$}}
\put(11086,-821){\makebox(0,0)[lb]{$M_1$}}
\put(8281,-421){\makebox(0,0)[lb]{$L_{01}$}}
\put(7471,-61){\makebox(0,0)[lb]{$M_0$}}
\put(8686,-61){\makebox(0,0)[lb]{$M_1$}}
\put(11706,-961){\makebox(0,0)[lb]{$L_{01}$}}
\put(11761,-385){\makebox(0,0)[lb]{$\otimes$}}
\end{picture}
\caption{Two views of the quilted cylinder defining $\Phi_{L_{01}}$.
It consists of two half cylinders (the domains of maps to $M_0$ and $M_1$ -- in this order to fix orientations), with a seam identifying the two circle boundaries (giving rise to a seam condition in $L_{01}$).
Arrows indicate incoming and outgoing ends, the outside circle is always an outgoing end, and $\otimes$ indicates incoming ends.}
\label{qhmorph}
\end{figure}

\noindent
Moreover, this map factors $\Phi_{L_{01}}=\Theta_{L_{01}}\circ\Psi_{L_{01}}$
according to Figure~\ref{qh factors} into maps
$$
\Psi_{L_{01}}: HF(\Delta_{M_0})\to HF(L_{01},L_{01}), \qquad
\Theta_{L_{01}}: HF(L_{01},L_{01}) \to  HF(\Delta_{M_1}).
$$
Here $\Psi_{L_{01}}$ is a ring morphism, whereas $\Theta_{L_{01}}$ satisfies
$$
\Theta_{L_{01}}(x \circ y) \;-\; \Theta_{L_{01}}(x) \circ \Theta_{L_{01}}(y) = x \circ T_{L_{01}}  \circ y
$$
for all $x,y\in HF(L_{01},L_{01})$ and a fixed element $T_{L_{01}} \in HF(L_{01}^t,L_{01} , L_{01}^t,L_{01} )$ that only depends on $L_{01}$.
For the latter we assume monotonicity of the sequence $(L_{01}^t,L_{01} , L_{01}^t,L_{01})$ and the composition\footnote{
Strictly speaking, the composition is defined by viewing $x$ and $y$ as elements in the trivially isomorphic quilted Floer cohomologies $x\in HF(\Delta_{M_1}, (L_{01}^t,L_{01}))$ and  $y\in HF((L_{01}^t,L_{01}),\Delta_{M_1})$ -- here in the notation for the pair of generalized Lagrangian correspondences $\Delta_{M_1}$ and $(L_{01}^t,L_{01})$ from $M_1$ to $M_1$.
Then $x \circ T_{L_{01}}  \circ y$ is the composition constructed in \cite{ww:cat} in the refined Donaldson-Fukaya category ${\rm Don}^\#(M_1,M_1)$ of correspondences from $M_1$ to $M_1$. In fact, it is the composition of morphisms
$\Delta_{M_1} \overset{x}{\longrightarrow} (L_{01}^t,L_{01})
\overset{T_{L_{01}}}{\longrightarrow} (L_{01}^t,L_{01}) \overset{y}{\longrightarrow} \Delta_{M_1}$,
taking values in $HF(\Delta_{M_1},\Delta_{M_1})$, which again is trivially identified with  $HF(\Delta_{M_1})$.
}
on the right hand side as well as $T_{L_{01}}$ are defined by relative quilt invariants as indicated in Figure~\ref{theta noring}.
\end{theorem}

\begin{figure}[ht]
\begin{picture}(0,0)
\includegraphics{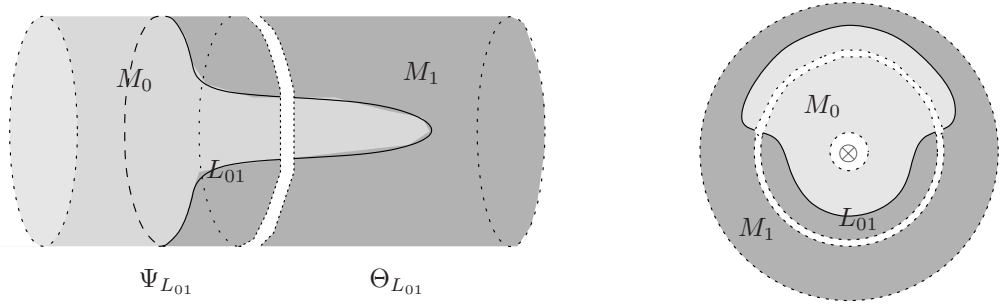}
\end{picture}
\setlength{\unitlength}{3729sp}
\begingroup\makeatletter\ifx\SetFigFont\undefined
\gdef\SetFigFont#1#2#3#4#5{
  \reset@font\fontsize{#1}{#2pt}
  \fontfamily{#3}\fontseries{#4}\fontshape{#5}
  \selectfont}
\fi\endgroup
\begin{picture}(6670,1998)(7270,-1330)
\put(9636,-1276){\makebox(0,0)[lb]{$\Theta_{L_{01}}$}}
\put(12751,-851){\makebox(0,0)[lb]{$L_{01}$}}
\put(12086,-916){\makebox(0,0)[lb]{$M_1$}}
\put(12521,-101){\makebox(0,0)[lb]{$M_0$}}
\put(7950, 74){\makebox(0,0)[lb]{$M_0$}}
\put(8551,-531){\makebox(0,0)[lb]{$L_{01}$}}
\put(9856,119){\makebox(0,0)[lb]{$M_1$}}
\put(8101,-1276){\makebox(0,0)[lb]{$\Psi_{L_{01}}$}}
\put(12750,-401){\makebox(0,0)[lb]{$\otimes$}}
\end{picture}
\caption{Two views of the splitting $\Phi_{L_{01}}=\Theta_{L_{01}}\circ\Psi_{L_{01}}$.
The quilted surface defining $\Psi_{L_{01}}$ consists of a disk labeled $M_0$ with one interior puncture (incoming end) and one boundary puncture (outgoing end), a disk labeled $M_1$ with one boundary puncture (outgoing end), and a seam labeled $L_{01}$ identifying the two boundary components.
The quilted surface defining $\Theta_{L_{01}}$ is the same with $M_0,M_1$ and incoming/outgoing ends interchanged.
}
\label{qh factors}
\end{figure}

\begin{remark}
As pointed out by the referee, the composition of pull-back $\Psi_{L_{01}}$ with push-forward $\Theta_{L_{01}}$ is not the one that classically preserves multiplicative structures. However, the classical projection formula $f_{!}( x \cup f^*(y) ) = f_{!}(x) \cup y$ relating multiplicative structures under pull-back $f^*$ and push-forward $f_{!}$ has a Floer theoretic version as follows.

We view $\Psi_{L_{01}^T}: HF(\Delta_{M_1}) \to HF(L_{01}^T,L_{01}^T)\cong HF(L_{01},L_{01})$ as push-back and $\Theta_{L_{01}}$ again as push-forward, then the Floer theoretic projection formula holds:
$$
\Theta_{L_{01}}\bigl( x \circ \Psi_{L_{01}^T}(y) \bigr) = \Theta_{L_{01}}(x) \circ y
\qquad
\forall
x\in HF(L_{01},L_{01}), y\in HF(\Delta_{M_1}) .
$$
The proof of this identity is a recommended exercise in applying the gluing and deformation laws for holomorphic quilts.
Hint: Read the proof of Theorem~\ref{phi l01} for very similar arguments.
\end{remark}

\begin{remark}
\begin{enumerate}
\item
One can check with Remark~\ref{degree shift} that the degree (modulo the minimal Maslov number $N_{L_{01}}$) of $\Phi_{L_{01}}$ agrees with the degree $d_{L_{01}}:=\frac 12 ( \dim(M_1) - \dim(M_0) )$ of the slant product. Indeed, the outgoing cylindrical end counts as $\#\cE_+(S_1)=2$. To see this, alternatively add a seam connecting the outgoing end to itself, labeled with the diagonal of $M_1$. Then we have two patches mapping to $M_1$ with one outgoing end each and total Euler characteristic $1$.
\item
By concatenating the morphism $\Phi_{L_{01}}: HF(\Delta_{M_0})\to HF(\Delta_{M_1})$ with the PSS-isomorphisms $H_*(M_0;\Z_2) \to HF(\Delta_{M_0})$
and $HF(\Delta_{M_1}) \to H_*(M_1;\Z_2)$ from \cite{pss} one can see why this should be the slant product: The concatenation is a map $H_*(M_0;\Z_2) \to H_*(M_1;\Z_2)$ on Morse homology. Between two critical points $x_0\in {\rm Crit}(f_0)\subset M_0$ and $x_1\in {\rm Crit}(f_1)\subset M_1$ the coefficient of the map is given by counting holomorphic disks in $M_0^-\times M_1$ with boundary on $L_{01}$ and a marked point in the interior mapping to the product $W^u(x_0, f_0)\times W^s(x_1,f_1)$ of unstable and stable manifolds.
Here the holomorphic curve equation is perturbed by Hamiltonian vector fields and uses almost complex structures parametrized by the disk. One can homotope the Hamiltonian perturbations to zero, not affecting the map on homology.

Next, recall that due to the grading ambiguity of $N_{L_{01}}$ on $L_{01}$, all gradings on Floer and Morse homology above should be taken modulo $N_{L_{01}}$. Hence, restricted to a fixed grading $k=0,\ldots,\dim M_0$ the morphism $\Phi_{L_{01}}$ induces a map
$$
\Phi_{L_{01}}^k :
H_k(M_0;\Z_2) \to \bigoplus_{\ell\in\N_0}  H_{k + d_{L_{01}} + \ell N_{L_{01}}} (M_1;\Z_2) .
$$ Here the contributions for each $\ell\in\N_0$ arise from
holomorphic disks as above, of Maslov index $\ell N_{L_{01}}$. For the
leading part of this map, the contributions for $\ell=0$ are those of
zero energy, that is exactly the constant maps to $L_{01}$.  This is
essentially the same calculation as in \cite[Remark
  1.3]{albers:erratum}.

With this one can check that the leading part of $\Phi_{L_{01}}^k$ is indeed the Morse theoretic version of the slant product.
If one could moreover achieve transversality with $S^1$-invariant almost complex structures, then (w.l.o.g.\ putting the marked point in the center of the disk) the only isolated solutions would be constant disks. This would identify $\Phi_{L_{01}}$ with a Morse theoretic version of the slant product.
\end{enumerate}
\end{remark}

\begin{proof}[Proof of Theorem~\ref{phi l01}]
Our assumptions guarantee that Theorem~\ref{quilttraj} applies to the moduli spaces of pseudoholomorphic maps from the quilted cylinder $\ul{S}$ indicated in Figure~\ref{qhmorph} to $M_0$ and $M_1$, with seam condition on $L_{01}$.
(In fact, doubling this quilted surface exactly amounts to an annulus mapping to $M_0^-\times M_1$ with both boundary conditions in $L_{01}$. Monotonicity for that surface is given by the monotonicity for Floer theory.)
Hence these moduli spaces define a map $\Phi_{L_{01}}$ on Floer cohomology as in \eqref{relinv}.
Strictly speaking, the cyclic generalized correspondences associated to the two ends of $\ul{S}$ are the empty sequences -- as trivial correspondence from $M_i$ to $M_i$. However, quilted Floer cohomology for these is trivially identified with quilted $HF(\Delta_{M_i})$, which is also trivially identified with the usual Floer cohomology for the pair $HF(\Delta_{M_i},\Delta_{M_i})$. All of these count pseudoholomorphic cylinders in $M_i$ with a Hamiltonian perturbation.

The map $\Phi_{L_{01}}: HF(\Delta_{M_0})\to HF(\Delta_{M_1})$ is natural in the sense that, by Theorems~\ref{thm:inv} and \ref{thm:inv2} it is independent of the choice of complex structures and seam maps on $\ul{S}$, almost complex structures, and Hamiltonian perturbations. The map is hence determined purely by $L_{01}$ and the chosen combinatorial structure of the quilted cylinder.

Next, the splitting $\Phi_{L_{01}}=\Theta_{L_{01}}\circ\Psi_{L_{01}}$ follows from Theorem~\ref{thm:quiltglue} applied to the gluing of quilted surfaces indicated in Figure~\ref{qh factors}.\footnote{
When working with orientations, the gluing sign is $+1$ since there are no further incoming ends, see the oriented, quilted version of Corollary~\ref{cor Fglue2} in \cite{orient}. Here we fix the order of patches for $\Theta_{L_{01}}$ as induced by $\Phi_{L_{01}}$. For
$\Psi_{L_{01}}$ the order of patches is irrelevant for orientations since each patch $S_k$ has even deficiency $\#\cE_+(S_k)-b(S_k)$, where $b(S_k)$ denotes the number of boundary components of $\overline{S}_k$.}
Monotonicity for these quilts as well as the pair of pants quilts below follows from the monotonicity for disks and spheres ensured by (M2) and (L2) since all Hamiltonian orbits resp.\ Hamiltonian chords generating the Floer cohomologies can be contracted.\footnote{
Strictly speaking, we either work with $C^1$-small Hamiltonian perturbations whose orbits resp.\ chords are automatically contractible, or we restrict the map to contractible orbits resp.\ chords. The isomorphism of Floer cohomologies for different Hamiltonian perturbations automatically maps contractible to contractible generator, and hence shows that the noncontractible generators do in fact not contribute to the cohomology.
}
To establish the interaction of these maps with the product structures recall that the product on $HF(\Delta_{M_i})$ is the relative invariant $\Phi_{S_{pop}}$ defined (see Section~\ref{rel inv}) by the pair of pants surface (a sphere with two incoming and one outgoing puncture). As usual, we draw this surface as a disk with two inner disks removed -- their boundaries are the incoming ends, while the outer circle is the outgoing end.
The product on $HF(L_{01},L_{01})$ is in the setting of Section~\ref{rel inv} given by the half pair of pants surface (a disk with two incoming and one outgoing puncture on the boundary). However, since the disk maps to $M_0^-\times M_1$, we can also view
it as two disks mapping to $M_0$ and $M_1$ respectively (where one carries the opposite orientation), with three boundary punctures each, all three boundary components identified and satisfying the seam condition in $L_{01}$.
Viewed as quilted surface, this is a pair of pants surface on which three seams connect each pair of ends, see Figure~\ref{quiltpop}.
\begin{figure}
\begin{picture}(0,0)
\includegraphics{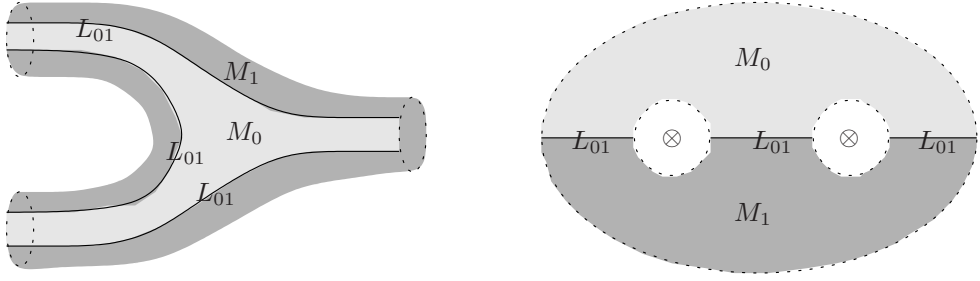}
\end{picture}
\setlength{\unitlength}{3729sp}
\begingroup\makeatletter\ifx\SetFigFont\undefined
\gdef\SetFigFont#1#2#3#4#5{
  \reset@font\fontsize{#1}{#2pt}
  \fontfamily{#3}\fontseries{#4}\fontshape{#5}
  \selectfont}
\fi\endgroup
\begin{picture}(6473,1818)(7850,-3082)
\put(12100,-2240){\makebox(0,0)[lb]{$\otimes$}}
\put(13800,-2277){\makebox(0,0)[lb]{$L_{01}$}}
\put(13277,-2240){\makebox(0,0)[lb]{$\otimes$}}
\put(11500,-2277){\makebox(0,0)[lb]{$L_{01}$}}
\put(12700,-2277){\makebox(0,0)[lb]{$L_{01}$}}
\put(12582,-1726){\makebox(0,0)[lb]{$M_0$}}
\put(12582,-2754){\makebox(0,0)[lb]{$M_1$}}
\put(9181,-1816){\makebox(0,0)[lb]{$M_1$}}
\put(9201,-2221){\makebox(0,0)[lb]{$M_0$}}
\put(9000,-2626){\makebox(0,0)[lb]{$L_{01}$}}
\put(8801,-2350){\makebox(0,0)[lb]{$L_{01}$}}
\put(8200,-1535){\makebox(0,0)[lb]{$L_{01}$}}
\end{picture}
\caption{Two views of the quilted pair of pants defining the product on $HF(L_{01},L_{01})$}
\label{quiltpop}
\end{figure}
If we view $HF(L_{01},L_{01})$ as quilted Floer cohomology for the sequence $(L_{01},L_{01}^t)$, then the relative invariant $\Phi_{\ul{S}_{pop}}$
associated to this quilted pair of pants is exactly the product. We can hence use this definition to calculate for all $f,g\in HF(M_0)$
$$
\Psi_{L_{01}}(f) \circ \Psi_{L_{01}}(g)
= \bigl(\Phi_{\ul{S}_{pop}}\circ(\Psi_{L_{01}}\otimes\Psi_{L_{01}})\bigr)(f,g)
= \Phi_{\ul{S}'}(f,g)
= \bigl(\Psi_{L_{01}}\circ\Phi_{S_{pop}}\bigr)(f,g)
= \Psi_{L_{01}}(f \circ g).
$$
Here the first and third equality is by definition. The middle identity
$\Phi_{\ul{S}_{pop}}\circ(\Psi_{L_{01}}\otimes\Psi_{L_{01}}) =
\Psi_{L_{01}}\circ\Phi_{S_{pop}}$
follows from gluing the quilted surfaces in different order (using Theorem~\ref{thm:quiltglue}) and homotoping (using Theorem~\ref{thm:inv}) between the two resulting quilted pair of pants surfaces that both have one seam connecting the outgoing end to itself and the two incoming ends on $M_0$, see Figure~\ref{psi ring}.\footnote{
When working with orientations, the gluing signs are $+1$, in the first case since the incoming patches have one boundary component, in the second case since the outgoing patch has only one incoming end, c.f.\ the oriented, quilted version of Corollary~\ref{cor Fglue2} in \cite{orient}.}
\begin{figure}
\begin{picture}(0,0)
\includegraphics{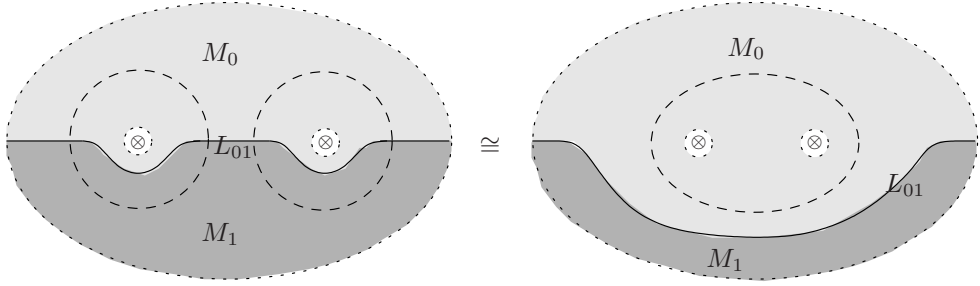}
\end{picture}
\setlength{\unitlength}{3812sp}
\begingroup\makeatletter\ifx\SetFigFont\undefined
\gdef\SetFigFont#1#2#3#4#5{
  \reset@font\fontsize{#1}{#2pt}
  \fontfamily{#3}\fontseries{#4}\fontshape{#5}
  \selectfont}
\fi\endgroup
\begin{picture}(6335,1820)(11390,-3089)
\put(12556,-1681){\makebox(0,0)[lb]{$M_0$}}
\put(15976,-1646){\makebox(0,0)[lb]{$M_0$}}
\put(12090,-2230){\makebox(0,0)[lb]{$\scriptstyle{\otimes}$}}
\put(13310,-2230){\makebox(0,0)[lb]{$\scriptstyle{\otimes}$}}
\put(15735,-2230){\makebox(0,0)[lb]{$\scriptstyle{\otimes}$}}
\put(16490,-2230){\makebox(0,0)[lb]{$\scriptstyle{\otimes}$}}
\put(15841,-3020){\makebox(0,0)[lb]{$M_1$}}
\put(12556,-2851){\makebox(0,0)[lb]{$M_1$}}
\put(12630,-2285){\makebox(0,0)[lb]{$L_{01}$}}
\put(17000,-2536){\makebox(0,0)[lb]{$L_{01}$}}
\put(14366,-2266){\makebox(0,0)[lb]{$\cong$}}
\end{picture}
\caption{Gluing and homotopy of quilts proving that $\Psi_{L_{01}}$ is a ring homomorphism}
\label{psi ring}
\end{figure}
This proves that $\Psi_{L_{01}}$ is indeed a ring morphism.

Similarly, we can write for all $x,y\in HF(L_{01},L_{01})$
\begin{align*}
\Theta_{L_{01}}(x) \circ \Theta_{L_{01}}(y)
- \Theta_{L_{01}}(x \circ y)
&= \bigl(\Phi_{S_{pop}}\circ(\Theta_{L_{01}}\otimes\Theta_{L_{01}})
- \Theta_{L_{01}}\circ\Phi_{\ul{S}_{pop}}\bigr)(x,y) \\
&= \bigl( \Phi_{\ul{S}_{3p}} - \Phi_{\ul{S}_{2p}}\bigr)(x,y) \\
&=  \bigl( \Phi_{\ul{S}_{3\text{comp}}} \circ
\bigl( {\rm Id} \otimes ( \Phi_{\ul{S}_{1}} - \Phi_{\ul{S}_{0}} ) \otimes {\rm Id} \bigr)(x,y) \\
&=  \Phi_{\ul{S}_{3\text{comp}}} (x, T_{L_{01}}, y) = x\circ T_{L_{01}}\circ y .
\end{align*}
Again, the first equality is by definition and the second identity follows from the gluing Theorem~\ref{thm:quiltglue}. In this case, however, we obtain two different quilted surfaces: The first composition results in the quilted pair of pants $\ul{S}_{3p}$ with two seams between three patches (and the $M_0$-patches ordered according to order of incoming ends), whereas the second composition results in the quilted pair of pants $\ul{S}_{2p}$ with two seams between two patches (with the $M_0$-patch ordered before the $M_1$-patch), see Figure~\ref{theta noring}.\footnote{
With orientations, the gluing signs are $+1$ since the incoming patches have one boundary component.
}
\begin{figure}
\begin{picture}(0,0)
\includegraphics{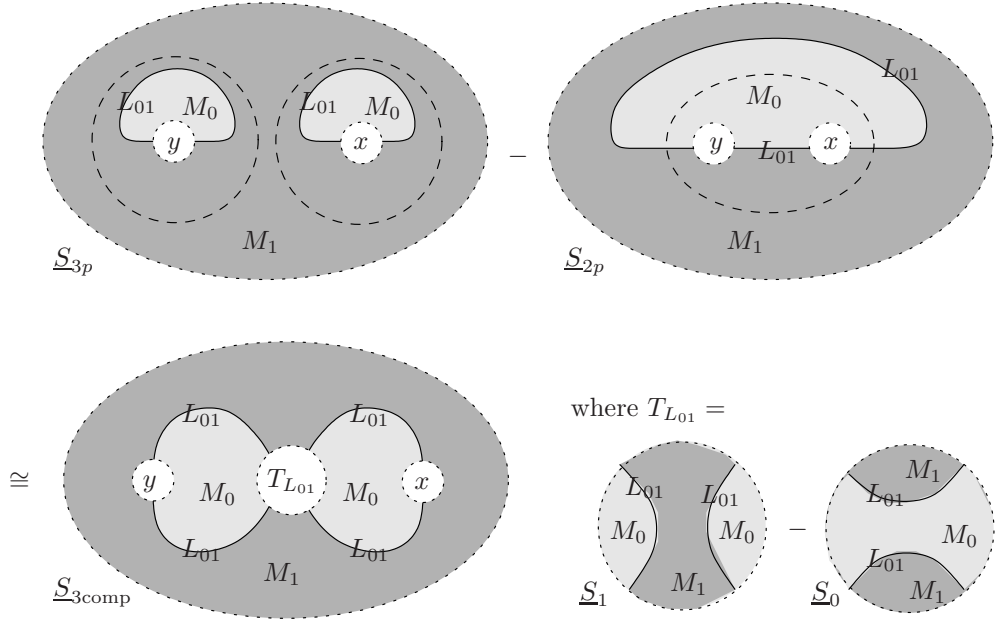}
\end{picture}
\setlength{\unitlength}{3729sp}
\begingroup\makeatletter\ifx\SetFigFont\undefined
\gdef\SetFigFont#1#2#3#4#5{
  \reset@font\fontsize{#1}{#2pt}
  \fontfamily{#3}\fontseries{#4}\fontshape{#5}
  \selectfont}
\fi\endgroup
\begin{picture}(6436,4017)(11300,-5286)
\put(12751,-2851){\makebox(0,0)[lb]{$M_1$}}
\put(11500,-3000){\makebox(0,0)[lb]{$\ul{S}_{3p}$}}
\put(12255,-2190){\makebox(0,0)[lb]{$y$}}
\put(13495,-2160){\makebox(0,0)[lb]{$x$}}
\put(13131,-1920){\makebox(0,0)[lb]{$L_{01}$}}
\put(12350,-1956){\makebox(0,0)[lb]{$M_0$}}
\put(13565,-1956){\makebox(0,0)[lb]{$M_0$}}
\put(11916,-1920){\makebox(0,0)[lb]{$L_{01}$}}
\put(15861,-2190){\makebox(0,0)[lb]{$y$}}
\put(16619,-2160){\makebox(0,0)[lb]{$x$}}
\put(15976,-2850){\makebox(0,0)[lb]{$M_1$}}
\put(16180,-2250){\makebox(0,0)[lb]{$L_{01}$}}
\put(17000,-1700){\makebox(0,0)[lb]{$L_{01}$}}
\put(16100,-1886){\makebox(0,0)[lb]{$M_0$}}
\put(14521,-2266){\makebox(0,0)[lb]{$-$}}
\put(14900,-3000){\makebox(0,0)[lb]{$\ul{S}_{2p}$}}
\put(16380,-4749){\makebox(0,0)[lb]{$-$}}
\put(16540,-5200){\makebox(0,0)[lb]{$\ul{S}_{0}$}}
\put(13900,-4420){\makebox(0,0)[lb]{$x$}}
\put(11206,-4426){\makebox(0,0)[lb]{$\cong$}}
\put(11500,-5200){\makebox(0,0)[lb]{$\ul{S}_{3\text{comp}}$}}
\put(12920,-4450){\makebox(0,0)[lb]{$T_{L_{01}}$}}
\put(14941,-3978){\makebox(0,0)[lb]{where $T_{L_{01}}=$}}
\put(15000,-5200){\makebox(0,0)[lb]{$\ul{S}_{1}$}}
\put(15300,-4483){\makebox(0,0)[lb]{$L_{01}$}}
\put(13421,-4526){\makebox(0,0)[lb]{$M_0$}}
\put(16899,-4530){\makebox(0,0)[lb]{$L_{01}$}}
\put(16900,-4962){\makebox(0,0)[lb]{$L_{01}$}}
\put(13466,-4000){\makebox(0,0)[lb]{$L_{01}$}}
\put(13466,-4900){\makebox(0,0)[lb]{$L_{01}$}}
\put(12095,-4440){\makebox(0,0)[lb]{$y$}}
\put(15199,-4779){\makebox(0,0)[lb]{$M_0$}}
\put(15800,-4536){\makebox(0,0)[lb]{$L_{01}$}}
\put(15920,-4779){\makebox(0,0)[lb]{$M_0$}}
\put(12351,-4000){\makebox(0,0)[lb]{$L_{01}$}}
\put(12351,-4900){\makebox(0,0)[lb]{$L_{01}$}}
\put(17186,-5175){\makebox(0,0)[lb]{$M_1$}}
\put(15600,-5121){\makebox(0,0)[lb]{$M_1$}}
\put(17400,-4802){\makebox(0,0)[lb]{$M_0$}}
\put(17160,-4374){\makebox(0,0)[lb]{$M_1$}}
\put(12896,-5056){\makebox(0,0)[lb]{$M_1$}}
\put(12466,-4526){\makebox(0,0)[lb]{$M_0$}}
\end{picture}
\caption{Manipulation of quilts proving the product identity for $\Theta_{L_{01}}$.
Each of the two quilts $\ul{S}_{3p}$ and $\ul{S}_{2p}$ in the first row can also be obtained by gluing one of the two quilts $\ul{S}_{1}$ and $\ul{S}_{0}$ in the definition of $T_{L_{01}}$ into the middle end of the quilt $\ul{S}_{3\text{comp}}$ in the second row.}
\label{theta noring}
\end{figure}
These two quilted surfaces are connected only by a 'singular homotopy' which lets the two seams intersect at a point and then connects them differently. Within our framework we can express this as follows: Both $\ul{S}_{3p}$ and $\ul{S}_{2p}$ are obtained from the same quilted surface $\ul{S}_{3\text{comp}}$ (consisting of one $M_1$ disk with three boundary and one interior puncture and two $M_0$ disks with two boundary punctures) by gluing different quilted disks into the middle end (which has four seams ending in it, hence each disk must have two seams connecting these ends). (The induced ordering on $\ul{S}_{3\text{comp}}$ has the two $M_0$-patches according to incoming ends before the $M_1$-patch.)
To obtain $\ul{S}_{3p}$ we glue in the quilted disk $\ul{S}_{1}$ which has one $M_1$ strip connecting different ends of the $M_1$ patch in $\ul{S}_{3\text{comp}}$ to produce an annulus, which is the $M_1$ patch in $\ul{S}_{3p}$.
To obtain $\ul{S}_{2p}$ we use the quilted disk $\ul{S}_{0}$ with one $M_0$ strip connecting the two $M_0$ patches in $\ul{S}_{3\text{comp}}$ to provide a single $M_0$ patch in $\ul{S}_{2p}$.\footnote{
When working with orientations,  
the ordering of patches on $\ul{S}_1$ and $\ul{S}_0$ is immaterial since only the middle patch in each has odd deficiency. The gluing signs are $+1$ for each patch since the incoming patches have one boundary component. Note however that the middle patch gets glued at two ends.
For $\ul{S}_0$ these are two composition gluings, joining different surfaces, but for $\ul{S}_1$ the second gluing joins a surface to itself -- a case that is discussed in~\cite{orient}.
}
The smooth part of these homotopies (using Theorem~\ref{thm:inv}) together with again the gluing Theorem~\ref{thm:quiltglue} prove the third identity.
Since both quilted disks $\ul{S}_{1}$ and $\ul{S}_{0}$ just have an outgoing end, they define an element of quilted Floer cohomology
$$
T_{L_{01}} := \Phi_{\ul{S}_{1}} - \Phi_{\ul{S}_{0}}  \in HF(L_{01}^t,L_{01} , L_{01}^t,L_{01} ).
$$
(Here we need a separate (e.g.\ monotonicity) assumption to ensure the Floer cohomology is well defined. Then the quilts $\ul{S}_0, \ul{S}_1$ are monotone since their double can be glued and homotoped to an annulus with boundaries on $L_{01}$. By a similar homotopy, or using the gluing identities, we deduce monotonicity for $\ul{S}_{3\text{comp}}$.)
With this definition, the fourth identity again is just by convention, and since we define the triple composition $x\circ T\circ y$ by the relative quilt invariant
$$
\Phi_{\ul{S}_{3\text{comp}}} :
HF(L_{01},L_{01})\otimes HF(L_{01}^t,L_{01} , L_{01}^t,L_{01}) \otimes HF(L_{01},L_{01}) \to HF(\Delta_{M_1}) ,
$$
this proves the product identity for $\Theta_{L_{01}}$.
Finally, the quilted Floer cohomology class $T_{L_{01}}$ only depends on $L_{01}$ and the (given) simple combinatorial type of the quilted surfaces $\ul{S}_{1}$ and $\ul{S}_{0}$, by Theorems~\ref{thm:inv} and \ref{thm:inv2}.
\end{proof}

\begin{remark}
Both Theorem~\ref{phi l01} and Theorem~\ref{quant} below hold with $\Z$-coefficients if the Lagrangian correspondences are equipped with relative spin structures with background classes $b_i\in H^2(M_i;\Z_2)$, and an induced relative spin structure on the composed correspondence, as detailed in Section~\ref{shrinkingquilt} and \cite{orient}.
Here any additional seams labeled by a diagonal can be canceled at the expense of shifting the background class on either $M_0$ or $M_2$ by $w_2$.
(Note here that the Hamiltonian Floer cohomology $HF(\Delta_M)$ is independent of choices of background class or relative spin structure for the diagonal, i.e.\ the quilted Floer cohomologies for different choices are canonically isomorphic.)
\end{remark}

Unlike the previous theorem, which directly generalizes to any version of Floer cohomology with Novikov rings, virtual fundamental classes, or abstract deformations,
the following ``quantization commutes with composition'' result requires the full strict monotonicity assumptions due to the potential figure eight bubbling described in \cite{isom}.

\begin{theorem} \label{quant}
Let $L_{01}\subset M_0^-\times M_1$ and $L_{12}\subset M_1^-\times M_2$ be Lagrangian correspondences satisfying (L1--2) between symplectic manifolds $M_0, M_1,M_2$ satisfying (M1--2) with a fixed monotonicity constant $\tau\geq 0$.
Assume that both pairs $(L_{01},L_{01})$ and $(L_{12},L_{12})$ are monotone for Floer theory and that the geometric composition $L_{01}\circ L_{12}$ is embedded.
Then
$$
\Phi_{L_{01}} \circ \Phi_{L_{12}} = \Phi_{L_{01}\circ L_{12}} .
$$
\end{theorem}

\begin{proof}
While the geometric composition $L_{01}\circ L_{12}$ of two monotone Lagrangian correspondences may not automatically satisfy (L2), the monotonicity for Floer theory transfers to the geometric composition (see \cite{quiltfloer}),
and then also implies (L2). Hence $L_{01}\circ L_{12}$ also satisfies the assumptions of Theorem~\ref{phi l01} and $\Phi_{L_{01}\circ L_{12}}$ is well defined.

Now, by the gluing Theorem~\ref{thm:quiltglue}\footnote{
When working with orientations, the gluing sign is $+1$ sign since no further incoming ends are involved.} 
we have
$\Phi_{L_{01}} \circ \Phi_{L_{12}} =  \Phi_{(L_{01},L_{12})}$, where the latter is
the relative quilt invariant associated to the quilted surface on the left in Figure~\ref{anal shrink}: two disks with interior puncture mapping to $M_0$ resp.\ $M_2$, one annulus $A$ without puncture mapping to $M_1$, and two seams connecting the boundary of each disk to one of the annulus boundary components, and with seam condition in $L_{01}$ and $L_{12}$ respectively.
In the definition of $\Phi_{(L_{01},L_{12})}$ we can pick any complex structure on (in particular) the annulus, e.g.\ $A=\R/\Z\times [0,\delta] \subset \C/\Z$ with the standard complex structure for any $\delta>0$. For the other patches we pick the same complex structures and perturbations as some regular choice in the definition of $\Phi_{L_{01}\circ L_{12}}$. We claim that this provides regular moduli spaces for $\Phi_{(L_{01},L_{12})}$ as well, and that the $0$-dimensional spaces of solutions are in fact bijective to those for $\Phi_{L_{01}\circ L_{12}}$. For the proof of this bold claim we refer to \cite{isom}, where the same analysis (just somewhat harder due to the noncompactness of the shrinking domain) is carried out in the case where $A$ is replaced by a strip $\R\times[0,\delta]$ between other strips of fixed width.
The crucial part of this proof is to exclude bubbling effects: As $\delta\to 0$, one may obtain sphere bubbles in $M_0$, $M_1$, or $M_2$ as well as disk bubbles in $(M_0^-\times M_1, L_{01})$, $(M_2^-\times M_2, L_{12})$, or $(M_0^-\times M_2, L_{01}\circ L_{12})$, and finally a novel type of bubble, called the ``figure eight bubble'', see \cite{isom, quiltfloer} for details. For the latter we are currently lacking a removable singularity theorem, so cannot exclude it by reasons such as a trivial homotopy class of figure eight maps or high codimension (in fact, we expect cases where these bubbles appear generically). Our solution is to use the full power of the monotonicity assumption together with an (indirectly proven) energy quantization for the figure eight bubble: If it appears then the remaining main component is a holomorphic quilt contributing to $\Phi_{L_{01}\circ L_{12}}$ which has less energy and (due to curvature bounds) also less symplectic area than the index $0$ solutions contributing to $\Phi_{(L_{01},L_{12})}$. However, the area for index $0$ solutions is the same for both relative quilt invariants, so by strict monotonicity the remaining solution would have negative index -- which is excluded by regularity of the moduli spaces.
This proves the equality $\Phi_{(L_{01},L_{12})}= \Phi_{L_{01}\circ L_{12}}$ and thus finishes the proof of the Theorem.
\begin{figure}
\begin{picture}(0,0)
\includegraphics{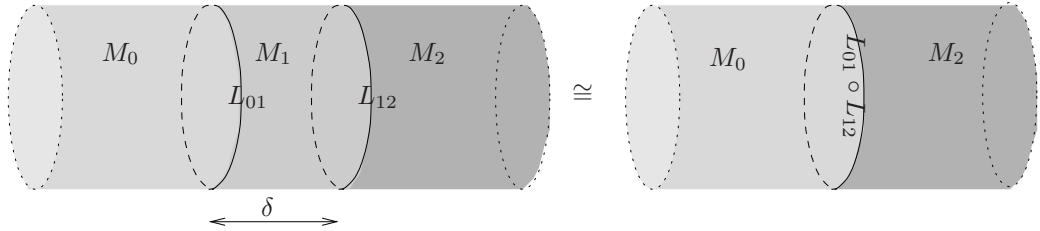}
\end{picture}
\setlength{\unitlength}{2983sp}
\begingroup\makeatletter\ifx\SetFigFont\undefined
\gdef\SetFigFont#1#2#3#4#5{
  \reset@font\fontsize{#1}{#2pt}
  \fontfamily{#3}\fontseries{#4}\fontshape{#5}
  \selectfont}
\fi\endgroup
\begin{picture}(8563,1871)(6600,-1408)
\put(9786,-61){\makebox(0,0)[lb]{$M_2$}}
\put(9361,-421){\makebox(0,0)[lb]{$L_{12}$}}
\put(8506,-61){\makebox(0,0)[lb]{$M_1$}}
\put(7236,-61){\makebox(0,0)[lb]{$M_0$}}
\put(8561,-1321){\makebox(0,0)[lb]{$\delta$}}
\put(11150,-376){\makebox(0,0)[lb]{$\cong$}}
\put(12286,-106){\makebox(0,0)[lb]{$M_0$}}
\put(14100,-61){\makebox(0,0)[lb]{$M_2$}}
\put(13400,200){\rotatebox{-90.0}{\makebox(0,0)[lb]{$L_{01}\circ L_{12}$}}}
\put(8281,-421){\makebox(0,0)[lb]{$L_{01}$}}
\end{picture}
\caption{Annulus shrinking proving the equivalence of algebraic and geometric composition}
\label{anal shrink}
\end{figure}
\end{proof}

\def\cprime{$'$} \def\cprime{$'$} \def\cprime{$'$} \def\cprime{$'$}
  \def\cprime{$'$} \def\cprime{$'$}
  \def\polhk#1{\setbox0=\hbox{#1}{\ooalign{\hidewidth
  \lower1.5ex\hbox{`}\hidewidth\crcr\unhbox0}}} \def\cprime{$'$}
  \def\cprime{$'$}


\begin{thebibliography}{10}

\bibitem{fl:cot}
A.~Abbondandolo and M.~Schwarz.
\newblock Floer cohomology of cotangent bundles and the loop product.
\newblock {\em Geom. Topol.},  14:1569â-1722, 2010.

\bibitem{alb:ext}
P.~Albers.
\newblock {On the extrinsic topology of Lagrangian submanifolds}.
\newblock {\em IMRN}, 38:2341--2371, 2005.

\bibitem{albers:erratum}
P.~Albers.
\newblock Erratum for ``{O}n the extrinsic topology of {L}agrangian
  submanifolds''.
\newblock {\em Int. Math. Res. Not. IMRN}, (7):1363--1369, 2010.

\bibitem{bottman}
N.~Bottman.
\newblock {\em Pseudoholomorphic quilts with figure eight singularity}.
\newblock \href{http://arxiv.org/abs/1410.3834}{http://arxiv.org/abs/1410.3834}.


\bibitem{cr:crep}
T.~Coates and Y.~Ruan.
\newblock Quantum cohomology and crepant resolutions: A conjecture, 2007.
\newblock \href{http://arxiv.org/abs/0710.5901}{http://arxiv.org/abs/0710.5901}.

\bibitem{kr:foam}
M.~Khovanov and L.~Rozansky.
\newblock Topological {L}andau-{G}inzburg models on the world-sheet foam.
\newblock {\em Adv. Theor. Math. Phys.}, 11(2):233--259, 2007.

\bibitem{lm}
Y.~Lekili and M.~Lipyanskiy. 
\newblock Geometric composition in quilted Floer theory. 
\newblock {\em Adv. Math.}, 236:1-23, 2013.

\bibitem{ms:jh}
D.~McDuff and D.A.~Salamon.
\newblock {\em {$J$}-holomorphic curves and symplectic topology}, volume~52 of
  {\em American Mathematical Society Colloquium Publications}.
\newblock American Mathematical Society, Providence, RI, 2004.

\bibitem{oh:fl1}
Y.-G. Oh.
\newblock {F}loer cohomology of {L}agrangian intersections and
  pseudo-holomorphic disks. {I}.
\newblock {\em Comm. Pure Appl. Math.}, 46(7):949--993, 1993.

\bibitem{per:lag}
T.~Perutz.
\newblock Lagrangian matching invariants for fibred four-manifolds. {I}.
\newblock {\em Geom. Topol.}, 11:759--828, 2007.

\bibitem{pss}
S.~Piunikhin, D.A.~Salamon, and M.~Schwarz.
\newblock Symplectic {F}loer-{D}onaldson theory and quantum cohomology.
\newblock In {\em Contact and symplectic geometry (Cambridge, 1994)}, volume~8
  of {\em Publ. Newton Inst.}, pages 171--200. Cambridge Univ. Press,
  Cambridge, 1996.

\bibitem{rs:spec}
J.~Robbin and D.A.~Salamon.
\newblock The spectral flow and the {M}aslov index.
\newblock {\em Bull. London Math. Soc.}, 27(1):1--33, 1995.

\bibitem{sa:lec}
D.A.~Salamon.
\newblock Lectures on {F}loer homology.
\newblock In {\em Symplectic geometry and topology ({P}ark {C}ity, {UT},
  1997)}, volume~7 of {\em IAS/Park City Math. Ser.}, pages 143--229. Amer.
  Math. Soc., Providence, RI, 1999.

\bibitem{sch:coh}
M.~Schwarz.
\newblock {\em Cohomology Operations from $S^1$-Cobordisms in {F}loer
  Homology,}.
\newblock PhD thesis, ETH Zurich, 1995.
\newblock
\href{http://www.math.uni-leipzig.de/~schwarz/}{www.math.uni-leipzig.de/$\sim$schwarz}.

\bibitem{se:gr}
P.~Seidel.
\newblock Graded {L}agrangian submanifolds.
\newblock {\em Bull. Soc. Math. France}, 128(1):103--149, 2000.

\bibitem{se:lo}
P.~Seidel.
\newblock A long exact sequence for symplectic {F}loer cohomology.
\newblock {\em Topology}, 42(5):1003--1063, 2003.

\bibitem{se:bo}
P.~Seidel.
\newblock {\em Fukaya categories and {P}icard-{L}efschetz theory}.
\newblock Zurich Lectures in Advanced Mathematics. European Mathematical
  Society (EMS), Z\"urich, 2008.

\bibitem{we:en}
K.~Wehrheim.
\newblock Energy quantization and mean value inequalities for nonlinear
  boundary value problems.
\newblock {\em J.Eur.Math.Soc.\ (JEMS)}, 7(3):305--318, 2005.

\bibitem{isom}
K.~Wehrheim and C.~Woodward.
\newblock {F}loer cohomology and geometric composition of {L}agrangian  correspondences.
\newblock {\em Advances in Mathematics} 230:1, 177--228, 2012.

\bibitem{ww:cat}
K.~Wehrheim and C.~Woodward.
\newblock Functoriality for {L}agrangian correspondences in {F}loer theory.
\newblock {\em Quantum Topology}, 1:129--170, 2010.

\bibitem{quiltfloer}
K.~Wehrheim and C.~Woodward.
\newblock Quilted {F}loer cohomology.
\newblock {\em Geom. Topol.} 14:833--902, 2010.

\bibitem{quiltconst}
K.~Wehrheim and C.~Woodward.
\newblock Quilted {F}loer trajectories with constant components.
\newblock Geometry \& Topology 16:1, 127--154, 2012.

\bibitem{orient}
K.~Wehrheim and C.~Woodward.
\newblock Orientations for pseudoholomorphic quilts.
\newblock 2009 preprint.

\end{thebibliography}
\end{document}